\documentclass[12pt]{amsart}
\usepackage[top=30truemm,bottom=30truemm,left=25truemm,right=25truemm]{geometry}
\usepackage{mathrsfs}
\usepackage{amsmath, amsthm, amssymb}
\usepackage{mathtools}
\usepackage{color}
\usepackage{bm}
\usepackage{amsfonts}
\usepackage{dsfont}
\usepackage{amscd}
\usepackage{extarrows}
\usepackage{mathrsfs}
\usepackage{enumerate}
\usepackage{amscd}   
\usepackage[all]{xy}
\usepackage{geometry}
\usepackage[colorlinks]{hyperref}
\usepackage[nameinlink, capitalize, noabbrev]{cleveref}
\usepackage[colorlinks]{hyperref}

\definecolor{darkblue}{RGB}{0, 102, 204}   

\hypersetup{
    pdfencoding=auto,
    colorlinks=true,    
    linkcolor=darkblue, 
    citecolor=darkblue, 
    urlcolor=darkblue,  
    filecolor=darkblue  
}
\newtheorem{theorem}{Theorem}[section]
\newtheorem{definition}{Definition}[section]
\newtheorem{lemma}{Lemma}[section]
\newtheorem{corollary}{Corollary}[section]
\newtheorem{proposition}{Proposition}[section]
\newtheorem{remark}{Remark}[section]
\newtheorem{example}{Example}[section]
\newtheorem{conjecture}{Conjecture}[section]

\newcommand{\be}{\begin{equation}}
	\newcommand{\ee}{\end{equation}}
\newcommand{\bea}{\begin{eqnarray}}
	\newcommand{\eea}{\end{eqnarray}}
\newcommand{\ben}{\begin{eqnarray*}}
	\newcommand{\een}{\end{eqnarray*}}
\newcommand{\bt}{\begin{split}}
	\newcommand{\et}{\end{split}}
\newcommand{\bet}{\begin{equation}}

	\numberwithin{equation}{section}
\begin{document}

\title[Gromov-Hausdorff limits of the Chern-Ricci flow]
{Gromov-Hausdorff limits of the Chern-Ricci flow on smooth Hermitian minimal models of general type}

\author[H. Sun]{Haoyuan Sun}
\address{Haoyuan Sun: School of Mathematical Sciences\\ Beijing Normal University\\ Beijing 100875\\ P. R. China}
\email{202531130037@mail.bnu.edu.cn}

\begin{abstract}
We establish uniform diameter estimates and volume non-collapsing estimates for the Chern-Ricci flow on smooth Hermitian minimal models of general type, assuming the initial metric is K\"ahler in a neighborhood of the null locus of the canonical bundle. This yields subsequential Gromov-Hausdorff convergence, partially resolving a conjecture of Tosatti and Weinkove. When the underlying manifold is K\"ahler, we further prove the uniqueness of the limit space. Analytically, we overcome the difficulties posed by non-K\"ahler torsion in the Green's formula by exploiting our local K\"ahler assumption, successfully adapting recent estimates of K\"ahler Green's function to the Hermitian setting. To prove the uniqueness of the limit, we introduce Perelman's reduced length to the Chern-Ricci flow. By establishing a uniform Chern scalar curvature bound and an almost monotonicity formula for the reduced volume, we deduce an almost-avoidance principle for the singular set, allowing us to effectively compare the flow distance with the canonical limit distance.
\end{abstract}

\maketitle
{
 \hypersetup{linkcolor=black} 
  \tableofcontents
}
\section{Introduction}
The K\"ahler-Ricci flow has proven to be a fundamental tool in the study of the geometry and topology of complex manifolds. The application of this flow traces back to the resolution of the Calabi conjecture by Yau \cite{Yau78}. A central theme in this field, particularly within the framework of the analytic Minimal Model Program initiated by Song and Tian in \cite{ST12, ST17}, is to understand the geometric behavior of the flow as it approaches finite time singularities or exists for infinite time. In this context, establishing uniform geometric bounds, such as diameter and scalar curvature bounds, is typically a key requirement for proving the convergence of the flow to a canonical metric limit. 

Among these geometric bounds, establishing uniform diameter estimates has been recognized as one of the most difficult and longstanding problems in complex geometry. In the K\"ahler setting, there is an extensive literature on this problem; see, for example, \cite{SW13, Song14, GSW16, TZ16, Guo17, Wang18, STZ19, GPSS23, GPSS24a, Vu26, GPSS24b, GT25, GGZ25} and references therein. Recently, Guo, Phong, Song, and Sturm \cite{GPSS24a, GPSS24b} developed a robust framework to establish uniform diameter estimates for a large class of K\"ahler metrics. A key aspect of their work is the reliance on an entropy bound, which avoids the pointwise Ricci lower bounds typically required in the classical Riemannian setting \cite{CL81}. Their approach, relying on novel estimates for Green’s functions and the Monge-Amp\`ere equations established in \cite{GPT23, GPTW24, GPS24}, successfully solved the long-standing problem of uniform diameter bounds for the K\"ahler-Ricci flow in a general setting. 

While the results in \cite{GPSS23, GPSS24a, GPSS24b} provide a comprehensive framework in the K\"ahler category, the situation for non-K\"ahler manifolds, which naturally arise in complex geometry and theoretical physics, remains far less understood. In this broader Hermitian setting, the Chern-Ricci flow, introduced by Tosatti and Weinkove in \cite{TW15}, serves as the natural evolution equation. The flow is given by 
\begin{equation}\label{main flow, introduction}
    \left\{
    \begin{aligned}
        &\frac{\partial}{\partial t}\omega(t)=-Ric^C
        (\omega(t)),\\
        &\omega(0)=\omega_0.
    \end{aligned}
    \right.
\end{equation}
where $\omega_0$ is an initial Hermitian metric and $Ric^C        
(\omega)$ is the Chern-Ricci form of $\omega$. The main focus of this paper is the following conjecture proposed by Tosatti and Weinkove in \cite{TW22}:
\begin{conjecture}\label{conjecture} \cite[Conjecture 4.1]{TW22}
    Let $X^n$ be a compact complex manifold with $K_X$ nef and big. Let $\omega(t)$ be the solution of the Chern-Ricci flow \eqref{main flow, introduction} starting at an arbitrary Hermitian metric $\omega_0$. Then we have
    $$
\operatorname{diam}\left(X,\frac{\omega(t)}{t}\right)\leq C,
    $$
    for all $t$ sufficiently large, and
    $$
\left(X,\frac{\omega(t)}{t}\right)\to (Z,d),\quad \operatorname{as}  t  \to\infty.
    $$
   in the Gromov-Hausdorff topology for some compact metric space $(Z,d)$. Furthermore, if we set $E:=\operatorname{Null}(K_X)$, then the compact metric space $(Z,d)$ can be identified with the metric completion of $(X\setminus E,\omega_{KE})$. Here, $\omega_{KE}$ is a closed positive current on $X$ which, moreover, restricts to a smooth K\"ahler-Einstein metric on $X\setminus E$.
\end{conjecture}

\cref{conjecture} was proved on K\"ahler surfaces independently by Guo-Song-Weinkove \cite{GSW16} and by Tian-Zhang \cite{TZ16}, the latter also proved the conjecture in the K\"ahler setting for dimension 3. Under the assumption that $\omega_0$ is K\"ahler and the flow admits a uniform Ricci lower bound, the conjecture was established by Guo \cite{Guo17}, and later completely solved by Wang \cite{Wang18} when $\omega_0$ is K\"ahler in arbitrary dimensions. In this paper, we generalize the above results by providing a proof of \cref{conjecture} under a specific local geometric assumption. Our first main result is:

\begin{theorem}\label{main theorem}
   In the setting of \cref{conjecture}, if $\omega_0$ is K\"ahler in an arbitrary neighborhood of $E=\operatorname{Null}(K_X)$, then there are uniform constants $c,C,\alpha,r_0$ such that
   \begin{enumerate}
       \item    For $t$ sufficiently large, we have
$$\operatorname{diam}\left(X,\frac{\omega(t)}{t}\right)\leq C$$
       \item   For any $x\in X$ and $r\in(0,r_0)$, we have
$$
\operatorname{Vol}_{\frac{\omega(t)}{t}}\left(B_{\frac{\omega(t)}{t}}(x,r)\right)\geq cr^{\alpha},
    $$
   where $B_{\omega}(x,r)$ is the geodesic ball in $(X,\omega)$ of radius $r$ centered at $x$.
    \item Given any sequence of times $t_i\to+\infty$, there exists a subsequence, still denoted by $\{t_i\}_i$, such that $\left(X,\frac{\omega(t_i)}{t_i}\right)\to (Z,d)$ in the Gromov-Hausdorff topology for some compact metric space $(Z,d)$.
   \end{enumerate}
\end{theorem}

\begin{example}
In the case $n=2$ (or in arbitrary dimension provided that $X$ is K\"ahler), one can construct numerous examples of Hermitian metrics that are K\"ahler near the analytic subvariety $E$ via conformal changes. Namely, given a K\"ahler metric $\omega_X$ on $X$ and a smooth function $f\in C^\infty(X)$ such that $f|_{U}\equiv1$, where $U$ is a neighborhood of $E=\operatorname{Null}(K_X)$, the resulting metric $\omega=e^f\omega_X$ will then satisfy our assumption.

    Furthermore, one can construct a global metric $\omega_0$ by gluing $\omega$ to an arbitrary Hermitian metric $\omega_1$ using a partition of unity: $\omega_0 = \eta\cdot \omega + (1-\eta)\omega_1$, where $\eta$ is a cutoff function supported in $U$. This $\omega_0$ is strictly K\"ahler in a neighborhood of $E$ and Hermitian elsewhere. 
\end{example}
\begin{example}\label{ex:nonkahler-example}
We can also give a non-K\"ahler example satisfying the assumption in \cref{main theorem}.
By \cite[Theorem C(i)]{PRS14}, there exists a hypersurface $Y$ of degree $d\geq6$ such that
\[
Y\subset \mathbb{P}^4
\]
with a unique ordinary double point \(p\) and no other singularities. Moreover, by
\cite[Theorem 4.1]{PRS14} (see also \cite[Example 2.8]{PRS14}), \(Y\) is factorial, hence
\(\mathbb{Q}\)-factorial.

Let
\[
f:X\to Y
\]
be one of the analytic small resolutions of the node \(p\). By the standard local model of a
threefold ordinary double point (cf. \cite[Example 3.22]{CHNP13}), \(X\) is smooth and the exceptional locus is a single rational
curve
\[
C:=\operatorname{Exc}(f)\cong \mathbb{P}^1,
\]
with normal bundle \(N_{C/X}\cong \mathcal{O}_{\mathbb{P}^1}(-1)\oplus
\mathcal{O}_{\mathbb{P}^1}(-1)\). Furthermore, a threefold
ordinary double point is a Gorenstein terminal singularity by \cite[Example 3.22]{CHNP13} again.

Since \(f\) is small, it is crepant, so
\[
K_X=f^*K_Y.
\]
On the other hand, by the adjunction formula for hypersurfaces in \(\mathbb{P}^4\),
\[
K_Y=(K_{\mathbb{P}^4}+Y)|_Y \cong \mathcal{O}_Y(d-5),
\]
which is ample because $d\geq6$. Hence \(K_X=f^*\mathcal{O}_Y(d-5)\) is nef and big and $X$ is Moishezon. Moreover, \(X\) is not K\"ahler. Indeed, since \(Y\) is
\(\mathbb{Q}\)-factorial, it admits no nontrivial projective small birational morphism by
\cite[Remark 3.13]{CHNP13}. Since a Moishezon manifold is K\"ahler if and only if it is projective, it follows that \(X\) is indeed non-K\"ahler, and therefore \(X\) is a
smooth Hermitian minimal model of general type.

We claim that
\[
\operatorname{Null}(K_X)=C.
\]
Indeed, since \(f(C)=\{p\}\), we have
\[
K_X\cdot C=f^*K_Y\cdot C=0.
\]
Now let \(V\subset X\) be any irreducible positive-dimensional analytic subvariety with
\(V\neq C\). Since \(f\) is small, it has no exceptional divisors, and thus \(f(V)\) is birational to \(V\). Therefore,
\[
\int_V c_1(K_X)^{\dim V}
=
\int_{f(V)} c_1(K_Y)^{\dim V}>0,
\]
because \(K_Y\) is ample. This proves the claim.

Finally, by the explicit local analytic description of the small resolution of an ordinary
double point in \cite[Example 3.22]{CHNP13}, a neighborhood of \(C\) in \(X\) is realized as
an open subset of a smooth submanifold (a smooth complete intersection) of \(\mathbb{C}^4\times\mathbb{P}^1\). Since
\(\mathbb{C}^4\times\mathbb{P}^1\) is K\"ahler, it follows that \(C=\operatorname{Null}(K_X)\)
admits a K\"ahler neighborhood. Hence \(X\) provides a non-K\"ahler example satisfying the
assumption of \cref{main theorem}.
\end{example}
Following \cite{GPSS24a}, our proof of \cref{main theorem} relies primarily on the Green’s function estimates and $L^\infty$-estimates for complex Monge-Amp\`ere equations. For the $L^\infty$-estimates, we adopt the recent results established in \cite{PSWZ25}, replacing the standard estimates from \cite{GPTW24}, which sharply generalized \cite{Yau78}, \cite{Koł98}, and \cite{EGZ09} in the K\"ahler setting. The central analytic difficulty lies in establishing the $L^{1+\varepsilon}$-estimates for the Green’s functions. Unlike the K\"ahler case, the Green’s formula in our Hermitian setting contains non-trivial torsion terms. To close the estimates, one must control the coupling between the torsion form and the gradients of geometric quantities, specifically the gradients of solutions to the Laplacian along the Chern-Ricci flow and the gradients of the Monge-Amp\`ere potentials. 

This is where our local K\"ahler hypothesis plays a pivotal role. Since we assume that the initial metric $\omega_0$ is K\"ahler in a neighborhood $U$ of the null locus $E = \operatorname{Null}(K_X)$, the torsion forms vanish identically within $U$ along the flow. Consequently, the gradient-torsion coupling terms only need to be controlled strictly outside of $U$. In this region, the flow is known to converge smoothly by \cite{Gill13}, and the geometry is non-degenerate. To fully exploit this, we introduce cutoff functions supported exclusively in this non-degenerate region to perform a localized integration-by-parts argument. This step is crucial, as it allows us to bound the localized $L^2$-gradient norms of the auxiliary solutions strictly by their $L^1$-norms. By substituting these Sobolev-type controls back into the Green's representation formula and applying Fubini's theorem to swap the order of integration, we effectively decouple the gradient terms. This reduces the problem to the $L^1$-integrability of the Green's function, enabling us to absorb the torsion error terms and successfully close the $L^{1+\varepsilon}$-estimates.

Consequently, our strategy centers on establishing a priori estimates for the Green's function $G(x,y)$ specifically when the pole $x$ lies within the K\"ahler neighborhood $U$ of $E$. A key reduction in our analysis is that it suffices to establish uniform diameter and volume non-collapsing estimates restricted solely to this neighborhood $U$. Since the exterior region is strictly non-degenerate, securing these local bounds near the null locus naturally yields global geometric control, allowing us to complete the proof of the conjecture.

Our second main result is to establish the uniqueness of the sequential Gromov-Hausdorff limits in \cref{main theorem}:
\begin{theorem}\label{main thm:uniqueness}
    In the setting of \cref{conjecture}, assume that the underlying manifold $X$ is K\"ahler (this is automatically true when $n=2$) and that $\omega_0$ is K\"ahler in a neighborhood of $E$, then $\left(X,\frac{\omega(t)}{t}\right)$ converges in the Gromov-Hausdorff topology to $\overline{(X\setminus E,\omega_{KE})}$, which is homeomorphic to the unique canonical model $X_{can}$ of $X$. 
\end{theorem}
\begin{remark}
The homeomorphism between $\overline{(X\setminus E,\omega_{KE})}$ and $X_{can}$ was already established by Song in \cite{Song14}.
\end{remark}
It is worth emphasizing that the assumption of $X$ being K\"ahler is strictly necessary here to circumvent fundamental algebraic and analytic obstructions. For instance, constructing a holomorphic birational morphism from $X$ to $X_{can}$ via holomorphic sections relies heavily on Kawamata's basepoint-free theorem and Iitaka's theorem, as described in \cite{Wang18} and \cite{LTZ26}. Moreover, the subsequent arguments intrinsically depend on the existence of partial $C^0$-estimates. These foundational results remain largely open in the general Hermitian setting.

As for the full Gromov-Hausdorff convergence, \cref{main thm:uniqueness} was completely proved in the K\"ahler setting by Wang \cite{Wang18}. Very recently, two alternative proofs have emerged: Jian-Song \cite{JS26} provided a streamlined proof utilizing Bamler's recent compactness theory for Ricci flows, and Lee-Tosatti-Zhang \cite{LTZ26} developed a completely different approach based on Perelman's $\mathcal{L}$-length that applies to arbitrary Kodaira dimensions. Let us briefly outline these three approaches.

In \cite{Wang18}, the author applied the methods of Chen-Wang \cite{CW17, CW19, CW20} to establish a Cheeger-Gromov type convergence of the KRF to a compact metric space $\mathcal{R}\cup\mathcal{S}$ with $\mathcal{R}$ an open convex K\"ahler-Einstein manifold. The weak convexity essentially follows from the observation that reduced geodesics degenerate to standard geodesics in the Einstein limit, while the strong convexity follows by invoking the arguments of \cite[Proposition 2.52]{CW17} (which is itself based on \cite{CN12}) to establish the global Harnack inequality for the volume radius. Once this limit space is established, techniques extending from \cite{DS14, Guo17, Song14} can be adapted to show that $\mathcal{R}$ is isometric to the regular part of the canonical model $X_{can}$. Jian and Song \cite{JS26} approach the convergence by leveraging Bamler's $\mathbb{F}$-convergence \cite{Bam20a, Bam20b, Bam23} and recent complex compactness results to give an alternative proof of the regularity and convexity of the limit space.

Conversely, Lee-Tosatti-Zhang \cite{LTZ26} circumvent the need to extract an intermediate Cheeger-Gromov limit space. Instead, they directly utilize Perelman's $\mathcal{L}$-length and reduced volume to compare the flow distance $d_T$ with the limit distance $d_{can}$ on the canonical model. The core of their argument is an almost-avoidance'' principle, which demonstrates that most minimizing $\mathcal{L}$-geodesics do not spend much time in the highly degenerate region near the singularities.

In this paper, we will mainly follow the approach in \cite{LTZ26}. While we believe the methods in \cite{Wang18} could also be adapted to our case, the approach of \cite{LTZ26} proves to be particularly direct in the Hermitian setting. Specifically, it allows us to bypass the highly technical construction of a robust intermediate limit space in the presence of non-K\"ahler torsion. The primary analytic difficulty we face here is that in the Hermitian case, Perelman's Laplacian estimates of the reduced length $L(q,\bar{\tau})$ and the exact monotonicity of the reduced volume $\tilde{V}(\tau)$ both fail. However, in our locally K\"ahler case, the torsion terms given by the difference of the Chern connection and the Levi-Civita connection are uniformly bounded, so one expects to establish almost versions of Perelman's analogous results. Namely, for a space-time curve $\gamma:[\tau_1,\tau_2]\to X$, we define its $\mathcal{L}$-length as
\begin{equation}
    \mathcal{L}(\gamma):=\int_{\tau_1}^{\tau_2}\sqrt{\tau}\left(S(\gamma(\tau))+|\dot{\gamma}(\tau)|^2_{g(\tau)}\right)d\tau,
\end{equation}
where we use the Chern scalar curvature $S$ under the Chern-Ricci flow instead of the Riemannian scalar curvature $R$ in \cite{Per02}. We then check all kinds of variational properties of the $\mathcal{L}$-length, including the first and second variations, the gradient and Laplacian estimates, and derive the corresponding $\mathcal{L}$-geodesic equations and $\mathcal{L}$-Jacobi fields in the Hermitian setting. Among which the uniform Chern scalar curvature bound plays a key role:
\begin{theorem}\label{introduction:scalar curvature bound}
    Let $X$ be a compact K\"ahler minimal model of general type and let $E$ be the null locus of $K_X$. Let $\omega_0$ be a Hermitian metric on $X$, which is moreover K\"ahler in a small neighborhood of $E$. Then, the Chern scalar curvature $S_C(t)$ of $\frac{\omega(t)}{t}$ under the Chern-Ricci flow \eqref{main flow, introduction} starting from $\omega_0$ is uniformly bounded.
\end{theorem}
When $\omega_0$ is K\"ahler, \cref{introduction:scalar curvature bound} was proved by Zhang \cite{Zhang09}. By proving a parabolic Schwarz lemma under our locally K\"ahler assumption, we demonstrate that the calculations in \cite{Zhang09} can be successfully adapted to our setting.
\begin{remark}
  \cref{introduction:scalar curvature bound} may hold When $X$ is non-K\"ahler as well; however, the current version suffices for our purposes.
\end{remark}
With the uniform scalar curvature bound and the variational properties at hand, we can rigorously define Perelman's reduced distance and reduced volume in our setting. For a fixed basepoint $p \in X$, the reduced distance $L(q, \bar{\tau})$ is defined as the infimum of the $\mathcal{L}$-length among all piecewise smooth curves connecting $p$ to $q$ in time $\bar{\tau}$ (see Section \ref{section Per}). We then denote the reduced length by $l(q,\bar{\tau}) := \frac{1}{2\sqrt{\bar{\tau}}}L(q,\bar{\tau})$ and set $\bar{L}(q,\bar{\tau}) := 2\sqrt{\bar{\tau}}L(q,\bar{\tau})$. Furthermore, by solving the corresponding $\mathcal{L}$-geodesic equation, we naturally define the $\mathcal{L}$-exponential map $\mathcal{L}\exp_\tau: T_p X \to X$ (see \cref{def:L exp map}). This allows us to parameterize the manifold using initial tangent vectors and formulate the reduced volume via its Jacobian $\mathcal{J}(v,\tau)$.

To effectively compare the flow distance with the canonical limit distance, a key analytical step is to establish an evolution estimate for the modified reduced length $\bar{L}$. Unlike the standard Ricci flow, where the corresponding Laplacian bound is strictly $4n$ (cf. \cite[section 7]{Per02}), the presence of non-K\"ahler torsion in our Hermitian setting weakens this to a linear growth bound. Specifically, we show that there is a uniform constant $C$ such that
$$
\left(\frac{\partial}{\partial\bar{\tau}}+\Delta_{g(\bar{\tau})}\right)\bar{L}(\cdot,\bar{\tau})\leq C(\bar{L}(\cdot,\bar{\tau})+1),
$$
where $\Delta_{g(\tau)}$ is the Riemannian Laplacian (see \cref{lem:evolution of L bar} below). Nevertheless, this estimate proves completely adequate for our analysis. Since our subsequent arguments evaluate this quantity strictly away from the degenerate locus $E$, we are able to establish an \textit{a priori} upper bound on $\bar{L}$, effectively absorbing the discrepancy caused by the torsion.

Although the exact monotonicity of the reduced volume established by Perelman \cite{Per02} fails under the Chern-Ricci flow due to the presence of non-K\"ahler torsion, our local K\"ahler assumption ensures that the torsion terms remain uniformly bounded globally. By carefully controlling the torsion terms arising from the difference between the Chern-Ricci curvature and the Riemannian Ricci curvature, and the coupling between the torsion forms and the gradients of the Chern scalar curvature, we are able to establish an \textit{almost} monotonicity formula for the reduced volume. Specifically, for the integrand of the reduced volume, we obtain the following estimate:
$$
\mathcal{J}(v,\tau_1)\geq\mathcal{J}(v,\tau_2)\left(\frac{\tau_1}{\tau_2}\right)^{-n}e^{l(\gamma(\tau_1),\tau_1)-l(\gamma(\tau_2),\tau_2)}\cdot e^{ -C \int_{\tau_1}^{\tau_2} l(\gamma(\tau), \tau) d\tau -C(\tau_2-\tau_1)},
$$
where $l$ is the reduced length, $\mathcal{J}(v,\tau)$ is the Jacobian of the differential of the $\mathcal{L}$-exponential map, $v\in T_pM$, $0\leq\tau_1<\tau_2\leq\bar{\tau}$, and $C$ is a uniform constant (for more details, see \cref{section Per}).

Under the standard Ricci flow, exact monotonicity holds, and the exponential error factor $e^{ -C \int_{\tau_1}^{\tau_2} l(\gamma(\tau), \tau) d\tau -C(\tau_2-\tau_1)}$ on the right-hand side of the inequality is absent. In our setting, this discrepancy is directly induced by the torsion. However, by leveraging a crucial observation from \cite{LT23, LTZ26} (see \cref{lem:L-geodesic-exit-time}), we can guarantee that $\tau_1 \geq C^{-1}\bar{\tau}$ for the relevant geodesics in our arguments. Coupled with the upper bound $l(\cdot,\tau) \leq \frac{C}{\tau}$, this effectively controls the time integral. Consequently, we can uniformly bound the exponential error factor, thereby recovering a robust almost-monotonicity principle for the Chern-Ricci flow.

With this almost monotonicity in hand, we can then adapt the almost-avoidance principle due to \cite{LTZ26} to our setting. A further analytical difficulty arises when establishing the spatial gradient bound for the reduced length $L$, which is crucial for the final distance comparison. Under the standard Ricci flow on manifolds with non-negative curvature operators, this bound is typically derived from Hamilton's global matrix Harnack inequality \cite{Ham95} by Perelman \cite{Per02}. However, this global condition fails in our setting due to the unboundedness of curvatures and the presence of non-K\"ahler torsion. To overcome this gap, we exploit the local uniform bounds of the geometry away from the singular set and develop a backward exit-time argument for the ODE of $\mathcal{L}$-geodesics (see \cref{lem:local_gradient_L}). By doing so, we establish a local spatial gradient estimate for the reduced length, effectively bypassing the need for a global Harnack inequality. 

Based on these estimates, by analyzing the time-integrated volume of the tubular neighborhood of the singular set $D\subset X_{can}$ (which is known to have Minkowski dimension at most $2n-2$), we show that the vast majority of minimizing $\mathcal{L}$-geodesics do not spend much time in the highly degenerate region near the singularities. This crucial avoidance estimate allows us to effectively compare the space-time flow distance $d_T$ with the canonical limit distance $d_{can}$. 

\begin{remark}
We conclude by noting that our techniques may also be applicable to the collapsing case. Specifically, when $X$ is K\"ahler and $K_X$ is semiample, Iitaka's theorem provides a holomorphic fibration $f: X \to Y$ onto a normal projective variety $Y$ with $\dim Y = \kappa(X) = m$. This map restricts to a proper holomorphic submersion over $X \setminus f^{-1}(D)$, where $D \subset Y$ is the discriminant locus of $f$, with $(n-m)$-dimensional Calabi-Yau fibers. Let $\omega_0$ be an initial Hermitian metric. If it can be shown that $\omega(t)$ converges to $f^*\omega_{can}$ in $C^\infty_{loc}(X \setminus f^{-1}(D))$, extending the K\"ahler case results of \cite{TWY18, HLT25}, then the arguments herein may be adapted to establish the Gromov-Hausdorff convergence of $(X, \omega(t))$ to $Y$ when $\omega_0$ is K\"ahler in a neighborhood of $f^{-1}(D)$. The general Hermitian case remains largely open.
\end{remark}

The paper is organized as follows. In Section 2, we collect some preliminary results and fundamental estimates for the Chern-Ricci flow, including the recent $L^\infty$-estimates for complex Monge-Amp\`ere equations. Section 3 sets up the linear elliptic tools, explicitly relating the real and complex Laplacians via torsion forms. In Section 4, we establish uniform bounds for solutions to a family of Laplacian equations along the flow. Section 5 is devoted to the core analysis of this paper, where we prove uniform $L^1$, $L^p$, and gradient estimates for the Green's functions. In Section 6, we combine these Green's function estimates to prove the uniform diameter bounds and volume non-collapsing estimates (\cref{main theorem}). Section 7 derives a parabolic Schwarz lemma and the uniform Chern scalar curvature bound. Section 8 introduces Perelman's $\mathcal{L}$-length and reduced volume in the Chern-Ricci flow setting, establishing the almost monotonicity. Section 9 reduces the the Gromov-Hausdorff convergence to a key distance estimate. Section 10 proves the almost avoidance principle for the singular set. Finally, in Section 11, we synthesize these ingredients to complete the proof of the key distance estimate and \cref{main thm:uniqueness}.

\subsection*{Acknowledgements}
The author expresses sincere gratitude to his advisor, Professor Zhiwei Wang, for his continuous support and encouragement. The author would also like to thank Professors Jian Song, Valentino Tosatti, Bing Wang, and Kewei Zhang for their interest in this work, as well as their helpful comments and suggestions. 

\section{Preliminary estimates for the Chern-Ricci flow}
In this section, we first recall some known estimates for the immortal Chern-Ricci flow established in \cite{TW15} and \cite{Gill13} and improve them using our new a priori estimates of complex Monge-Amp\`ere equations. Let $(X,\omega_X)$ be a smooth Hermitian minimal model of general type, i.e., $(X,\omega_X)$ is a compact Hermitian manifold and the canonical line bundle $K_X$ is nef and big. In this case, $X$ is a Moishezon manifold and hence lies in the Fujiki class $\mathcal{C}$. Let $\omega_0$ be an arbitrary Hermitian metric on $X$, we consider the following Chern-Ricci flow starting from $\omega_0$:
\begin{equation}\label{main flow}
\left\{
    \begin{aligned}
        &\frac{\partial}{\partial t}\omega(t)=-Ric^C
        (\omega(t)),\\
        &\omega(0)=\omega_0.
    \end{aligned}
    \right.
\end{equation}
Here $Ric^C(\omega):=-\sqrt{-1}\partial\bar{\partial}\operatorname{log}\omega^n$ denotes the Chern-Ricci curvature of $\omega$. Since $K_X$ is nef, the flow \eqref{main flow} admits a smooth long-time solution $\omega(t)$ on $X$ by \cite[Theorem 2.1]{TWY15}. In this setting, it is more convenient to make a change of variable $\omega(t):=\frac{\tilde{\omega}(s)}{s+1}$ with the new parameter $t:=\operatorname{log}(1+s)$. Note that if $\tilde{\omega}(s)$ solves \eqref{main flow}, then $\omega(t)$ solves the following normalized Chern-Ricci flow:
\begin{equation}\label{normalized main flow}
\left\{
\begin{aligned}        
&\frac{\partial}{\partial t}\omega(t)=-Ric^C        
(\omega(t))-\omega(t),\\        
&\omega(0)=\omega_0.    
\end{aligned}
\right.
\end{equation}
We recall the definition of null locus:
\begin{definition}\label{def: Null locus}
    The null locus $\operatorname{Null}(K_X)$ of the big line bundle $K_X$ is defined to be the union of all positive-dimensional irreducible analytic subvarieties $V\subset X$ such that if $\operatorname{dim}_{\mathbb{C}}V=k$, then
    $$
\int_V(c_1^{BC}(K_X))^k>0.
    $$
We remark that when $X$ is K\"ahler or lies in the Fujiki class $\mathcal{C}$ (our case), it was shown by Collins-Tosatti \cite[Theorem 1.1]{CT15} that $\operatorname{Null}(K_X)$ coincides with the non-K\"ahler locus of $K_X$ and hence is itself a proper analytic subvariety of $X$. When $X$ is a general Hermitian manifold, some progress towards this issue was made by Dang in \cite{Dang24}.   
\end{definition}

The following convergence result was shown by Gill:
\begin{theorem}\cite[Theorem 1.1]{Gill13} \label{smooth convergence}
    Let $(X,\omega_0)$ be a smooth Hermitian minimal model of general type. Then, the normalized Chern-Ricci flow \eqref{normalized main flow} has a smooth solution $\omega(t)$ for all time and there exists a closed positive current $\omega_{KE}$ on $X$ that restricts to a smooth K\"ahler-Einstein metric on $X\setminus E$ with the following properties:
    \begin{itemize}
        \item $Ric^C(\omega_{KE})=-\omega_{KE}$ on $X\setminus E$,
        \item   $\omega(t)$ converges to $\omega_{KE}$ as positive currents on $X$,
        \item $\omega(t)\xrightarrow{C^\infty_{loc}(X\setminus{E})}\omega_{KE}$.
    \end{itemize}
\end{theorem}

Let $\chi\in K_X$ be a nef Bott-Chern $(1,1)$-form represent $-c_1^{BC}(X)$, we can find a smooth volume form $\Omega$ on $X$ such that $\chi=\sqrt{-1}\partial\bar{\partial}\operatorname{log}\Omega$. For our purpose of diameter estimates, there is no loss of generality in assuming that $\omega_0>\chi$, up to a rescaling of $\omega_0$. Indeed, for the unnormalized flow \eqref{main flow}, if we set $\tilde{\omega}_0=\lambda\omega_0$ for some $\lambda>0$, then the solution becomes $\tilde{\omega}(t)=\lambda\omega\left(\frac{t}{\lambda}\right)$. This implies that $\frac{\tilde{\omega}(t)}{t}=\omega(\frac{t}{\lambda})/\frac{t}{\lambda}$, which clearly does not affect the diameter estimate and the volume non-collapsing estimate. For the normalized flow \eqref{normalized main flow}, the initial metric does not change, so the diameter estimate and the volume non-collapsing estimate are also unaffected.

Set $\hat{\omega}_t:=\chi+e^{-t}(\omega_0-\chi)$, it is classical that (cf, \cite{Tos18}) the normalized Chern-Ricci flow can be written as the following parabolic Monge-Amp\`ere equation:
\begin{equation}\label{parabolic MA}
\left\{
    \begin{aligned}        
&\frac{\partial\varphi}{\partial t}=\operatorname{log}\frac{\hat{\omega}_t+\sqrt{-1}\partial\bar{\partial}\varphi}{\Omega}-\varphi,\\        
&\varphi|_{t=0}=0.    
\end{aligned}
\right.
\end{equation}
By \cref{smooth convergence}, there is a $\chi$-plurisubharmonic function $\varphi_{\infty}$, smooth on $X\setminus E$, such that $\varphi(t)\xrightarrow{C^\infty_{loc}(X\setminus{E})}\varphi_{\infty}$. We have the following estimates for $\varphi(t)$ and the volume form $\omega(t)^n$:
\begin{lemma}\cite[Lemma 3.1, Lemma 3.3]{Gill13} \label{estimates of varphi}
Suppose that $\psi$ is a $\chi$-plurisubharmonic function which is smooth away from $E=\operatorname{Null}(K_X)$ such that $\chi+\sqrt{-1}\partial\bar{\partial}\psi\geq\delta\omega_0$ for some $\delta>0$. Then, there is a uniform constant $C_\delta$ such that
    \begin{enumerate}
        \item $\varphi\leq C$.
        \item $\partial_t\varphi\leq C$.
        \item $\varphi\geq\delta\psi-C_\delta$.
        \item $\partial_t\varphi\geq\delta\psi-C_\delta$.
        \item $\frac{1}{C_\delta}e^{\delta\psi}\Omega\leq\omega(t)^n\leq C\Omega$.
    \end{enumerate}
\end{lemma}
In order to state the $L^\infty$ estimates for nef classes on Hermitian manifolds, we first recall the definition of upper and lower volumes introduced in \cite{GL22},\cite{BGL25}.
\begin{definition}
    Let $(X,\omega_X)$ be a compact Hermitian manifold and let $\chi$ be a nef Bott-Chern class, then we define the lower volume $\underline{\operatorname{Vol}}(\omega_X)$ of $\omega_X$ by
    $$
\underline{\operatorname{Vol}}(\omega_X):=\underset{u\in C^\infty(X)\cap \operatorname{PSH^+}(X,\omega_X)}{\inf}\int_X(\omega_X+dd^cu)^n=\underset{u\in L^\infty(X)\cap \operatorname{PSH}(X,\omega_X)}{\inf}\int_X(\omega_X+dd^cu)^n,
    $$
    where $\operatorname{PSH}^+(X,\omega)$ is the set of strictly $\omega$-plurisubharmonic functions and the second equality is a simple consequence of Demailly's regularization theorem \cite{Dem92} and Bedford-Taylor's weak convergence theorem \cite{BT82}. The lower volume of the nef class $\{\chi\}$ is defined as the following decreasing limit:
    $$
\underline{\operatorname{Vol}}(\{\chi\}):=\underset{\varepsilon\rightarrow0}{\lim}\underline{\operatorname{Vol}}(\chi+\varepsilon\omega_X).
    $$
    Similarly, the upper volume of $\omega_X$ is defined to be
     $$
\overline{\operatorname{Vol}}(\omega_X):=\underset{u\in C^\infty(X)\cap \operatorname{PSH^+}(X,\omega_X)}{\sup}\int_X(\omega_X+dd^cu)^n=\underset{u\in L^\infty(X)\cap \operatorname{PSH}(X,\omega_X)}{\sup}\int_X(\omega_X+dd^cu)^n.
$$
We say that the manifold $(X,\omega_X)$ has the bounded mass property if $\overline{\operatorname{Vol}}(\omega_X)<+\infty$. Correspondingly, $(X,\omega_X)$ has the positive volume property if $\underline{\operatorname{Vol}}(\omega_X)>0$.
\end{definition}

Using the recent $L^\infty$-estimates of complex Monge-Amp\`ere equations established in \cite{PSWZ25}, we can get much more precise estimates of the lower bound of $\varphi(t)$ along the Chern-Ricci flow.
\begin{theorem}\cite[Theorem 15.4]{PSWZ25} \label{a priori in nef class for MA}
    Let $\{\beta\}\in BC^{1,1}(X)$ be a nef Bott-Chern class satisfying $\underline{\operatorname{Vol}}(\{\beta\})>0$. Assume also $\varphi_t\in \operatorname{PSH}(X,\chi+t\omega_X)\cap C^\infty(X)$ satisfying
$$
(\beta+t\omega_X+dd^c\varphi_t)^n= c_te^{F_t}\omega_X^n,\quad\sup_X\varphi_t=0.
$$
Here $F_t\in C^\infty(X)$. We also fix a constant $p>n$. Then there exists a uniform constant $C$ depending on $\chi,\omega_X,n,p$, the upper bound of $\int_{X}e^{F_t(z)}[\log(1+e^{F_t(z)})]^p\omega_X^n$, the lower bound of $\int_Xe^{\frac{F_t}{n}}\omega_X^n$ and the lower bound of  $\underline{\operatorname{Vol}}(\{\beta\})$ such that
$$
0\leq-\varphi_t+\mathcal{V}_t\leq C.
$$
Here $\mathcal{V}_t:=\sup\{u\;|\;u\in \operatorname{PSH}(X,\beta+t\omega_X),u\leq0\}$ is the largest non-positive $(\beta+t\omega_X)$-plurisubharmonic function.
\end{theorem}
We remark that we could replace $\underline{\operatorname{Vol}}(\{\beta\})$ with the more precise quantity $SL_{\omega_X,a}(\{\beta\})$ in \cref{a priori in nef class for MA}, where $SL_{\omega_X,a}(\{\beta\})$ is the sup-slope first introduced in \cite{GS24}, and then generalized to nef classes in \cite{PSWZ25}. In this article, positivity of the lower volume is enough for our application.
\begin{proposition}\label{lower bound of varphi(t)}
        If $(X,\omega_X)$ has the positive volume property $\underline{\operatorname{Vol}}(\omega_X)>0$, then there is a uniform constant $C$ such that the solution $\varphi(t)$ of the Chern-Ricci flow satisfies 
        $$
\varphi(t)\geq \mathcal{V}_t-C.
        $$
\end{proposition}
\begin{proof}
Since $\chi$ is nef and big and $(X,\omega_X)$ has the positive volume property, it follows from \cite[Theorem 4.6]{GL22} (see also \cite[Theorem 3.20]{BGL25}) that $\underline{\operatorname{Vol}}(\chi)>0$. We can rewrite the flow \eqref{parabolic MA} as a family of complex Monge-Amp\`ere equations
$$
(\chi+e^{-t}(\omega_0-\chi)+\sqrt{-1}\partial\bar{\partial}\varphi(t))^n=e^{\varphi+\partial_t\varphi}\Omega.
$$
It follows from \cref{estimates of varphi} that $e^{2\delta\psi-2C_\delta}\leq e^{\varphi+\partial_t\varphi}\leq C$. Taking $c_t=1$ in \cref{a priori in nef class for MA}, it is then easy to check that all the requirements are met in \cref{a priori in nef class for MA}. We can then conclude the proof by invoking the uniform boundedness of $\sup_X\varphi(t)$ from \cref{estimates of varphi}.
\end{proof}

\begin{remark}\label{rmk:semi positive rep}
   Actually ,as noted in \cite{Gill13}, a result due to \cite{KMM87} implies that $\chi$ can be chosen as a semi-positive and big $(1,1)$-form, hence we can conclude from the proof of \cref{lower bound of varphi(t)} that $\varphi(t)$ is indeed uniformly bounded. Moreover, the a priori estimates in \cite{GL23} can also be applied to derive the uniform boundedness of $\varphi(t)$.
\end{remark}

Having the uniform boundedness of $\varphi(t)$, we are able to improve the estimates of $\partial_t\varphi$ in \cref{estimates of varphi}, the proof is more or less standard:
\begin{lemma}\label{lem:boundedness of partial_t varphi}
    There is a uniform constant $C$ such that $\frac{\partial^2\varphi}{\partial t^2}+\frac{\partial\varphi}{\partial t}\leq C$, and hence
    $$
\frac{\partial\varphi}{\partial t}\geq-C.
    $$
\end{lemma}
\begin{proof}
    Set $u:=\frac{\partial\varphi}{\partial t}+\varphi$. Then we have $\omega(t)^n=e^u\Omega$. Therefore, we can write
    \begin{align*}
\Delta_{\omega(t)}u&=\Delta_{\omega(t)}\log\frac{\omega(t)^n}{\Omega}=tr_{\omega(t)}\left(\sqrt{-1}\partial\bar{\partial}\log\omega(t)^n-\sqrt{-1}\partial\bar{\partial}\log\Omega\right)\\
&=tr_{\omega(t)}(-Ric^C(\omega(t)))-\chi=-S_C-tr_{\omega(t)}(\chi)\leq C-tr_{\omega(t)}(\chi),
    \end{align*}
    where we have used the fact that the Chern scalar curvature $S_C$ is uniformly bounded from below. Actually, it is easy to see that the evolution of the Chern scalar curvature is 
   \begin{align*}
\left(\frac{\partial}{\partial t} - \Delta_{\omega(t)}\right)S_C(t)&=|Ric^C|^2+S_C(t)\ge \frac{S_C^2}{n} + S_C(t)\\
&=\frac{1}{n}(S_C(t)+n)^2-(S_C(t)+n).
   \end{align*}
 This implies that
\begin{align*}
    \left(\frac{\partial}{\partial t} - \Delta_{\omega(t)}\right)(e^{t}(S_C+n))\geq0,
\end{align*}
standard parabolic maximum principle then yields that
$$
S_C(t)+n\geq\left(\inf_XS_C(0)+n\right)e^{-t},
$$
whence the uniform lower bound for $S_C(t)$.

Consequently, we can write
   \begin{align*}
       \frac{\partial u}{\partial t}=\frac{\partial^2\varphi}{\partial t^2}+\frac{\partial\varphi}{\partial t}=&tr_{\omega(t)}\left(-e^{-t}(\omega_0-\chi)+\sqrt{-1}\partial\bar{\partial}\dot{\varphi}(t)\right)\\
       =&\Delta_{\omega(t)}u-n+tr_{\omega(t)}(\chi)\leq C.
   \end{align*}
   Now, basic calculus (see e.g. \cite[Lemma 9.6]{GPSS24a}) shows that $\frac{\partial\varphi}{\partial t}$ is uniformly bounded from below.
\end{proof}
As an immediate corollary, we obtain that the volume forms $\omega(t)^n$ are uniformly equivalent:
\begin{corollary}\label{cor:uniform equivalence of volume forms}
    There is a uniform constant $C$ such that
    $$
C^{-1}\omega_0^n\leq\omega(t)^n\leq C\omega_0^n.
    $$
\end{corollary}
\begin{proof}
    This follows immediately from the flow equation $\omega(t)^n=e^{\varphi+\partial_t\varphi}\Omega$ and the fact that $\varphi,\partial_t\varphi$ are all uniformly bounded.
\end{proof}
\begin{remark}
    We remark that the estimates in this section is valid on an arbitrary Hermitian minimal model of general type, the K\"ahler assumption of $\omega_0$ near $E$ is not required.
\end{remark}
\section{Elliptic tools}
Two technical lemmas of linear elliptic operators established in the classical book \cite{GT01} will be frequently used in the sequel.

\begin{theorem}\cite[Theorem 9.11]{GT01}\label{GT 9.11}
    Let $U$ be an open subset in $\mathbb{R}^n$ and let $u\in W^{2,p}(U)\cap L^p(U)$ be a strong solution of the elliptic equation
    $$
Lu=f \quad \operatorname{in}\,U.
    $$
    Here $Lu:=a^{ij}(x)D_{ij}u+b^i(x)D_iu+c(x)u$ is a second-order linear elliptic operator. If there are uniform constants $\lambda,\Lambda$ such that
    \begin{align*}
       & a^{ij}\in C^0(U), \quad b^i,c\in L^\infty(U),\quad f\in L^p(U);\\
        &a^{ij}\xi_i\xi_j\geq\lambda|\xi|^2,\quad\forall\xi\in\mathbb{R}^n;\\
        &|a^{ij}|,|b^i|,|c|\leq\Lambda.
    \end{align*}
    Then, for any smaller domain $U_1\Subset U$, there is a uniform constant $C=C\left(n,p,\lambda,\Lambda,U_1,U,\|a^{ij}\|_{C^0(U_1)}\right)$ such that
    $$
\|u\|_{W^{2,p}(U_1)}\leq C\left(\|u\|_{L^p(U)}+\|f\|_{L^p(U)}\right).
    $$
\end{theorem}
\begin{theorem}\cite[Theorem 9.20]{GT01}\label{GT 9.20}
    Notations as in \cref{GT 9.11}, assume $u\in W^{2,n}(U)$ and $Lu\geq f$ for some $f\in L^{n}(U)$. Denote this time $\lambda,\Lambda$ the smallest and largest eigenvalues of $a^{ij}(x)$ and assume moreover that there are uniform constants $\gamma,\nu>0$ such that
    $$
\frac{\Lambda}{\lambda}\leq\gamma,\quad\frac{|b|}{\lambda},\frac{|c|}{\lambda}\leq\nu.
    $$
    Then for each ball $B_{2R}(y)\subset U$, we have the following Harnack type inequality:
    $$
\sup_{B_R(y)}u\leq C\left[\int_{B_{2R}(y)}|u|dV+\frac{R}{\lambda}\|f\|_{L^n(B_{2R}(y))}\right],
    $$
    where $C$ depends on $n,\gamma,\nu R^2,U$.
\end{theorem}

We will mainly use  these two lemmas for the real and complex Laplacian. Let $g$ be the Riemannian metric associated to a Hermitian metric $\omega$ on $X$, i.e., if $J$ denotes the complex structure of $X$, then $g(JX,Y)=\omega(X,Y)$ for each real vector fields $X,Y$. The real Laplacian of $g$ acting on functions is defined to be
$$
\Delta_gf:=d^*df=\frac{1}{\sqrt{G}}\partial_A\left(\sqrt{G}g^{AB}\partial_Bf\right),
$$
Here $G$ denotes the determinant of the real matrix $(g_{AB})_{1\leq A,B\leq 2n}$. While the complex Laplacian is defined to be
$$
\Delta_\omega f:=\operatorname{tr}_\omega(\sqrt{-1}\partial\bar{\partial}f)=g^{j\bar{k}}\partial_j\partial_{\bar{k}}f.
$$
In order to illustrate the connection between real and complex Laplacians, we introduce the notion of torsion forms:
\begin{definition}
    Let $(X,\omega)$ be a compact Hermitian manifold, let $L_\omega$ be the Lefschetz operator associated to the metric $\omega$, i.e., for each form $\alpha\in\Lambda^k(TX\otimes\mathbb{C})^*$, $L_\omega\alpha:=\omega\wedge\alpha$. We denote $\Lambda_\omega$ its dual operator, that is, $\langle L_\omega\alpha,\beta\rangle=\langle\alpha,\Lambda_\omega\beta\rangle$. As in \cite[Chapter VI, Theorem (6.8)]{Dem}, we define a $(1,0)$-form $\tau$ to be $\tau:=[\Lambda_\omega,\partial\omega]$ and set $\theta:=\tau+\bar{\tau}$. Here, $[\cdot,\cdot]$ denotes the Lie bracket of two complex differential forms.
\end{definition}
\begin{lemma}\label{lemma:torison}
    The $(1,0)$-form $\tau$ satisfies $\tau\wedge\omega^{n-1}=\partial\omega^{n-1}$, and is given by
    $$
\tau=\Lambda_\omega(\partial\omega)=\left(g^{j\bar{k}}\partial_lg_{j\bar{k}}-g^{j\bar{k}}\partial_jg_{l\bar{k}}\right)dz^l
    $$
    in local coordinates. Here we are only concerned about the action of $\tau$ on functions, so $\tau=\Lambda_\omega(\partial\omega)$ since the operator $\Lambda_\omega$ has bidegree $(-1,-1)$.
\end{lemma}
\begin{proof}
    We only prove the first statement, the second follows from direct computations. Since $\partial\omega\in\Lambda^3(TX\otimes\mathbb{C})^*$, we can use \cite[Chapter VI, (5.10)]{Dem} to write
    \begin{align*}
        [L_\omega^{n-1},\Lambda_\omega](\partial\omega)=(n-1)(3-n+n-2)L_\omega^{n-2}(\partial\omega)=(n-1)\partial\omega\wedge\omega^{n-2}=\partial\omega^{n-1}.
    \end{align*}
    On the other hand, 
    $$
[L_\omega^{n-1},\Lambda_\omega](\partial\omega)=L_\omega^{n-1}\Lambda_\omega(\partial\omega)=L_\omega^{n-1}\tau=\tau\wedge\omega^{n-1}.
    $$
    The proof is therefore concluded.
\end{proof}
Now standard computations using local coordinates yield the following relationship between real and complex Laplacians:
\begin{lemma}\label{real and complex Laplacian}
    For each $f\in C^2(X)$, we have
    $$
\Delta_gf=2\Delta_\omega f+2\langle df,\theta\rangle_{\omega}=2\Delta_\omega f+2\langle\nabla f,\theta^{\#}\rangle_{\omega},
    $$
    where we raise the index of the $1$-form $\theta$ to obtain the vector field $\theta^\#$. 
\end{lemma}

\section{Estimates for families of Laplacian equations}
In this section, we establish estimates of solutions to a family of Laplacian equations along the Chern-Ricci flow, following the approach in \cite{GPSS24a} and \cite{GPS24}. 

Return to the background in \cref{conjecture}, since $K_X$ is nef and $\kappa(X)=n$, we automatically deduce that $X$ is a K\"ahler manifold when $n=2$. For $n\geq3$, $X$ is a Moishezon manifold, namely, it can be birationally embedded into a projective space. Consequently, it automatically lies in the Fujiki class $\mathcal{C}$ \cite{Fuj78}, i.e., it is bimeromorphic to a compact K\"ahler manifold. Since the bounded mass property and the positive volume property are both bimeromorphic invariants (see e.g. \cite[Theorem A]{GL22}), we obtain that $X$ satisfies both properties. 

Recall that we have chosen a nef representative $\chi$ of $K_X$ and set $\hat{\omega}_t:=\chi+e^{-t}(\omega_0-\chi)$, where $\omega_0$ is the initial metric for the normalized Chern-Ricci flow such that $\omega_0>\chi$. As in the proof of \cref{lower bound of varphi(t)}, we may use \cite[Theorem 4.6]{GL22} to obtain that $\underline{\operatorname{Vol}}(\chi)>0$. Consequently, we easily have that the volumes $V_{\omega_t}:=\int_X\omega_t^n$ along the flow is uniformly bounded away from zero.

For simplicity of notations, we also write $\omega_t=\hat{\omega}_t+\sqrt{-1}\partial\bar{\partial}\varphi_t$ as the solution of the normalized Chern-Ricci flow. Since the volumes $V_{\omega_t}$ are uniformly bounded away from zero, we will define the $p$-Nash Entropy of $\omega_t$ by
$$
\operatorname{Ent}_p(\omega_t):=\int_X\left|\log\left(\frac{\omega_t^n}{\omega_X^n}\right)\right|^p\omega_t^n=\int_Xe^{F_t}|F_t|^p\omega_X^n,
$$
where $e^{F_t}:=\frac{\omega_t^n}{\omega_X^n}$. It then follows easily from \cref{estimates of varphi} and the integrability of quasi-plurisubharmonic functions that $\operatorname{Ent}_p(\omega_t)$ remains uniformly bounded along the flow \eqref{main flow}.

\begin{lemma}\label{laplacian estimate 1}
    Suppose that $v\in L^1(X,\omega_t^n)$ is a function which satisfies $\int_Xv\omega_t^n=0$ and 
    $$
v\in C^2(\overline{\Omega_0}),\quad\Delta_{\omega_t}v\geq-a\quad \operatorname{on}\Omega_0,
    $$
    where $\Omega_s:=\{v>s\}$ for each $s\geq0$ and $a>0$ is a given constant. Then, there is a constant $C>0$ depending on $n,p,\chi,\underline{\operatorname{Vol}}(\chi),\omega_X,\operatorname{Ent}_p(\omega_t)$ such that
    $$
\sup_Xv\leq C(a+\|v\|_{L^1(X,\omega_t^n)}).
    $$
\end{lemma}
\begin{proof}
 We first claim that we can assume that $\|v\|_{L^1(X,\omega_t^n)}\leq 1$. If $v$ itself satisfies $\|v\|_{L^1(X,\omega_t^n)}\leq 1$, we are done. Otherwise if $\|v\|_{L^1(X,\omega_t^n)}\geq1$, replace $v$ by $\hat{v}:=\frac{v}{\|v\|_{L^1(X,\omega_t^n)}}$, we get $\Delta_{\omega_t}\hat{v}\geq-\frac{a}{\|v\|_{L^1(X,\omega_t^n)}}$ on $\Omega_0$, then the conclusion $\sup_X\hat{v}\leq C\left(\frac{a}{\|v\|_{L^1(X,\omega_t^n)}}+1\right)$ will imply the corresponding estimate of $v$. Also by considering $\frac{v}{a}$ we can assume that $a=1$. It thus remains to prove that $\sup_Xv\leq C$.
 
    \textbf{Step 1.} Choosing a sequence of smooth functions $\tau_k:\mathbb{R}\rightarrow\mathbb{R}^+$ such that $\tau_k(x)$ decreases to the function $x\cdot\mathds{1}_{\mathbb{R}^+}(x)$ as $k\rightarrow+\infty$ and $\tau_k(x)\equiv\frac{1}{k}$ for $x\leq-\frac{1}{2}$. These choice ensure that $\tau_k(v-s)$ belongs to $C^2(X)$ when $s>\frac{1}{2}$.
    For each $s>\frac{1}{2}$ and $k\in\mathbb{Z}^+$, we apply the main theorem of \cite{TW10} to solve the following auxiliary Monge-Amp\`ere equations
    \begin{equation}\label{auxiliary MA 1}
        (\hat{\omega}_t+\sqrt{-1}\partial\bar{\partial}\psi_{t,k,s})^n=c_{t,k,s}\frac{\tau_k(v-s)}{A_{t,k,s}}\omega_t^n,\quad\sup_X\psi_{t,k,s}=0,
    \end{equation}
    where 
    \begin{align*}
        A_{t,k,s}:=\int_X\tau_k(v-s)\omega_t^n\to A_{t,s}:=\int_{\Omega_s}(v-s)\omega_t^n\,\,\,\operatorname{as}k\to\infty.
    \end{align*}
    We shall assume that $\Omega_s$ is non-empty so that $A_{t,s}>0$, for otherwise we are already done. Now integrating both sides of \eqref{auxiliary MA}, we get that
    $$
c_{t,k,s}=\int_X(\hat{\omega}_t+\sqrt{-1}\partial\bar{\partial}\psi_{t,k,s})^n.
    $$
    It follows that $c_{t,k,s}\geq\underline{\operatorname{Vol}}(\hat{\omega}_t)$. As in the proof of \cref{lower bound of varphi(t)}, invoking that $X$ satisfies the positive volume property, we may use \cite[Theorem 4.6]{GL22} to obtain that $\underline{\operatorname{Vol}}(\chi)>0$. The monotonicity of lower volumes (see \cite[Proposition 3.7]{BGL25}) then yields that $\underline{\operatorname{Vol}}(\hat{\omega}_t)\geq\underline{\operatorname{Vol}}(\chi)>0$ since we have assumed that $\omega_0>\chi$. Similarly, the bounded mass property of $X$ ensures that $c_{t,k,s}$ is uniformly bounded from above. Consequently, we conclude that $c_{t,k,s}$ is uniformly bounded away from zero.

    \textbf{Step 2.} As in \cite[Lemma 2]{GPS24}, we apply a priori estimates \cref{lower bound of varphi(t)} and \cref{estimates of varphi} to choose $\Lambda:=C_0+1$ for a uniform constant $C_0$ satisfying $|-\varphi_t+\mathcal{V}_t|\leq C_0,\varphi_t\leq C_0$ and set
    $$
\Phi_{t,k,s}:=-\varepsilon_{t,k,s}(-\psi_{t,k,s}+\varphi_t+\Lambda)^{\frac{n}{n+1}}+(v-s)
    $$
    for some constant $\varepsilon_{t,k,s}$ to be assigned. By the assumption on $v$ we have that $\Phi_{t,k,s}\in C^2(\overline{\Omega_0})$. Our goal is to choose appropriate $\varepsilon_{t,k,s}$ such that $\Phi_{t,k,s}\leq0$ on $X$. By the choice of $\Lambda$, we automatically have that $-\psi_{t,k,s}+\varphi_t+\Lambda=(\mathcal{V}_t-\psi_{t,k,s})+(\varphi_t-\mathcal{V}_t+C_0)+1\geq1$ and hence $\Phi_{t,k,s}|_{X\setminus\Omega_s}<0$, so we may assume that $\Phi_{t,k,s}$ achieves maximum at some point $x_0\in\Omega_s$. Invoking the maximum principle, we can write
    \begin{align*}
        0&\geq\Delta_{\omega_t}\Phi_{t,k,s}(x_0)\\
        &\geq-\frac{\varepsilon_{t,k,s} n}{n+1}(-\psi_{t,k,s}+\varphi_t+\Lambda)^{-\frac{1}{n+1}}(\operatorname{tr}_{\omega_t}(\sqrt{-1}\partial\bar{\partial}(\varphi_t-\psi_{t,k,s})))+\Delta_{\omega_t}v\\
        &=\frac{\varepsilon_{t,k,s} n}{n+1}(-\psi_{t,k,s}+\varphi_t+\Lambda)^{-\frac{1}{n+1}}\operatorname{tr}_{\omega_t}(\hat{\omega_t}+\sqrt{-1}\partial\bar{\partial}\psi_{t,k,s})-\frac{\varepsilon_{t,k,s}n^2}{n+1}+\Delta_{\omega_t}v\\
        &\geq\frac{\varepsilon_{t,k,s} n^2}{n+1}(-\psi_{t,k,s}+\varphi_t+\Lambda)^{-\frac{1}{n+1}}\left(\frac{(\hat{\omega_t}+\sqrt{-1}\partial\bar{\partial}\psi_{t,k,s})^n}{\omega_t^n}\right)^{\frac{1}{n}}-\frac{\varepsilon_{t,k,s}n^2}{n+1}+\Delta_{\omega_t}v\\
        &\geq c_{t,k,s}^{\frac{1}{n}}\frac{(v-s)^{\frac{1}{n}}}{A_{t,k,s}^{\frac{1}{n}}}\frac{\varepsilon_{t,k,s} n^2}{n+1}(-\psi_{t,k,s}+\varphi_t+\Lambda)^{-\frac{1}{n+1}}-\varepsilon_{t,k,s}n-1.
    \end{align*}
    It follows that 
    $$
-\varepsilon_{t,k,s}(-\psi_{t,k,s}+\varphi_t+\Lambda)^{\frac{n}{n+1}}\leq-\frac{c_{t,k,s}\varepsilon_{t,k,s}^{n+1}}{(1+\varepsilon_{t,k,s}n)^n}\left(\frac{n^2}{n+1}\right)^n\frac{v-s}{A_{t,k,s}}.
    $$
    Now, we choose 
    $$
\varepsilon_{t,k,s}^{n+1}:=\frac{1}{c_{t,k,s}}\left(\frac{n+1}{n^2}\right)^n(1+\varepsilon_{t,k,s}n)^nA_{t,k,s}
    $$
    to conclude that $\Phi_{t,k,s}\leq0$ on $\Omega_s$.

    \textbf{Step 3.} The assumption $\|v\|_{L^1(X,\omega_t^n)}\leq 1$ we made implies that $A_{t,s}=\int_{\Omega_s}(v-s)\omega_t^n\leq 1$ for each $s>\frac{1}{2}$, hence $A_{t,k,s}\leq 2$ for fixed $t,s$ and large $k$. This combined with the uniform lower bound of $c_{t,k,s}$ implies that 
    $$
\varepsilon_{t,k,s}\leq C(\underline{\operatorname{Vol}}(\chi),n)A_{t,k,s}^{\frac{1}{n+1}}.
    $$
    This estimate combined with $\Phi_{t,k,s}\leq0$ yields that
    $$
(v-s)A_{t,k,s}^{-\frac{1}{n+1}}\leq C(\underline{\operatorname{Vol}}(\chi),n)(-\psi_{t,k,s}+\varphi_t+\Lambda)^{\frac{n}{n+1}}.
    $$
    On $\Omega_s=\{v>s\}$, we can take a power to write
    $$
\frac{(v-s)^{\frac{n+1}{n}}}{A_{t,k,s}^{\frac{1}{n}}}\leq C(\underline{\operatorname{Vol}}(\chi),n)^{\frac{n+1}{n}}(-\psi_{t,k,s}+\varphi_t+\Lambda)\leq C(\underline{\operatorname{Vol}}(\chi),n)^{\frac{n+1}{n}}(-\psi_{t,k,s}+C_0+\Lambda).
    $$
    By Skoda's uniform integrability theorem, we can find a small uniform constant $\alpha=\alpha(\chi,\omega_0)$ such that
    \begin{equation}\label{exp lemma}
        \int_{\Omega_s}\operatorname{exp}\left(\alpha\frac{(v-s)^{\frac{n+1}{n}}}{A_{t,k,s}^{\frac{1}{n}}}\right)\omega_X^n\leq C_2\int_X \operatorname{exp}\left(-\alpha C(\underline{\operatorname{Vol}}(\chi),n)^{\frac{n+1}{n}}\psi_{t,k,s}\right)\omega_X^n\leq C_3
    \end{equation}
    for some uniform constant $C_2,C_3$ depending on the given data.

    \textbf{Step 4.} Given \eqref{exp lemma}, the argument is now a direct adaptation from \cite{GPS24, GPT23}, we give the details for the reader's convenience. Fix $p>n$ and define $\eta:\mathbb{R}^+\rightarrow\mathbb{R}^+$ by $\eta(x):=[\log(1+x)]^p$. Observe that $\eta$ is an increasing function with $\eta(0)=0$ and let $\eta^{-1}(y)=\exp(y^{\frac{1}{p}})-1$ be its inverse function. Recall that Young's inequality yields for any numbers $a,b\geq0$,
    \begin{align}\label{eq 21}
        ab&\leq\int_0^a\eta(x)dx+\int_0^b\eta^{-1}(y)dy=\int_0^a\eta(x)dx+\int_0^{\eta^{-1}(b)}x\eta^{\prime}(x)dx\\
        &\leq a\cdot\eta(u)+b\cdot\eta^{-1}(b)=a[\log(1+a)]^p+b(\exp(b^{\frac{1}{p}})-1).
    \end{align}
Set
$$
w:=\frac{\alpha}{2}\frac{ (v-s)^{\frac{n+1}{n}}}{A_{t,k,s}^{\frac{1}{n}}}
$$
We apply \eqref{eq 21} with $a=e^{F_t}$ and $w^p$ to derive
\begin{align}
    w(z)^pe^{F_t(z)}\leq e^{F_t(z)}[\log(1+e^{F_t(z)})]^p+w(z)^p(e^{w(z)}-1)\leq e^{F_t(z)}\left(1+|F_t(z)|\right)^p+C_pe^{2w(z)}.
\end{align}
Plugging the value of $w$ and integrating both sides over $\Omega_s$ we obtain
\begin{align}
   & \left(\frac{\alpha}{2}\right)^p\int_{\Omega_s}\frac{ (v-s)^{\frac{p(n+1)}{n}}}{A_{t,k,s}^{\frac{p}{n}}}e^{F_t}\omega_X^n\\
    \leq&\int_{\Omega_s}e^{F_t(z)}(1+|F_t(z)|)^p\omega_X^n+C_p\int_{\Omega_s} \exp\left({\alpha\frac{ (v-s)^{\frac{n+1}{n}}}{A_{t,k,s}^{\frac{1}{n}}} }\right)\omega_X^n\leq C.\\
\end{align}
    Where the last inequality follows from \eqref{exp lemma} and $C$ is a uniform constant depending on the given data. As a consequence, we can write
    \begin{equation}\label{core inequality}
        \int_{\Omega_s}(v-s)^{\frac{p(n+1)}{n}}e^{F_t}\omega_X^n\leq CA_{t,k,s}^{\frac{p}{n}}\to CA_{t,s}^{\frac{p}{n}}\quad\operatorname{as}k\to\infty.
    \end{equation}
   On the other hand, applying H\"older's inequality and invoking the definition of $A_{t,s}$ to write
    \begin{align}
        A_{t,s}=&\int_{\Omega_s}(v-s)e^{F_t}\omega_X^n\leq\left(\int_{\Omega_s}(v-s)^{\frac{p(n+1)}{n}}e^{F_t}\omega_X^n\right)^{\frac{n}{p(n+1)}}\cdot\left(\int_{\Omega_s}e^{F_t}\omega_X^n\right)^{\frac{1}{q}}\\
        \leq&CA_{t,s}^{\frac{1}{n+1}}\cdot\left(\int_{\Omega_s}e^{F_t}\omega^n\right)^{\frac{1}{q}}.
    \end{align}
    Where $q:=\frac{p(n+1)}{p(n+1)-n}$. The above inequality can be written by
    $$
A_{t,s}\leq C\left(\int_{\Omega_s}e^{F_t}\omega_X^n\right)^{\frac{n+1}{nq}}.
    $$
    Observe that $\frac{n+1}{nq}=\frac{p(n+1)-n}{pn}:=1+\delta_0>1$. Define $\phi:\mathbb{R}\rightarrow\mathbb{R}$ by $\phi(s):=\int_{\Omega_s}e^{F_t}\omega_X^n$. By the definition of $A_{t,s}$ it is easy to check that
    $$
r\phi(s+r)=\int_{\Omega_{s+r}}re^{F_t}\omega_X^n\leq\int_{\Omega_{s}}(v-s)e^{F_t}\omega_X^n=A_{t,s}\leq C\phi(s)^{1+\delta_0},
    $$
    for any $0\leq r\leq1$. The result then follows from a classic lemma due to De Giorgi, see \cite[Lemma 2]{GPT23} and \cite[Lemma 2, Step 4]{GPS24}. The proof is therefore concluded.
\end{proof}

The following corollary is an immediate consequence of \cref{laplacian estimate 1}:
\begin{corollary}\label{coro:laplacian estimate 1}
  Assume that $d\omega_0=0$ in a neighborhood $U$ of $E=\operatorname{Null}(K_X)$. Suppose that $v\in C^2(X)$ satisfies 
    $$
|\Delta_{\omega_t}v|\leq 1,\quad \int_Xv\omega_t^n=0.
    $$
    Then, there exists a uniform constant $C=C(n,p,\chi,\underline{\operatorname{Vol}}(\chi),\omega_X,\operatorname{Ent}_p(\omega_t))$ such that
    $$
\sup_X|v|\leq C(1+\|v\|_{L^1(X,\omega_t^n)}).
    $$
\end{corollary}

\begin{lemma}\label{Laplacian estimate 2}
    Suppose that $v\in C^2(X)$ satisfies 
    $$
|\Delta_{\omega_t}v|\leq 1,\quad \int_Xv\omega_t^n=0.
    $$
    Then, there exists a uniform constant $C=C(n,p,\chi,\underline{\operatorname{Vol}}(\chi),\omega_X,\operatorname{Ent}_p(\omega_t))$ such that
$$
\int_X|v|\omega_t^n\leq C.
$$
\end{lemma}
\begin{proof}
    We will mainly follow the blow-up arguments in \cite[Lemma 5.3]{GPSS24a}. Suppose that there is a sequence of metrics $\omega_{t_j}$ along the CRF (with $t_j\to+\infty$) and $v_j\in C^2(X)$ such that 
    $$
\Delta_{\omega_{t_j}}v_j=h_j,\quad\int_Xv_j\omega_{t_j}^n=0,
    $$
    where $h_j$ are continuous functions satisfying $|h_j|\leq1$, and we have
    $$
\int_X|v_j|\omega_{t_j}^n:=N_j\to+\infty,
    $$
    as $j\to+\infty$. Set $\hat{v}_j:=\frac{v_j}{N_j}$, then we have
    $$
|\Delta_{\omega_{t_j}}\hat{v}_j|=\frac{|h_j|}{N_j}\to0,\quad\int_X|\hat{v}_j|\omega_{t_j}^n=1.
    $$
\cref{coro:laplacian estimate 1} then yields a uniform constant $C$ such that
\begin{equation}\label{eq:sup vj}
\sup_X|\hat{v}_j|\leq C.
\end{equation}
Let $\tau_j:=\Lambda_{\omega_{t_j}}(\partial{\omega_{t_j}})=e^{-{t_j}}\Lambda_{\omega_{t_j}}(\partial{\omega_{0}})$ be the torsion $(1,0)$-form of $\omega_{t_j}$. Since we have assumed that $d\omega_0=0$ in a neighborhood $U$ of $\operatorname{Null}(K_X)$ and that $\omega_{t_j}$ are uniformly equivalent outside $U$, we can write
\begin{align*}
    \int_X|\nabla\hat{v}_j|^2_{\omega_{t_j}}\omega_{t_j}^n&\leq C\int_X\sqrt{-1}\partial\hat{v}_j\wedge\bar{\partial}\hat{v}_j\wedge\omega_{t_j}^{n-1}\\
    &=\frac{C}{n}\int_X(-\hat{v}_j)(\Delta_{\omega_{t_j}}\hat{v}_j)\omega_{t_j}^n+C\int_X(-\hat{v}_j)\bar{\partial}\hat{v}_j\wedge\partial\omega_{t_j}^{n-1}\\
    &\leq C\int_X|\hat{v}_j||\frac{h_j}{N_j}|\omega_{t_j}^n+C\int_X(-\hat{v}_j)\bar{\partial}\hat{v}_j\wedge\tau_j\wedge\omega_{t_j}^{n-1}\\
    &\leq C\int_X|\hat{v}_j||\frac{h_j}{N_j}|\omega_{t_j}^n+C\int_X|\nabla\hat{v}_j|_{\omega_{t_j}}|\tau_j|_{\omega_{t_j}}\omega_{t_j}^n\\
    &\leq C\int_X|\hat{v}_j||\frac{h_j}{N_j}|\omega_{t_j}^n+C\left(\int_X|\nabla\hat{v}_j|^2_{\omega_{t_j}}\omega_{t_j}^n\right)^{\frac{1}{2}}\left(\int_X|\tau_j|_{\omega_{t_j}}^2\omega_{t_j}^n\right)^\frac{1}{2},
\end{align*}
where in the second inequality we have used \cref{lemma:torison} and in the third inequality we have used \eqref{eq:sup vj}. Since $|\tau_j|_{\omega_{t_j}}=o(e^{-t_j})$ and $|\frac{h_j}{N_k}|\to0$ uniformly, we conclude that 
\begin{equation}\label{eq:L2 of nabla v_j}
\underset{j\to\infty}{\lim}  \int_X|\nabla\hat{v}_j|^2_{\omega_{t_j}}\omega_{t_j}^n\to0.
\end{equation}
Since $X$ is a connected complex manifold and $E=\operatorname{Null}(K_X)$ is a proper analytic subvariety, it is well-known that $E$ is automatically thin in $X$ and hence $X\setminus E$ is connected. For each $\varepsilon>0$, we can thus find small connected neighborhoods $U_\varepsilon\Subset U_{2\varepsilon}\Subset U$ with smooth boundaries such that $X\setminus U_\varepsilon,X\setminus U_{2\varepsilon}$ are connected and that $|U_\varepsilon|<\varepsilon,|U_{2\varepsilon}|<2\varepsilon$ (here we use $|\cdot|$ to denote the volume of a set with respect to $\omega_X^n$).

Now Since $\omega_{t_j}$ are uniformly equivalent outside $U_\varepsilon$, we conclude from \eqref{eq:L2 of nabla v_j} that
\begin{equation}\label{eq:L2 limit of vj 2}
    \underset{j\to\infty}{\lim}\int_{X\setminus U_{\varepsilon}}|\nabla\hat{v}_j|^2_{\omega_X}\omega_X^n=0.
\end{equation}
It follows that the sequence $\hat{v}_j$ is uniformly bounded in the sobolev space $W^{1,1}(X\setminus U_{\varepsilon},\omega_X^n)$ and hence we can use sobolev's embedding theorem to extract a converging subsequence (still denoted by $\hat{v}_j$) such that $\hat{v}_j\to v_\infty$ in $L^q(X\setminus U_{2\varepsilon},\omega_X^n)$ for some $q>1$. By \eqref{eq:sup vj} we have that $\hat{v}_j$ is uniformly bounded, thus so does $v_\infty$. Recall that we have smooth convergence $\omega_{t_j}\to\omega_{KE}$ on $X\setminus U_{2\varepsilon}$, this clearly yields uniform convergence of the coefficients of the formal adjoint operators $\Delta_{\omega_{t_j}}^*\to\Delta_{\omega_{KE}}^*$ on $X\setminus U_{2\varepsilon}$. Consequently, we easily have the convergence 
$$\frac{h_j}{N_j}=\Delta_{\omega_{t_j}}\hat{v}_j\to\Delta_{\omega_{KE}}v_\infty$$
on $X\setminus U_{2\varepsilon}$ in the distributional sense. Therefore, we have $\Delta_{\omega_{KE}}v_\infty=0$ on $X\setminus U_{2\varepsilon}$ and hence $v_\infty$ is smooth on $X\setminus U_{2\varepsilon}$ by standard elliptic regularity theory (see for example \cite[Chapter IV, Theorem 4.9]{Wells08}).

Now we claim that $v_\infty$ is constant in $X\setminus U_{2\varepsilon}$. For each smooth vector field $Y\in C_c^{\infty}(X\setminus U_{2\varepsilon})$, we use the divergence theorem for the underlying Riemannian manifold to write
\begin{align*}
    \int_X\langle\nabla v_\infty,Y\rangle_{\omega_X}\omega_X^n&=-\int_Xv_\infty\cdot\operatorname{div}_{\omega_X}Y\omega_X^n=-\underset{j\to\infty}{\lim}\int_X\hat{v}_j\operatorname{div}_{\omega_X}Y\omega_X^n\\
  &  =\underset{j\to\infty}{\lim} \int_X\langle\nabla \hat{v}_j,Y\rangle_{\omega_X}\omega_X^n\leq C\underset{j\to\infty}{\lim} \int_{X\setminus U_{\varepsilon}}|\nabla\hat{v}_j|_{\omega_X}\omega_X^n=0.
\end{align*}
Here, we have used \eqref{eq:L2 limit of vj 2} in the last equality. This immediately yields that $\nabla v_\infty\equiv0$ on $X\setminus U_{2\varepsilon}$, whence the claim follows.

It remains to establish a contradiction. Write $\hat{v}_j=\hat{v}_j^+-\hat{v}_j^-$ and set $c=v_\infty$ on $X\setminus U_{2\varepsilon}$.

\textbf{Case 1.} Suppose first $c\geq0$. Since $\int_X\hat{v}_j\omega_{t_j}^n=0$, we have $\int_X\hat{v}_j^+\omega_{t_j}^n=\int_X\hat{v}_j^-\omega_{t_j}^n$. By \eqref{eq:sup vj} and \cref{estimates of varphi} we have
$$
\int_{U_{2\varepsilon}}\hat{v}_j^{-}\omega_{t_j}^n\leq C|U_{2\varepsilon}|\leq 2C\varepsilon.
$$
Since $c\geq0$, $\hat{v}_j^{-}:=-\min(\hat{v}_j,0)\to0$ on $X\setminus U_{2\varepsilon}$. The dominated convergence theorem then yields that
$$
\underset{j\to\infty}{\lim}\int_{X\setminus U_{2\varepsilon}}\hat{v}_j^{-}\omega_{t_j}^n=0.
$$
We now choose $\varepsilon$ so small such that $2C\varepsilon<\frac{1}{4}$, then for $j$ large
$$
1=\int_X|\hat{v}_j|\omega_{t_j}^n=\int_X(\hat{v}_j^-+\hat{v}_j^+)\omega_{t_j}^n=2\int_X\hat{v}_j^-\omega_{t_j}^n=2\int_{U_{2\varepsilon}}\hat{v}_j^-\omega_{t_j}^n+2\int_{{X\setminus U_{2\varepsilon}}}\hat{v}_j^-\omega_{t_j}^n<\frac{1}{2},
$$
this gives a contradiction.

\textbf{Case 2.} For the case $c<0$, the argument is exactly the same as in \textbf{Case 1} by considering $\hat{v}_j^+$, the proof is therefore concluded.
\end{proof}

Unlike the K\"ahler case, for any Hermitian metric $\omega$ on $X$, the formal adjoint $\Delta_{\omega}^*$ does not necessarily coincide with $\Delta_{\omega}$ itself on a Hermitian manifold and we also do not have 
$$
\int_X(\Delta_{\omega}u)\omega^n=0.
$$
 In this setting, basic elliptic PDE theory (see e.g. \cite[Chapter VI, Corollary (2.4)]{Dem}) yields an orthonormal decomposition in $L^2(X,\omega^n)$:
$$
C^\infty(X)=\Delta_{\omega}(C^\infty(X))\oplus\operatorname{ker}\Delta_{\omega}^*.
$$
Now, given any $h\in C^\infty(X)$, $h$ may not lie in the image of $\Delta_{\omega}$, and it follows from the above decomposition that $h\in\Delta_{\omega}(C^\infty(X))$ if and only if $h\perp\operatorname{ker}\Delta_{\omega}^*$. 

On the other hand, the standard maximum principle yields that harmonic functions are constant, namely, $\dim_{\mathbb{C}}\operatorname{Ker}\Delta_{\omega}=1$. Since the Riemannian Laplacian $\Delta_{g}$ is self-adjoint and it has the same second-order term with $\Delta_{\omega}$, the famous index theorem yields that the index $\dim_{\mathbb{C}}\operatorname{Ker}\Delta_{\omega}-\dim_{\mathbb{C}}\operatorname{Ker}\Delta_{\omega}^*=0$. Consequently, $\dim_{\mathbb{C}}\operatorname{Ker}\Delta_{\omega}^*=1$. Now, it is well-known from Gauduchon's theorem (see \cite{Gau77}) that we can choose a positive generator $0<\rho\in\operatorname{Ker}\Delta_{\omega}^*$ (namely, $\partial\bar{\partial}(\rho\omega^{n-1})=0$) such that $\int_X\rho\omega^n=1$. For an arbitrary $h\in C^\infty(X)$, choose $c=-\int_Xh\rho\omega^n$, then
$$
\int_X(h+c)\rho\omega^n=\int_Xh\rho\omega^n+c=0.
$$
This implies that $h+c\in\Delta_{\omega}(C^\infty(X))$. We thus arrive at the following:

\begin{lemma}
\label{solve Laplacian}
    Let $(X,\omega)$ be a compact Hermitian manifold, for each $h\in C^\infty(X)$ there exists a constant $c$ and a function $u\in C^\infty(X)$ such that
    $$
\Delta_\omega u=h+c.
    $$
\end{lemma}
The next two lemmas give some uniform estimates of the twist constant $c$ along the Chern-Ricci flow:
\begin{lemma}\label{twist constant 1}
    Let $h\in C^\infty(X)$ be such that $|h|\leq2$. Let $c$ be a constant and $v\in C^\infty(X)$ be the solution of the Laplacian equation
    $$
\Delta_{\omega_t}v=h+c,\quad\int_Xv\omega_t^n=0.
    $$
    Then, we have
    $$
|c|\leq\sup_X|h|\leq2.
    $$
\end{lemma}
\begin{proof}
    Fix a smooth positive function $\rho_t\in\operatorname{Ker}(\Delta_{\omega_t}^*)$ such that $\int_X\rho_t\omega_{t}^n=1$, it follows from the above discussion that 
    $$
|c|=\left|\int_X\rho_t h\omega_t^n\right|\leq\sup_X|h|\int_X\rho_t\omega_t^n=\sup_X|h|\leq2.
    $$
\end{proof}

We will also need the following more general version:

\begin{lemma}\label{twist constant 2}    
Let $U$ be a neighborhood of $E=\operatorname{Null}(K_X)$ where $d\omega_0=0$ and $U_1$ be a smaller neighborhood of $E$ such that $E\subset U_1\Subset U$. Let $h\in C^\infty(X)$ be such that $\int_X|h|\omega_t^n\leq b$ for some uniform (in $t$) constant $b$. Let $c$ be a constant and $v\in C^\infty(X)$ be the solution of the Laplacian equation    
$$
\Delta_{\omega_t}v=h+c,\quad\int_Xv\omega_t^n=0.    
$$    
Then, there exists a uniform constant $C=C(X,U,U_1,b,\chi,\omega_0,n)$ such that
$$
|c|\leq C.    
$$
\end{lemma}
\begin{proof}
    Since we have assumed that $\omega_0$ is K\"ahler in $U$, so does $\omega_t$ for all $t>0$. It follows that $\Delta_{\omega_t}^*=\Delta_{\omega_t}$ in $U$. Let $0<\rho_t\in\operatorname{Ker}\Delta_{\omega_t}^*$ be such that $\int_X\rho_t\omega_t^n=1$. On $U$, we have $\Delta_{\omega_t}\rho_t=0$. The elliptic maximum principle yields that $\sup_U\rho_t\leq\sup_{\partial U}\rho_t$ and hence $\sup_X\rho_t=\sup_{X\setminus U}\rho_t$. Choosing a finite cover of the compact set $X\setminus U$ by small balls contained in $X\setminus U_1$ and invoking the uniform equivalence of $\omega_t$ outside $U_1$, we deduce from \cref{GT 9.20} (take $L=\Delta_{\omega_t}^*$ here) that
    $$
\sup_X\rho_t=\sup_{X\setminus U}\rho_t\leq C,
    $$
    for a uniform $C=C(U,U_1,X,\chi,\omega_0,n)$. Consequently, we can further write
    \begin{align*}
        |c|&=\left|\int_X\rho_t h\omega_t^n\right|
        \leq\sup_X\rho_t\int_X|h|\omega_t^n\leq C\cdot b.
    \end{align*}
    The proof is therefore concluded.
\end{proof}

\section{Estimates of Green's functions}
In this section we shall build on various uniform estimates of Green's functions along the Chern-Ricci flow. First recall the definition of Green functions: let $(X,\omega)$ be a compact Hermitian manifold and let $g$ be the Riemannian metric associated to $\omega$, the Green's function $G(x,y)$ is defined to be the unique function satisfying 
\begin{equation}\label{eq:def of green}
\Delta_gG(x,\cdot)dV_g=-\delta_{x}+\frac{1}{\operatorname{Vol}_g(X)}dV_g,\quad\int_XG(x,y)dV_g(y)=0.
\end{equation}
where $\Delta_g$ is the Riemannian Laplacian of $g$. On a Hermitian manifold, we know $dV_g=\frac{\omega^n}{n!}$ and $\operatorname{Vol}_g(X)=\frac{1}{n!}V_\omega:=\frac{1}{n!}\int_X\omega^n$, so $\frac{1}{\operatorname{Vol}_g(X)}dV_g=\frac{1}{V_\omega}\omega^n$ and hence \eqref{eq:def of green} can be rewritten as 
$$
\Delta_gG(x,\cdot)\frac{\omega^n}{n!}=-\delta_x+\frac{\omega^n}{V_\omega}.
$$
Set now $\tilde{G}:=\frac{1}{n!}G$, we finally have
\begin{equation}\label{complex green formula}
\Delta_g\tilde{G}(x,\cdot)\omega^n=-\delta_x+\frac{\omega^n}{V_\omega}.
\end{equation}
 In this language, the Green's formula can be written as
\begin{align*}
    u(x)&=\frac{1}{V_g}\int_Xu(y)dV_g(y)-\int_XG(x,y)\Delta_gv(y)dV_g(y)\\
    &=\frac{1}{V_\omega}\int_Xu(y)\omega^n(y)-\int_X\tilde{G}(x,y)\Delta_gv(y)\omega^n(y),
\end{align*}
for any $v\in C^2(X)$. By abuse of notations, we still denote $G$ instead of $\tilde{G}$ in \eqref{complex green formula} in the sequel.

Throughout the subsequent lemmas, we shall fix a family of relatively compact neighborhoods of $E$ such that $E\subset U_3\Subset U_2\Subset U_1\Subset U$, where we have assumed that $\omega_0$ is K\"ahler in $U$. Since $X\setminus E$ is connected, we can also arrange that $X\setminus U_i$ are all connected for $i=1,2,3$.
\begin{lemma}\label{L^1 bound of G}
    Let $G_t$ be the Green's function associated with $\omega_t$, as defined in \eqref{complex green formula}. Then along the Chern-Ricci flow, there exists a uniform constant $C=C(X,U,U_1,\lambda[\omega_t|_{X\setminus U_1}],n,p,\chi,\allowbreak \omega_X,\operatorname{Ent}_p(\omega_t))$ such that for any $x\in X$, we have
    $$
\int_X|G_t(x,y)|\omega_t^n\leq C. 
    $$
    Where $\lambda[\omega_t|_{X\setminus U_1}]$ denotes the eigenvalue vectors of $\omega_t$ on $X\setminus U_1$.
\end{lemma}
\begin{proof}
    Fix $x\in X$, take a sequence of functions $h_k\in C^\infty(X)$ such that $h_k\to-\mathds{1}_{\{G_t(x,\cdot)\geq0\}}$ in $L^q(X,\omega_t^n)$ for some large $q>1$. We can assume without loss of generality that $\sup_X|h_k|\leq2$. It follows from \cref{solve Laplacian} and \cref{twist constant 1} that we can find $v_k\in C^\infty(X)$ and constants $c_k$ with $|c_k|\leq2$ such that
    $$
\Delta_{\omega_t}v_k=h_k+c_k,\quad\int_Xv_k\omega_t^n=0.
    $$
   It follows from \cref{Laplacian estimate 2} that $v_k$ is uniformly bounded. Now, Green's formula yields that
\begin{equation}\label{eq:green of vk}
\begin{aligned}
    v_k(x)&=\int_XG_t(x,y)(-\Delta_{g_t}v_k(y))\omega_t^n(y)\\
    &=\int_XG_t(x,y)(-2\Delta_{\omega_t}v_k(y))\omega_t^n(y)-2\int_XG_t(x,y)\langle\nabla v_k(y),\theta_t^\#(y)\rangle_{\omega_t(y)}\omega_t^n(y)\\
    &=\int_XG_t(x,y)(-2h_k)\omega_t^n(y)-2\int_XG_t(x,y)\langle\nabla v_k(y),\theta_t^\#(y)\rangle_{\omega_t(y)}\omega_t^n(y),
\end{aligned}
\end{equation}
where $\theta_t^\#$ is the vector field associated to the torsion $1$-form of $\omega_t$, which vanishes on $U$. Invoking $h_k\to -\mathds{1}_{\{G_t(x,\cdot)\geq0\}}$, the first term in the right-hand side of \eqref{eq:green of vk} tends to $2\int_{\{G_t(x,\cdot)\geq0\}}G_t(x,y)\omega_t^n(y)=\int_X|G_t(x,y)|\omega_t^n(y)$. For the second term, we can write
\begin{equation}\label{term 2}
    \int_XG_t(x,y)\langle\nabla v_k(y),\theta_t^\#(y)\rangle_{\omega_t(y)}\omega_t^n(y)\leq\int_X|G_t(x,y)||\nabla v_k(y)|_{\omega_t}|\theta_t^\#(y)|_{\omega_t}\omega_t^n(y).
\end{equation}
We claim that $|\nabla v_k|_{\omega_t}$ is uniformly bounded outside $U$. Choose a relatively compact subdomain $U_1\Subset U$, since $\omega_t$ is uniformly equivalent outside $U_1$, \cref{GT 9.11} yields that there is a uniform constant $C_1=C_1(X,U,U_1,\lambda[\omega_t|_{X\setminus U_1}],p)$ such that  
$$
\|v_k\|_{W^{2,p}(X\setminus U)}\leq C_1\left(\|v_k\|_{L^p(X\setminus U_1)}+\|h_k+c_k\|_{L^p(X\setminus U_1)}\right)\leq C_1C:=C_2,
$$
where the second inequality is due to the uniform boundedness of $v_k,h_k$ and $c_k$. Choose $p>2n$, Sobolev's embedding theorem then yields that there exists a positive constant $\alpha>0$ such that the embedding $W^{2,p}\hookrightarrow C^{1,\alpha}$ is a compact operator. Consequently, we derive the uniform bound 
$$
\|\nabla v_k\|_{C^0(X\setminus U)}\leq C_3,
$$
for some uniform constant $C_3$. Now, since $\theta_t^\#$ vanishes on $U$ and $|\theta_t^\#|_{\omega_t}=o(e^{-t})$, we continue with \eqref{term 2} to write
\begin{align*}
     \int_XG_t(x,y)\langle\nabla v_k(y),\theta_t^\#(y)\rangle_{\omega_t(y)}\omega_t^n(y)\leq C_4e^{-t}\int_X|G_t(x,y)|\omega_t^n\leq \frac{1}{4}\int_X|G_t(x,y)|\omega_t^n,
\end{align*}
for $t> \log C_4+\log 4$. Putting everything into \eqref{eq:green of vk}, we get
$$
\int_X|G_t(x,y)|\omega_t^n(y)\leq C+\frac{1}{2}\int_X|G_t(x,y)|\omega_t^n(y),
$$
the proof is therefore concluded.
\end{proof}

\begin{lemma}\label{lem:inf of green}
For each $x\in U_2$, there exists a uniform constant $C$ depending on $X, U, U_1, U_2, \allowbreak \lambda[\omega_t|_{X\setminus U_2}], n, \chi, \omega_0, \operatorname{Ent}_p(\omega_t)$ such that
    $$
\underset{y\in X}{\inf}G_t(x,y)\geq-C.
    $$
\end{lemma}
\begin{proof}
    Fix $x\in U_2$ and set $v(y):=-G_t(x,y)$, which is smooth on $X\setminus U_2$. On $\{v\geq0\}=\{G_t(x,\cdot)\leq0\}$ (where $v$ is also smooth and the region is away from $\{x\}$), we compute
    \begin{equation}\label{eq:inf Delta v}
\Delta_{\omega_t}v=\frac{1}{2}\Delta_{g_t}v-\langle\nabla v,\theta_t^\#\rangle=\frac{1}{2}\left(\delta_x-\frac{1}{V_{\omega_t}}\right)-\langle\nabla v,\theta_t^\#\rangle\geq-a-\langle\nabla v,\theta_t^\#\rangle,
    \end{equation}
    because $\delta_x\geq0$ and $V_{\omega_t}$ is uniformly bounded away from zero. To use \cref{laplacian estimate 1}, we have to show that $|\langle\nabla v,\theta_t^\#\rangle|$ is uniformly bounded from above on $\{v\geq0\}$. Since $\theta_t^\#$ vanishes on $U$, it suffices to show that $|\nabla v|_{\omega_t}=|\nabla G_t(x,\cdot)|_{\omega_t}$ is uniformly bounded on $X\setminus U$. Similar to that in \cref{L^1 bound of G}, the sobolev's embedding theorem yields that it is enough to give a uniform bound of $\|v\|_{W^{2,p}(X\setminus U,\omega_t^n)}$ for some $p>2n$. 

   Since the metrics $\omega_t$ are uniformly equivalent outside $U_1$, by \cref{GT 9.11} (take $L=\Delta_{g_t}$ here), there is a uniform constant $C=C(X,U,U_1,n,p,\omega_t|_{X\setminus U_1})$ such that
    \begin{equation}\label{eq: v W2,p}
    \begin{aligned}
\|v\|_{W^{2,p}(X\setminus U)}&\leq C_1\left(\|v\|_{L^p(X\setminus U_1)}+\|\Delta_{g_t}v\|_{L^p(X\setminus U_1)}\right)\\
&= C_1\left(\|v\|_{L^p(X\setminus U_1)}+\|\frac{1}{V_{\omega_t}}\|_{L^p(X\setminus U_1)}\right),
    \end{aligned}
    \end{equation}
    because $x\in U_2$ implies that $\delta_x=0$ on $X\setminus U_1$. It thus remains to establish a uniform bound of $v$ on $X\setminus U_1$. Observe that the family of metrics $\omega_t$ are uniformly equivalent on $X\setminus U_2$, the Harnack type inequality \cref{GT 9.20} (applied with $L=\Delta_{g_t}$) yields that
    \begin{align*}
        \sup_{X\setminus U_1}|v|\leq C_2\left(\|v\|_{L^1(X,\omega_t^n)}+\|V_{\omega_t}^{-1}\|_{L^{2n}(X,\omega_t^n)}\right)\leq C_3,
    \end{align*}
    where $C_2=C_2(X,U_1,U^{\prime\prime
    },\lambda[\omega_t|_{X\setminus U_2}],\chi,\omega_0,n)$ and\\ $C_3=C_3(X,U,U_1,U_2,\lambda[\omega_t|_{X\setminus U_2}],n,\chi,\omega_0,\operatorname{Ent}_p(\omega_t))$. We remark that we applied \cref{GT 9.20} by choosing a finite covering of the compact set $X\setminus U_1$ by small Euclidean balls of fixed radius (depending on the distance between $\partial U_1$ and $\partial U_2$) contained in $X\setminus U_2$ and using \cref{GT 9.20} on all the balls. The last inequality is thus a consequence of \cref{L^1 bound of G}.

    Now we have shown that $|\langle\nabla v,\theta_t^\#\rangle|$ is uniformly bounded from above on $\{v\geq0\}$; hence, we conclude from \eqref{eq:inf Delta v} that 
    $$
\Delta_{\omega_t}v\geq-a-b,
    $$
   on $\{v\geq0\}$ for some uniform constants $a,b$. It thus follows from \cref{laplacian estimate 1} that 
   \begin{align*}
       -\underset{y\in X}{\inf}G_t(x,y)=\sup_Xv(y)\leq C(a+b+\|v\|_{L^1(X,\omega_t^n)})\leq C,
   \end{align*}
   where the last inequality is due to \cref{L^1 bound of G}. The proof is now finished by reversing the signs on both sides.
\end{proof}

\begin{remark}\label{rmk:uni bound of G_t}
    From the proof of \cref{lem:inf of green}, we can extract a useful fact: for each pair of neighborhoods $E\subset U_1\Subset U_2$ and each $x\in U_1$, the Green's function $G_t(x,\cdot)$ is uniformly bounded on $X\setminus U_2$ by a uniform constant $C$ depending on $U_1$ and $U_2$ and the given uniform data.
\end{remark}

As in \cite{GPSS24a}, for each $x\in U_2$, we set 
$$
\mathcal{G}_t(x,\cdot)=G_t(x,\cdot)-\underset{y\in X}{\inf}G_t(x,y)+V_{\omega_t}^{-1}>0.
$$
Integrating both sides, we get from \cref{lem:inf of green} that $\int_X\mathcal{G}_t(x,\cdot)\omega_t^n\leq C$ by some constant $C$ depending on $U_2$ and the data in \cref{lem:inf of green}.
\begin{lemma}\label{L^p bound of green functions}
There exist constants $C=C(X,U,U_1,U_2,\lambda[\omega_t|_{X\setminus U_2}],n,p,\chi,\omega_X,\operatorname{Ent}_p(\omega_t))$ and $\varepsilon=\varepsilon(n,p)$ such that for any $x\in U_2$, we have
$$
\int_X\mathcal{G}_t(x,\cdot)^{1+\varepsilon}\omega_t^n\leq C.
$$
\end{lemma}
\begin{proof}
    The idea is inspired by \cite[Lemma 5.5]{GPSS24a}, in our setting we have to overcome the difficulty brought by the gradient of various solutions. In the sequel, we will use $C$ to denote various uniform constants depending on the given data.
    
    As in \cite{GPSS24a}, we fix $x\in U_2$ and choose smooth positive functions $H_k$, which is a smoothing of $\min\{\mathcal{G}_t(x,\cdot),k\}$, such that $H_{t,k}$ converges increasingly to $\mathcal{G}_t(x,\cdot)$. Therefore, we obtain
    \begin{equation}\label{L^1 of H_k}
0<\int_XH_{t,k}\omega_t^n\leq\int_X\mathcal{G}_t(x,\cdot)\omega_t^n\leq C.
    \end{equation}
    Use \cref{solve Laplacian}, we can find $u_k\in C^\infty(X)$ and constants $c_k$ solving 
    \begin{equation}
        \begin{cases}
            \Delta_{\omega_t}u_{t,k}=-H_{t,k}^\varepsilon+c^\prime_{t,k},\\
            \int_Xu_{t,k}\omega_t^n=0.
        \end{cases}
    \end{equation}
    Since $\varepsilon<1$, $H_{t,k}^{\varepsilon}$ is uniformly bounded in $L^1(X,\omega_t^n)$ by \eqref{L^1 of H_k}. It thus follows from \cref{twist constant 2} that 
    $$
|c^\prime_{t,k}|\leq C
    $$
    is uniformly bounded. Now that $H_{t,k}$ is not uniformly bounded, to show the uniformly boundedness of $u_k$, we use the main theorem of \cite{TW10} to solve the following auxiliary complex Monge-Amp\`ere equations
    \begin{equation}\label{auxiliary MA}
        (\hat{\omega}_t+\sqrt{-1}\partial\bar{\partial}\psi_{t,k})=c_{t,k}\frac{H_{t,k}^{n\varepsilon}+1}{\int_X(H_{t,k}^{n\varepsilon}+1)\omega_t^n}\omega_t^n=c_{t,k}\frac{H_{t,k}^{n\varepsilon}+1}{B_k}e^{F_t}\omega_X^n,
    \end{equation}
    with
    $$
c_{t,k}>0,\quad\sup_X\psi_{t,k}=0,\quad B_k=\int_X(H_{t,k}^{n\varepsilon }+1)\omega_t^n,\quad e^{F_t}=\frac{\omega_t^n}{\omega_X^n}.
    $$
    Integrating both sides of \eqref{auxiliary MA}, we obtain that
    $$
0<\underline{\operatorname{Vol}}(\chi)\leq\underline{\operatorname{Vol}}(\hat{\omega}_t)\leq c_{t,k}\leq\overline{\operatorname{Vol}}(\hat{\omega}_t).
    $$
    Recall that $\hat{\omega}_t:=\chi+e^{-t}(\omega_0-\chi)$, the bounded mass property and the positive volume property yields that $c_{t,k}$ is uniformly bounded away from zero (see the proof in \cref{laplacian estimate 1}). To use the $L^\infty$-estimate \cref{a priori in nef class for MA}, we have to check that the quantity 
    $$
\int_X\frac{H_{t,k}^{n\varepsilon}+1}{B_k}e^{F_t}\log\left(1+\frac{H_{t,k}^{n\varepsilon}+1}{B_k}e^{F_t}\right)^p\omega_X^n
    $$ 
    is uniformly bounded from above. Here, we remark that there's no need to check the uniform lower bound of $\int_X\operatorname{exp}\left(\frac{(H_{t,k}^{n\varepsilon}+1)e^{F_t}}{nB_k}\right)\omega_X^n$ since this quantity was just used to obtain the uniform upper bound of $c_{t,k}$ in the proof of \cite[Theorem 15.4]{PSWZ25}, which is automatically satisfied here thanks to the bounded mass property of $X$.

    Now, we use the basic inequality $\log(1+x)\leq x$ for $x>0$ to write
    \begin{equation}\label{entropy bound}
\begin{aligned}
    &\int_X\frac{H_{t,k}^{n\varepsilon}+1}{B_k}e^{F_t}\log\left(1+\frac{H_{t,k}^{n\varepsilon}+1}{B_k}e^{F_t}\right)^p\omega_X^n\\
    \leq&p\int_X\frac{H_{t,k}^{n\varepsilon}+1}{B_k}e^{F_t}\frac{H_{t,k}^{n\varepsilon}+1}{B_k}e^{F_t}\omega_X^n\\
    =&\frac{p}{B_k^2}\int_X(H_{t,k}^{2n\varepsilon}+2H_{t,k}^{n\varepsilon}+1)e^{2F_t}\omega_X^n.
\end{aligned}
\end{equation}
To show the uniform boundedness of the right-hand side of \eqref{entropy bound}, observe that
$$
\int_X\omega_t^n<B_k=\int_XH_{t,k}^{n\varepsilon}\omega_t^n+\int_X\omega_t^n\leq\left(\int_XH_{t,k}\omega_t^n\right)^{n\varepsilon}V_{\omega_t}^{1-\frac{1}{n\varepsilon}}+V_{\omega_t}\leq C\overline{\operatorname{Vol}}(\hat{\omega}_t)\leq C,
$$
if we choose $\varepsilon<\frac{1}{n}$ and invoking \eqref{L^1 of H_k}. Hence, $B_k$ is uniformly bounded away from zero. Similarly, by choosing $\varepsilon<\frac{1}{2n}$ and using H\"older's inequality, we can easily show the uniform boundedness of the integral
$$
\int_X(H_{t,k}^{2n\varepsilon}+2H_{t,k}^{n\varepsilon}+1)\omega_X^n.
$$
Finally, note that 
$$
e^{F_t}=\frac{\omega_t^n}{\omega_X^n}=e^{\varphi_t+\partial_t\varphi_t}\frac{\Omega}{\omega_X^n}\leq C,
$$
thanks to \cref{estimates of varphi}. Putting everything into \eqref{entropy bound} we get the uniform bound of the quantity $\int_X\frac{H_{t,k}^{n\varepsilon}+1}{B_k}e^{F_t}\log\left(1+\frac{H_{t,k}^{n\varepsilon}+1}{B_k}e^{F_t}\right)^p\omega_X^n$.

It follows from \cref{a priori in nef class for MA} that $|\psi_{t,k}-\mathcal{V}_{\hat{\omega}_t}|\leq C$ for a uniform constant $C$ depending on the given data. By \cref{lower bound of varphi(t)}, we also have $|\varphi_t-\mathcal{V}_{\hat{\omega}_t}|\leq C$; hence, we finally derive
\begin{equation}\label{L infty of psi - varphi}
    \sup_X|\psi_{t,k}-\varphi_t|\leq C.
\end{equation}
As in \cite{GPSS24a}, we consider the auxiliary function
\begin{equation}\label{def of v_{t,k}}
    v_{t,k}:=(\psi_{t,k}-\varphi_t)-\frac{1}{V_{\omega_t}}\int_X(\psi_{t,k}-\varphi_t)\omega_t^n+\varepsilon_1u_{t,k},
\end{equation}
where $\varepsilon_1>0$ is a small positive constant to be assigned. By definition $v_{t,k}\in C^\infty(X)$ and $\int_Xv_{t,k}\omega_t^n=0$. Now we compute
\begin{align*}
    \Delta_{\omega_t}v_{t,k}=&\operatorname{tr}_{\omega_t}(\hat{\omega}_t+\sqrt{-1}\partial\bar{\partial}\psi_{t,k})-n+\varepsilon_1\Delta_{\omega_t}u_{t,k}\\
    \geq&n\left(\frac{c_{t,k}\frac{H_{t,k}^{n\varepsilon}+1}{B_k}\omega_t^n}{\omega_t^n}\right)^{\frac{1}{n}}-n-\varepsilon_1H_{t,k}^\varepsilon+\varepsilon_1c_{t,k}^\prime\\
    =&n\left(\frac{c_{t,k}}{B_k}\right)^{\frac{1}{n}}(H_{t,k}^{n\varepsilon}+1)^{\frac{1}{n}}-n-\varepsilon_1H_{t,k}^\varepsilon+\varepsilon_1c_{t,k}^\prime.
\end{align*}
Thanks to the uniform boundedness (away from zero) of $c_{t,k}$, we can choose a uniform constant $\varepsilon_1<n\left(\frac{c_{t,k}}{B_k}\right)^{\frac{1}{n}}$, it then follows from the uniform boundedness of $c_{t,k}^\prime$ that 
\begin{equation}\label{Delta v_{t,k}}
    \Delta_{\omega_t}v_{t,k}\geq-n+\varepsilon_1c_{t,k}^\prime\geq-C.
\end{equation}
The Green's formula then yields that
\begin{equation}\label{v green}
    \begin{aligned}
        v_{t,k}(x)=&\int_X\mathcal{G}_t(x,y)(-\Delta_{g_t}v_{t,k})\omega_t^n\\
        =&\int_X\mathcal{G}_t(x,y)(-2\Delta_{\omega_t}v_{t,k})\omega_t^n-2\int_X\mathcal{G}_t(x,y)\langle\nabla v_{t,k},\theta_t^\#\rangle_{\omega_t}\omega_t^n\\
        \leq&C+2\int_X|\mathcal{G}_t(x,y)||\nabla v_{t,k}|_{\omega_t}|\theta_t^\#|_{\omega_t}\omega_t^n,
    \end{aligned}
\end{equation}
where the last inequality follows from \eqref{Delta v_{t,k}} and the fact that $\mathcal{G}_t(x,\cdot)$ is uniformly bounded in $L^1(X,\omega_t^n)$. 

Next, we claim that 
$$
\int_X|\nabla v_{t,k}||\theta_t^\#|\omega_t^n\leq\int_X|\nabla \psi_{t,k}||\theta_t^\#|\omega_t^n+\int_X|\nabla \varphi_t||\theta_t^\#|\omega_t^n+\varepsilon_1\int_X|\nabla u_{t,k}||\theta_t^\#|\omega_t^n
$$
is uniformly bounded. Here, all the norms are computed with respect to $\omega_t$, so we omit the subscripts for simplicity of notations.

\textbf{Claim 1.} $|\nabla\varphi_t|$ is uniformly bounded on $\operatorname{supp}\theta_t^\#$. This is clear since $\theta_t^\#$ vanishes on $U$ and the gradient estimate of $\varphi_t$ has already been established by Gill outside $U$.

\textbf{Claim 2.} $\int_X|\nabla\psi_{t,k}||\theta_t^\#|\omega_t^n$ is uniformly bounded. To prove this, select a cutoff function $\rho$ such that $0\leq\rho\leq1$, $\rho\equiv1$ on $\operatorname{supp}\theta_t$ and $\operatorname{Supp}\rho\subset X\setminus U_1$. Then we compute
\begin{equation}\label{L2 gradient estimate of psi}
\begin{aligned}
\int_X\rho^2|\nabla\psi_{t,k}|^2\omega_t^n=&n\int_X\rho^2\sqrt{-1}\partial\psi_{t,k}\wedge\bar{\partial}\psi_{t,k}\wedge\omega_t^{n-1}\\
    =&n\int_X(-\psi_{t,k})2\rho\sqrt{-1}\partial\rho\wedge\bar{\partial}\psi_{t,k}\wedge\omega_t^{n-1}+n\int_X(-\psi_{t,k})\rho^2\sqrt{-1}\partial\bar{\partial}\psi_{t,k}\wedge\omega_t^{n-1}\\
    &+n\int_X\rho^2(-\psi_{t,k})\tau_t\wedge\bar{\partial}\psi_{t,k}\wedge\omega_t^{n-1}\\
\leq&C\int_X\rho|\nabla\rho||\nabla\psi_{t,k}|\omega_t^n+C\int_X\rho^2(\hat{\omega}_t+\sqrt{-1}\partial\bar{\partial}\psi_{t,k})\wedge\omega_t^{n-1}\\
&+C\int_X\rho^2\hat{\omega}_t\wedge\omega_t^{n-1}+C\int_X\rho^2|\nabla\psi_{t,k}||\tau_t|\omega_t^n.
\end{aligned}
\end{equation}
For the first term on the right-hand side of \eqref{L2 gradient estimate of psi}, we can write
\begin{align*}
    C\int_X\rho|\nabla\rho||\nabla\psi_{t,k}|\omega_t^n&\leq C\int_X\left(\frac{1}{4C}\rho^2|\nabla\psi_{t,k}|^2+4C|\nabla\rho|^2\right)\omega_t^n\\
 &   \leq\frac{1}{4}\int_X\rho^2|\nabla\psi_{t,k}|^2\omega_t^n+C\int_X|\nabla\rho|^2\omega_t^n.
    \end{align*}
    For the second and the third terms, observe that
    \begin{align*}
\int_X(\hat{\omega}_t+\sqrt{-1}\partial\bar{\partial}\psi_{t,k})\wedge\omega_t^{n-1}&=\int_X(\hat{\omega}_t+\sqrt{-1}\partial\bar{\partial}\psi_{t,k})\wedge(\hat{\omega}_t+\sqrt{-1}\partial\bar{\partial}\varphi_t)^{n-1}\\
&\leq\int_X(2\hat{\omega}_t+\sqrt{-1}\partial\bar{\partial}(\varphi_t+\psi_{t,k}))^n\\
&\leq\overline{\operatorname{Vol}}(2\hat{\omega}_t)\leq C,
    \end{align*}
    thanks to the bounded mass property of the Fujiki manifold $X$.

    The fourth term in \eqref{L2 gradient estimate of psi} is handled similarly as the first term:
    \begin{align*}
       C\int_X\rho^2|\nabla\psi_{t,k}||\tau_t|\omega_t^n\leq\frac{1}{4}\int_X\rho^2|\nabla\psi_{t,k}|^2\omega_t^n+C\int_X\rho^2|\tau_t|^2\omega_t^n\leq C+\frac{1}{4}\int_X\rho^2|\nabla\psi_{t,k}|^2\omega_t^n.
    \end{align*}
    Putting everything into \eqref{L2 gradient estimate of psi}, we obtain
    $$
\int_X\rho^2|\nabla\psi_{t,k}|^2\omega_t^n\leq\frac{1}{2}\int_X\rho^2|\nabla\psi_{t,k}|^2\omega_t^n+C.
    $$
 Now since $\theta_t$ vanishes on $U$ and $|\theta_t^\#|=o(e^{-t})$, the claim is an easy consequence of the Cauchy-Schwarz inequality.
    
\textbf{Claim 3.} $\int_X|\nabla u_{t,k}||\theta_t^\#|\omega_t^n$ is uniformly bounded. Select a cutoff function $\rho$ as above, we compute as in Claim 2:
\begin{equation}\label{L2 estimate of gradient of u}
    \begin{aligned}
        \int_X\rho^2|\nabla u_{t,k}|^2\omega_t^n   =&n\int_X(-u_{t,k})2\rho\sqrt{-1}\partial\rho\wedge\bar{\partial}u_{t,k}\wedge\omega_t^{n-1}+n\int_X(-u_{t,k})\rho^2\sqrt{-1}\partial\bar{\partial}u_{t,k}\wedge\omega_t^{n-1}\\
           &+n\int_X\rho^2(-u_{t,k})\tau_t\wedge\bar{\partial}u_{t,k}\wedge\omega_t^{n-1}\\
\leq&C\int_X\rho|\nabla\rho||u_{t,k}||\nabla u_{t,k}|\omega_t^n+n\int_X\rho^2(-u_{t,k})(\Delta_{\omega_t}u_{t,k})\omega_t^{n}\\
&+C\int_X\rho^2|u_{t,k}||\nabla u_{t,k}||\tau_t|\omega_t^n.
    \end{aligned}
\end{equation}
Similar to the arguments in Claim 2, the first and the third term on the right-hand side of \eqref{L2 estimate of gradient of u} can be controlled by
\begin{align*}
    C\int_X\rho|\nabla\rho||u_{t,k}||\nabla u_{t,k}|\omega_t^n\leq\frac{1}{4}\int_X\rho^2|\nabla u_{t,k}|^2\omega_t^n+C\int_X|\nabla\rho|^2|u_{t,k}|^2\omega_t^n, 
\end{align*}
and
\begin{align*}
    C\int_X\rho^2|u_{t,k}||\nabla u_{t,k}||\tau_t|\omega_t^n\leq\frac{1}{4}\int_X\rho^2|\nabla u_{t,k}|^2\omega_t^n+C\int_X\rho^2|u_{t,k}|^2|\tau_t|^2\omega_t^n.
\end{align*}
While for the second term, we have
\begin{align*}
 \int_X\rho^2(-u_{t,k})(\Delta_{\omega_t}u_{t,k})\omega_t^{n}=\int_X\rho^2(-u_{t,k})(-H_{t,k}^{\varepsilon}+c_{t,k}^\prime)\omega_t^n\leq C\int_X\rho^2|u_{t,k}|\omega_t^n,
\end{align*}
where the last inequality follows since $c^\prime_{t,k}$ is uniformly bounded and $H_{t,k}^\varepsilon$ is uniformly bounded on $\operatorname{Supp}\rho\subset X\setminus U_1$ thanks to \cref{rmk:uni bound of G_t}.

Putting everything into \eqref{L2 estimate of gradient of u}, we get
\begin{equation}\label{final for nabla u}
    \int_X\rho^2|\nabla u_{t,k}|^2\omega_t^n\leq C\int_{X\setminus U_1}|u_{t,k}|^2\omega_t^n.
\end{equation}
Now, choosing an open neighborhood $V$ of $E$ such that $U_2\Subset V\Subset U_1$ and using \eqref{rmk:uni bound of G_t} again, we deduce from \cref{GT 9.20} that
$$
\sup_{X\setminus U_1}|u_{t,k}|\leq C\|u_{t,k}\|_{L^1(X\setminus V,\omega_t^n)}+C.
$$
This combined with \eqref{final for nabla u} yields that
\begin{equation}\label{control nabla u by L1 of u}
     \int_X\rho^2|\nabla u_{t,k}|^2\omega_t^n\leq C\left(\|u_{t,k}\|_{L^1(X\setminus V,\omega_t^n)}+\|u_{t,k}\|_{L^1(X\setminus V,\omega_t^n)}^2\right)+C.
\end{equation}
On the other hand, we can invoke Green's formula to write
\begin{align*}
    u_{t,k}(x)=\int_XG_t(x,y)(-2\Delta_{\omega_t}u_{t,k}(y))\omega_t^n(y)-2\int_XG_t(x,y)\langle\nabla u_{t,k}(y),\theta_t^\#(y)\rangle\omega_t^n.
\end{align*}
Taking absolute values and integrating both sides, we have
\begin{equation}\label{L1 of u_k}
\begin{aligned}
    \int_X|u_{t,k}(x)|\omega_t^n(x)\leq&\int_X\left(\int_X|G_t(x,y)||-H_{t,k}^\varepsilon(y)+c_{t,k}^\prime|\omega_t^n(y)\right)\omega_t^n(x)\\
    &+\int_X\left(\int_X|G_t(x,y)||\nabla u_{t,k}(y)||\theta_t^\#(y)|\omega_t^n(y)\right)\omega_t^n(x)\\
    \leq&\int_X|-H_{t,k}^\varepsilon(y)+c_{t,k}^\prime|\left(\int_X|G_t(x,y)|\omega_t^n(x)\right)\omega_t^n(y)\\
    &+\int_X|\nabla u_{t,k}(y)||\theta_t^\#(y)|\left(\int_X|G_t(x,y)|\omega_t^n(x)\right)\omega_t^n(y)\\
    \leq&C+C\int_X|\nabla u_{t,k}||\theta_t^\#|\omega_t^n\\
    \leq&C+Ce^{-t}\int_{\operatorname{Supp}\theta_t}|\nabla u_{t,k}|\omega_t^n ,
\end{aligned}
\end{equation}
where in the second inequality we have used Fubini's theorem and in the third inequality we have used \cref{L^1 bound of G}, \eqref{L^1 of H_k} and the uniform boundedness of $c_{t,k}^\prime$.

Combining \eqref{L1 of u_k} and \eqref{control nabla u by L1 of u}, and noting that $\rho|_{\operatorname{Supp}\theta_t}\equiv1$, we obtain a chain of inequalities
\begin{align*}
     \int_X|u_{t,k}(x)|\omega_t^n(x)\leq& C+Ce^{-t}\int_{\operatorname{Supp}\theta_t}|\nabla u_{t,k}|\omega_t^n\leq C+Ce^{-t}\int_{X}\rho^2|\nabla u_{t,k}|\omega_t^n\\
     \leq&C+Ce^{-t}\left(\int_X\rho^2|\nabla u_{t,k}|^2\omega_t^n\right)^{\frac{1}{2}}\left(\int_X\rho^2\omega_t^n\right)^{\frac{1}{2}}\\
     \leq&C+Ce^{-t}\left(\int_X\rho^2|\nabla u_{t,k}|^2\omega_t^n\right)^{\frac{1}{2}}\\
     \leq&C+Ce^{-t}\left(C\|u_{t,k}\|_{L^1(X\setminus V,\omega_t^n)}+C\|u_{t,k}\|_{L^1(X\setminus V,\omega_t^n)}^2+C\right)^{\frac{1}{2}}.
\end{align*}
This immediately yields the uniform boundedness of the integral $\int_X|u_{t,k}|\omega_t^n$ for $t$ large enough, whence the claim follows by \eqref{control nabla u by L1 of u} again.

The three claims above immediately yield the uniform boundedness of $\int_X|\nabla v_{t,k}||\theta_t^\#|\omega_t^n$. Return to \eqref{v green}, since $x\in U_2$, we can use \cref{lem:inf of green} and \cref{rmk:uni bound of G_t} to see that $\mathcal{G}_t(x,\cdot)$ is uniformly bounded on $\operatorname{Supp}\theta_t\subset X\setminus U$. This implies that
\begin{align*}
    v_{t,k}(x)\leq C+2\int_X|\mathcal{G}_t(x,y)||\nabla v_{t,k}|_{\omega_t}|\theta_t^\#|_{\omega_t}\omega_t^n\leq C.
\end{align*}
Therefore, by the definition of $v_{t,k}$ \eqref{def of v_{t,k}} and the estimate \eqref{L infty of psi - varphi}, we derive the uniform upper bound of $u_{t,k}$:
$$
\sup_Xu_{t,k}\leq C.
$$
At this stage, we apply the Green's formula for $u_{t,k}$
\begin{align*}
    u_{t,k}(x)=&\int_X\mathcal{G}_t(x,y)(-\Delta_{g_t}u_{t,k})\omega_t^n(y)\\        
=&\int_X\mathcal{G}_t(x,y)(2H_{t,k}^\varepsilon(y)-2c_{t,k}^\prime)\omega_t^n(y)-2\int_X\mathcal{G}_t(x,y)\langle\nabla u_{t,k},\theta_t^\#\rangle_{\omega_t}\omega_t^n(y).
\end{align*}
Rearranging the terms, we obtain:
\begin{equation}\label{finial Lp}
    2\int_X\mathcal{G}_t(x,\cdot)H_{t,k}^\varepsilon\omega_t^n\leq u_{t,k}(x)+2c_{t,k}^\prime\int_X\mathcal{G}_t(x,\cdot)\omega_t^n+2\int_X\mathcal{G}_t(x,\cdot)|\nabla u_{t,k}||\theta_t^\#|\omega_t^n\leq C.
\end{equation}
Here, all the terms are uniformly bounded from above by our estimates. Since $H_{t,k}^\varepsilon$ increases to $\mathcal{G}_t^\varepsilon$, taking the limit with respect to $k$ in \eqref{finial Lp}, we obtain the desired estimate and thereby conclude our proof.  

\end{proof}

Having the $L^{1+\varepsilon}$-estimate of the Green's function in hand, we can now follow the strategy of \cite{GPS24} and \cite{GPSS24a} to derive the gradient bounds of $\mathcal{G}_t(x,\cdot)$:

\begin{lemma}\label{gradient of G 1}
    For any $x\in U_2$ and any $\beta>0$, we have
$$
\sup_{x\in U_2}\int_X\frac{|\nabla_y\mathcal{G}_t(x,y)|^2_{\omega_t(y)}}{\mathcal{G}_t(x,y)^{1+\beta}}\omega_t^n(y)\leq C\cdot\beta^{-1}.
$$
Where $C$ is a uniform constant depending on $X,\omega_0,\chi,\beta$.
\end{lemma}
\begin{proof}
    Fix a point $x\in U_2$ and a number $\beta>0$. Set $u_t(y):=\mathcal{G}_t(x,y)^{-\beta}$. By the properties of $G_t$ we have that $u\in C^\infty(X\setminus\{x\})\cap C^0(X)$ with $u(x)=0$. By the definition of $\mathcal{ G}_t(x,\cdot)$, we have a uniform upper bound of $u_t$:
    \begin{equation}\label{upper bound of u}
        0\leq u_t(y)\leq V_{\omega_t}^{\beta}\leq C.
    \end{equation}
    Using Green's formula for the Riemannian Laplacian we can write
    \begin{equation}\label{integrate by parts green}
    \begin{aligned}
        0=u_t(x)&=\frac{1}{V_{\omega_t}}\int_Xu_t\omega_t^n+\int_X\mathcal{G}_t(x,\cdot)(-\Delta_{g_t}u_t)\omega_t^n\\
        &=\frac{1}{V_{\omega_t}}\int_Xu_t\omega_t^n-\beta\int_X\frac{|\nabla\mathcal{G}_t(x,\cdot)|^2_{g_t}}{\mathcal{G}_t(x,\cdot)^{1+\beta}}\omega_t^n,
    \end{aligned}
    \end{equation}
 where in the last equality we have used integration by parts. The result is then an easy consequence of \eqref{upper bound of u}. 
 
 Let us check the validity of the integration by parts in \eqref{integrate by parts green} in detail, the usage of Green's formula for $u_t$ can be checked similarly. Consider the $\delta$-geodesic ball $B(x,\delta)$ in $X$, where all the functions are smooth. The classical Green's formula yields that
 \begin{align*}
     \int_{X\setminus B(x,\delta)}\mathcal{G}_t(x,\cdot)(-\Delta_{g_t}u_t)\omega_t^n&=\int_{X\setminus B(x,\delta)}\langle\nabla\mathcal{G}_t(x,\cdot),\nabla u_t\rangle_{g_t}\omega_t^n-\int_{\partial B(x,\delta)}\mathcal{G}_t(x,\cdot)\frac{\partial u_t}{\partial\nu}d\sigma_t\\
     &=-\beta\int_{X\setminus B(x,\delta)}\frac{|\nabla\mathcal{G}_t(x,\cdot)|^2_{g_t}}{\mathcal{G}_t(x,\cdot)^{1+\beta}}\omega_t^n-\int_{\partial B(x,\delta)}\mathcal{G}_t(x,\cdot)\frac{\partial u_t}{\partial\nu}d\sigma_t,
 \end{align*}
 where $\nu$ is the outer normal direction of $B(x,\delta)$ and $d\sigma_t$ is the surface measure on $\partial B(x,\delta)$ induced by $\omega_t^n$. It remains to check that
 $$
\int_{\partial B(x,\delta)}\mathcal{G}_t(x,\cdot)\frac{\partial u_t}{\partial\nu}d\sigma_t\to0\quad\operatorname{as}\delta\to0.
 $$
 Since $n\geq2$, the asymptotic behavior of $\mathcal{G}_t(x,\cdot)$ near $x$ looks like
 $$
\mathcal{G}_t(x,\cdot)\sim r^{2-2n},\quad\operatorname{for} r=d(x,\cdot).
 $$
 It follows that $u\sim r^{2\beta(n-1)}$ and hence $\frac{\partial u}{\partial\nu}\sim|\nabla u|\sim r^{2\beta(n-1)-1}$, while the surface area $\int_{\partial B(x,\delta)}d\sigma_t\sim \delta^{2n-1}$. Taking into account all of these we see that
 $$
\int_{\partial B(x,\delta)}\mathcal{G}_t(x,\cdot)\frac{\partial u_t}{\partial\nu}d\sigma_t\sim \delta^{2-2n+2\beta(n-1)-1+2n-1}=\delta^{2\beta(n-1)}\to0\quad\operatorname{as}\delta\to0,
 $$
 where the last limit is because $n\geq2$ and $\beta>0$. 
\end{proof}

\begin{lemma}\label{gradient of G 2}
    Let $\varepsilon$ be the constant in \cref{L^p bound of green functions} and let $x\in U_2$ be a fixed point. Then, for any $s\in\left[1,\frac{2+2\varepsilon}{2+\varepsilon}\right)$, where $\varepsilon$ is the constant in \cref{L^p bound of green functions}, there is a uniform constant $C=C(s,X,U,U_1,U_2,\lambda[\omega_t|_{X\setminus U_2}],n,p,\chi,\omega_X,\operatorname{Ent}_p(\omega_t))$ such that
    $$
\int_X|\nabla G_t(x,\cdot)|_{g_t}^s\omega_t^n\leq C.
    $$
\end{lemma}
\begin{proof}
    Applying H\"older's inequality, we have
    \begin{align*}
\int_X|\nabla\mathcal{G}_t(x,\cdot)|_{g_t}^s\omega_t^n\leq\left(\int_X\frac{|\nabla\mathcal{G}_t(x,\cdot)|_{g_t}^2}{\mathcal{G}_t(x,\cdot)^{1+\beta}}\omega_t^n\right)^{\frac{s}{2}}\left(\int_X\mathcal{G}_t(x,\cdot)^{1+\varepsilon}\omega_t^n\right)^{\frac{2-s}{2}}\leq C,
    \end{align*}
    by \cref{gradient of G 1} and \cref{L^p bound of green functions} if we choose $1+\beta=(1+\varepsilon)\frac{2-s}{s}$.
\end{proof}

\section{Diameter and volume non-collapsing estimates}

In this section, we use our estimates for Green's functions for $x\in U_2$ to derive the uniform diameter and volume non-collapsing estimates along the Chern-Ricci flow, thereby completing the proof of \cref{main theorem}.

\begin{theorem}\label{diameter estimates}
    In the setting of \cref{main theorem}, $\left(X,\frac{\omega(t)}{t}\right)$ has uniformly bounded diameter, i.e., there exists a uniform constant $C=C\left(X,U,U_1,U_2,U_3,\lambda\left[\frac{\omega(t)}{t}|_{X\setminus U_3}\right],n,p,\chi,\omega_X,\operatorname{Ent}_p(\frac{\omega(t)}{t})\right)$ such that
    $$
\operatorname{diam}\left(X,\frac{\omega(t)}{t}\right)\leq C.
    $$
\end{theorem}
\begin{proof}
    We consider the normalized flow \eqref{normalized main flow}. In this setting it suffices to prove 
    $$
\operatorname{diam}\left(X,\omega(t)\right)\leq C.
    $$
   We have shown all the estimates of Green's functions for $x\in U_2$, we now choose a smaller neighborhood of $E$ such that $E\subset U_3\Subset U_2$. Since the metrics $\omega(t)$ are uniformly equivalent outside $U_3$ and we can also arrange that both $U_3$ and $X\setminus U_3$ are connected, it suffices to show that
   \begin{equation}\label{diam U2}
       \operatorname{diam}(U_2,\omega(t))\leq C.
   \end{equation}
  Indeed, let $d_{g_t}(\cdot,\cdot)$ be the distance function with respect to $g_t$. For each $x_0,y_0\in X$, if both points lie in $X\setminus U_3$, we choose $x_1,y_1\in\partial U_3$ and write
    $$
d_{g_t}(x_0,y_0)\leq d_{g_t}(x_0,x_1)+ d_{g_t}(x_1,y_1)+ d_{g_t}(y_1,y_0)\leq C
    $$
    because $\operatorname{diam}(U_2,\omega(t))\leq C$ and the metrics $g_t$ are uniformly equivalent outside $U_3$. If $x_0\in U_3$ and $y_0\in X\setminus U_3$, we similarly choose $x_1\in\partial U_3$ and write
$$
d_{g_t}(x_0,y_0)\leq d_{g_t}(x_0,x_1)+ d_{g_t}(x_1,y_0)\leq C.
$$
Now, we turn to prove \eqref{diam U2}. Fix $x_0,y_0\in U_2$ and set $d_t(y):=d_{g_t}(x_0,y)$ be the $1$-Lipschitz distance function on $X$. Green's formula at $x_0$ then yields that
\begin{equation}\label{Green 1}
    0=d_t(x_0)=\frac{1}{V_{\omega_t}}\int_Xd_t(y)\omega_t^n(y)+\int_X\langle\nabla G_t(x_0,\cdot),\nabla d_t\rangle_{g_t}\omega_t^n.
\end{equation}
To justify this formula, observe that $d_t\in C^0(X)$ and $|\nabla d_t|\leq 1$ almost everywhere (specifically, $d_t$ is smooth and $1$-Lipschitz outside the cut locus of $(X,\omega(t))$ and away from the point $x_0$). So, we can use convolutions to find $d_k\in C^\infty(X)$ such that 
$$
d_k\xrightarrow{C^0(X)}d_t\quad\operatorname{and}\quad\nabla d_k\xrightarrow{L^q(X)}\nabla d_t,
$$
for an arbitrary large $q$. Since $d_k$ is smooth, \eqref{Green 1} is valid for $d_k$ in place of $d_t$. Note that $|\nabla G_t|^s$ is integrable for some $s>1$, we can then choose $q>1$ large enough and take the limit with respect to $k$ and to obtain \eqref{Green 1}.

\eqref{Green 1} implies that 
$$
\frac{1}{V_{\omega_t}}\int_Xd_t(y)\omega_t^n(y)=-\int_X\langle\nabla G_t(x_0,\cdot),\nabla d\rangle_{g_t}\omega_t^n\leq\int_X|\nabla G_t(x_0,\cdot)|_{g_t}\omega_t^n\leq C,
$$
where the first inequality follows since $|\nabla d_t|_{g_t}\leq1$ and the second is due to \cref{gradient of G 2} because $x_0\in U_2$. Finally, we use Green's formula again to write
\begin{align*}
    d_{g_t}(x_0,y_0)&=d_t(y_0)=\frac{1}{V_{\omega_t}}\int_Xd_t(y)\omega_t^n(y)+\int_X\langle\nabla G_t(y_0,\cdot),\nabla d\rangle_{g_t}\omega_t^n\\
    &\leq C+\int_X|\nabla G_t(y_0,\cdot)|_{g_t}\omega_t^n\leq C.
\end{align*}
The proof is concluded by taking supremum with respect to $x_0,y_0\in U_2$.
\end{proof}

\begin{theorem}\label{volume non-collapsing estiamte}
    In the setting of \cref{main theorem}, there exists uniform constants $\alpha=\alpha(n,p)$, $r_0=r_0\left(X,U_2,U_3,U_4,\lambda\left[\frac{\omega(t)}{t}|_{X\setminus U_4}\right]\right)$ and $c=c\left(X,U,U_1,U_2,U_3,\lambda\left[\frac{\omega(t)}{t}|_{X\setminus U_3}\right],n,p,\chi,\omega_X,\operatorname{Ent}_p(\frac{\omega(t)}{t})\right)$ such that
    $$
\operatorname{Vol}_{\frac{\omega(t)}{t}}\left(B_{\frac{\omega(t)}{t}}(x,r)\right)\geq cr^{\alpha},
    $$
    for any $x\in X$ and $r\in(0,r_0)$. Here, $B_{\omega}(x,r)$ is the geodesic ball in $(X,\omega)$ of radius $r$ centered at $x$.
\end{theorem}
\begin{proof}
  We consider again the normalized Chern-Ricci flow \eqref{normalized main flow} and denote $\omega_t$ the corresponding solution. As in \cref{diameter estimates}, we choose $E\subset U_4\Subset U_3\Subset U_2$ and set 
    $$
    r_0:=\frac{1}{2}\inf_{t>0}\min\{\operatorname{dist}_{g_t}(\partial U_3,\partial U_4),\operatorname{dist}_{g_t}(\partial U_2,\partial U_3)\}>0,
    $$
the minimal of distances between the boundaries of $U_2$, $U_3$ and $U_4$, which can be arranged positive. Since the metrics $\omega(t)$ are uniformly equivalent outside $U_4$, basic Riemannian geometry yields that the volume non-collapsing estimates automatically hold for $x\in X\setminus U_3$ and $r\in(0,r_0)$ (which means that $B_{\omega_t}(x,r)\subset X\setminus U_4$).

    It suffices to prove it for $x\in U_3$ and $r\in(0,r_0)$. In this stage, $B_{\omega_t}(x,r)\Subset U_2$. Choose a cutoff function $\eta$ with compact support in $B_{\omega_t}(x,r)$ such that
    $$
\eta|_{B_{\omega_t}(x,\frac{r}{2})}\equiv1,\quad\sup_X|\nabla\eta|_{g_t}\leq\frac{4}{r}.
    $$
   Let $d(y):=d_{g_t}(x,y)$ be the distance function with respect to $x$. Select a point $z\in U_2\setminus B_{\omega_t}(x,r)$, then $d(z)\eta(z)=0$ and we apply Green's formula to the Lipschitz function $d\cdot\eta$ to get
    \begin{equation}
        0=\frac{1}{V_{\omega_t}}\int_Xd(y)\eta(y)\omega_t^n(y)+\int_X\langle\nabla G_t(z,\cdot),\eta\nabla d+d\nabla\eta\rangle_{g_t}\omega_t^n.
    \end{equation}
    This implies that
    \begin{align*}
       \frac{1}{V_{\omega_t}}\int_Xd(y)\eta(y)\omega_t^n(y) &\leq\int_X|\nabla G_t(z,\cdot)|_{g_t}(\eta+d|\nabla\eta|_{g_t})\omega_t^n\leq\int_X|\nabla G_t(z,\cdot)|_{g_t}\left(1+r\cdot\frac{4}{r}\right)\cdot\mathds{1}_{\operatorname{Supp}\eta}\omega_t^n\\
       &=5\int_{\operatorname{Supp}\eta}|\nabla G_t(z,\cdot)|_{g_t}\omega_t^n\leq5\left(\int_X|\nabla G_t(z,\cdot)|_{g_t}^s\omega_t^n\right)^\frac{1}{s}\cdot\operatorname{Vol}_{\omega_t}(B_{\omega_t}(x,r))^{\frac{s-1}{s}}\\
       &\leq C\operatorname{Vol}_{\omega_t}(B_{\omega_t}(x,r))^{\frac{s-1}{s}},
    \end{align*}
for $s:=\frac{2+1.5\varepsilon}{2+\varepsilon}$ and $\varepsilon$ is the constant in \cref{L^p bound of green functions}.  Next, choose a point $\hat{z}\in\partial B_{\omega_t}(x,\frac{r}{2})\subset U_2$, where $d(\hat{z})\eta(\hat{z})=\frac{r}{2}$, and apply Green's formula again to write
\begin{align*}
    \frac{r}{2}=d(\hat{z})\eta(\hat{z})&=\frac{1}{V_{\omega_t}}\int_Xd(y)\eta(y)\omega_t^n(y)+\int_X\langle\nabla G_t(\hat{z},\cdot),\eta\nabla d+d\nabla\eta\rangle_{g_t}\omega_t^n\\
    &\leq C\operatorname{Vol}_{\omega_t}(B_{\omega_t}(x,r))^{\frac{s-1}{s}}+\int_X|\nabla G_t(\hat{z},\cdot)|_{g_t}(\eta+d|\nabla\eta|_{g_t})\omega_t^n\\
    &\leq C\operatorname{Vol}_{\omega_t}(B_{\omega_t}(x,r))^{\frac{s-1}{s}}.
\end{align*}
    This gives our desired estimates
    $$
\operatorname{Vol}_{\omega_t}(B_{\omega_t}(x,r))\geq cr^{\frac{s}{s-1}}.
    $$
\end{proof}

\begin{remark}
    The above volume non-collapsing estimate is not optimal because the exponent $\alpha$ is strictly larger than $2n$, as opposed to \cite{Wang18}, where in the K\"ahler case we have precisely $\alpha=2n$. We believe that $\alpha$ could be improved to $2n$ by modifying the techniques developed in \cite{Wang18}. The strict inequality $\alpha>2n$ cannot guarantee the existence of tangent cones of the Gromov-Hausdorff limit space at a singular point, nor that it is isometric to a metric cone, as illustrated in \cite[Proposition 3.18]{CW20}. Luckily, we can bypass this difficulty by directly estimating Perelman's reduced length below, where we do not have to establish an intermediate Cheeger-Gromov limit space first as in \cite{Guo17, TZ16, Wang18, CW19}.
\end{remark}

\begin{proof}[End of proof of Theorem 1.1]
    As illustrated in \cite{GPSS24a}, the diameter estimate and the volume non-collapsing estimate combined with Gromov's precompactness theorem easily yields that for any sequence $0<t_i\to\infty$, there exists a subsequence (still denoted by $t_i$) such that $(X,\omega(t_i))$ converges in the Gromov-Hausdorff topology to a compact metric space $(Z,d)$. 
\end{proof}

\section{Bounding the Chern scalar curvature along the flow}
We next move to the second part of our paper, namely, the proof of \cref{main thm:uniqueness}. We will mainly follow the approach in \cite{LTZ26}. In \cite[section 2]{LTZ26}, the authors listed several properties of the immortal K\"ahler-Ricci flow, among which the parabolic Schwarz estimate and the uniform Chern scalar curvature bound are unknown in our case now, we shall check them in this section, following \cite{Zhang09}.

As in the introduction, we will assume that $X$ is a compact K\"ahler manifold from now on. Then it follows from Kawamata's base point free theorem and Iitaka's theorem in algebraic geometry that $K_X$ is semi-ample and there is a birational holomorphic map $f:X\to X_{can}$ to a normal projective variety $X_{can}$ induced by the sections of $mK_X$ for some large integer $m$ such that $f:X\setminus f^{-1}(D)\to X_{can}\setminus D$ is a biholomorphism, where $D$ is an analytic subset of $X_{can}$. Clearly, we have $E=Null(K_X)=f^{-1}(D)$ and $X_{can}^{sing}\subset D$. For simplicity of notations, we write $Y$ for $X_{can}$ and let $\omega_Y$ be the restriction of the Fubini-Study metric on $Y$ in the sequel.

We first prove a parabolic Schwarz lemma in our case:
\begin{lemma}\label{parabolic Schwarz}
Let $X$ be a compact K\"ahler minimal model of general type and let $E$ be the null locus of $K_X$. Let $\omega_0$ be a Hermitian metric on $X$, which is moreover K\"ahler in a small neighborhood of $E$. there is a uniform constant $C$ such that
    $$
\omega(t)\geq C^{-1}f^*\omega_Y,\quad\forall t\in[0,\infty).
    $$
\end{lemma}
\begin{proof}
    It suffices to establish an upper bound for $tr_{\omega(t)}f^*\omega_Y$. We let $z^i$ be coordinates on $X$ and $w^\alpha$ be coordinates on $Y$ and write $g_{i\bar{j}}$ for the entries of $\omega(t)$. We also write $h_{\alpha\bar{\beta}}$ for the entries of $\omega_Y$ and $f^\alpha$ for those of $f$. It follows that 
\begin{equation}\label{partial t}
\phi := \mathrm{tr}_g(f^*\omega_Y) = g^{i\bar{j}} f_i^\alpha \overline{f_j^\beta} h_{\alpha\bar{\beta}},
\end{equation}
and thus
\begin{align*}
    \frac{\partial}{\partial t}\phi=R^{C,i\bar{j}} f_i^\alpha \overline{f_j^\beta} h_{\alpha\bar{\beta}}+\phi,
\end{align*}
where $R^{C,i\bar{j}}:=g^{i\bar{l}}g^{k\bar{j}}R^C_{i\bar{j}}$ and $R^C_{i\bar{j}}$ are components of the Chern-Ricci $Ric^C$. The Laplacian of $\phi$ is
$$
\Delta_{\omega(t)} \phi = g^{l\bar{k}}\nabla_k \nabla_{\bar{l}} (g^{j\bar{i}} f_i^\alpha \overline{f_j^\beta} h_{\alpha\bar{\beta}}),
$$
where $\nabla:=\nabla^X\otimes Id+Id\otimes f^*\nabla^Y$ is a tensor product of the Chern connections on $X$ and the pullback bundle $f^*TY$. Since $f$ is holomorphic, we have $\nabla_{\bar{l}} f_i^\alpha = 0$, hence
\begin{equation}
    \nabla_{\bar{l}} (g^{j\bar{i}}f_i^\alpha \overline{f_j^\beta} h_{\alpha\bar{\beta}}) = g^{j\bar{i}}f_i^\alpha \overline{(\nabla_l f_j^\beta)} h_{\alpha\bar{\beta}}.
\end{equation}
Act once more $\nabla_k$ we can write
\begin{equation}
    \nabla_k \nabla_{\bar{l}} (g^{j\bar{i}}f_i^\alpha \overline{f_j^\beta} h_{\alpha\bar{\beta}}) = g^{j\bar{i}}(\nabla_k f_i^\alpha)\overline{(\nabla_l f_j^\beta)}h_{\alpha\bar{\beta}} + g^{j\bar{i}}f_i^\alpha \overline{(\nabla_{\bar{k}}\nabla_l f_j^\beta)}h_{\alpha\bar{\beta}}.
\end{equation}
Since $\nabla_{\bar{k}} f_j^\beta = 0$, we have $\nabla_{\bar{k}} \nabla_l f_j^\beta=-[\nabla_l, \nabla_{\bar{k}}] f_j^\beta$. View $f_j^\beta$ as a section of the bundle $T^*X \otimes f^*TY$, we can use the definition of tensor products of connections to write  $$[\nabla_l, \nabla_{\bar{k}}] f_j^\beta = \big( [\nabla_l, \nabla_{\bar{k}}]^X f_j \big)^\beta + \big( [\nabla_l, \nabla_{\bar{k}}]^{f^*Y} f^\beta \big)_j.$$
Therefore, we have
\begin{equation}
    \nabla_{\bar{k}} \nabla_l f_j^\beta = R_{l\bar{k}j}^p f_p^\beta - R^{f^*h,\beta}_{l\bar{k}\gamma}f_j^\gamma= R_{l\bar{k}j}^p f_p^\beta - R^{h,\beta}_{\mu\bar{\nu}\gamma} f_l^\mu \overline{f_k^\nu} f_j^\gamma.
\end{equation}
Where $R^h$ is the curvature of $(Y,\omega_Y)$ and we are using the notations in \cite[Page 131-132]{TW15} to do calculations.

Taking complex conjugate and using $\overline{R_{l\bar{k}j}^p} = \overline{g^{\bar{m}p} R_{l\bar{k}j\bar{m}}} = g^{\bar{p}m} R_{k\bar{l}m\bar{j}}$, we obtain
\begin{equation}
    \overline{\nabla_{\bar{k}} \nabla_l f_j^\beta} = R_{k\bar{l}m\bar{j}} g^{\bar{p}m} \overline{f_p^\beta} - \overline{R^{h,\beta}_{\mu\bar{\nu}\gamma}} \overline{f_l^\mu} f_k^\nu \overline{f_j^\gamma}.
\end{equation}
Utilizing $h_{\alpha\bar{\beta}} \overline{R^{h,\beta}_{\mu\bar{\nu}\gamma}} = R^h_{\nu\bar{\mu}\alpha\bar{\gamma}}$ and multiplying both sides by $h_{\alpha\bar{\beta}}$ we derive the expression
\begin{equation}\label{Lap 2}
    \Delta_{\omega(t)} \phi = |\nabla f|^2 + g^{\bar{l}k} g^{\bar{j}i} g^{\bar{p}m} R_{k\bar{l}m\bar{j}} f_i^\alpha \overline{f_p^\beta} h_{\alpha\bar{\beta}} - g^{\bar{l}k} g^{\bar{j}i} R^h_{\nu\bar{\mu}\alpha\bar{\gamma}} f_i^\alpha \overline{f_l^\mu} f_k^\nu \overline{f_j^\gamma}.
\end{equation}
As in \cite[Lemma 5.12]{Tos18}, the last term can be bounded by
$$
g^{\bar{l}k} g^{\bar{j}i} R^h_{\nu\bar{\mu}\alpha\bar{\gamma}} f_i^\alpha \overline{f_l^\mu} f_k^\nu \overline{f_j^\gamma}\leq C(tr_{\omega(t)}(f^*\omega_Y))^2=C\phi^2,
$$
where $C$ is a constant depending on the bisectional curvature of $(Y,\omega_Y)$.

For the second term in \eqref{Lap 2}, we use the commutation formulas and Bianchi identities from \cite[Page 132, 136]{TW15} to commute the indices. First, by \cite[Equation (3.6)]{TW15}, we have
\begin{equation}
    R_{k\bar{l}m\bar{j}} = R_{m\bar{l}k\bar{j}} + g_{s\bar{j}} \nabla_{\bar{l}} T_{mk}^s.
\end{equation}
Next, taking the complex conjugate of this identity yields $R_{m\bar{l}k\bar{j}} = R_{m\bar{j}k\bar{l}} - g_{k\bar{s}} \nabla_m \overline{T_{lj}^s}$. 
Taking the trace with $g^{\bar{l}k}$ and taking into account the definition of the Chern-Ricci curvature $R^C_{m\bar{j}} = g^{\bar{l}k} R_{m\bar{j}k\bar{l}}$, we obtain
\begin{equation}
\begin{split}
    g^{\bar{l}k} R_{k\bar{l}m\bar{j}} &= g^{\bar{l}k} (R_{m\bar{j}k\bar{l}} - g_{k\bar{s}} \nabla_m \overline{T_{lj}^s} + g_{s\bar{j}} \nabla_{\bar{l}} T_{mk}^s) \\
    &= R^C_{m\bar{j}} + g^{\bar{l}k} g_{s\bar{j}} \nabla_{\bar{l}} T_{mk}^s - \nabla_m \overline{T_{lj}^l}.
\end{split}
\end{equation}
Putting everything into \eqref{Lap 2} and using $R^{C, i\bar{p}} = g^{i\bar{j}}g^{m\bar{p}}R^C_{m\bar{j}}$, we arrive at
\begin{equation}\label{Lap 3}
\begin{split}
     \Delta_{\omega(t)} \phi = &|\nabla f|^2 + R^{C,i\bar{p}} f_i^\alpha \overline{f_p^\beta} h_{\alpha\bar{\beta}} - g^{\bar{l}k} g^{\bar{j}i} R^h_{\nu\bar{\mu}\alpha\bar{\gamma}} f_i^\alpha \overline{f_l^\mu} f_k^\nu \overline{f_j^\gamma} \\
     &+ g^{\bar{l}k} g^{\bar{p}m} (\nabla_{\bar{l}} T_{mk}^i) f_i^\alpha \overline{f_p^\beta} h_{\alpha\bar{\beta}} - g^{\bar{j}i} g^{\bar{p}m} (\nabla_m \overline{T_{lj}^l}) f_i^\alpha \overline{f_p^\beta} h_{\alpha\bar{\beta}}.
\end{split}
\end{equation}
Combining the evolution equation of $\phi$ with \eqref{Lap 3}, we have
\begin{equation}\label{final of phi}
\begin{split}
    \left(\frac{\partial}{\partial t} - \Delta_{\omega(t)}\right) \phi =& \phi - |\nabla f|^2 + g^{\bar{l}k} g^{\bar{j}i} R^h_{\alpha\bar{\beta}\mu\bar{\nu}} f_k^\alpha \overline{f_l^\beta} f_i^\mu \overline{f_j^\nu} \\
    &- g^{\bar{l}k} g^{\bar{p}m} (\nabla_{\bar{l}} T_{mk}^i) f_i^\alpha \overline{f_p^\beta} h_{\alpha\bar{\beta}} + g^{\bar{j}i} g^{\bar{p}m} (\nabla_m \overline{T_{lj}^l}) f_i^\alpha \overline{f_p^\beta} h_{\alpha\bar{\beta}}.
\end{split}
\end{equation}
Since in our situation all the torsions and their covariant derivatives vanish in the neighborhood $U$ of $Null(K_X)$ and remain uniformly bounded globally on $X$, the last two torsion terms in \eqref{final of phi} can be bounded by a uniform constant $C$, and thus we obtain
\begin{equation}\label{final of phi 2}
     \left(\frac{\partial}{\partial t} - \Delta_{\omega(t)}\right) \phi \le \phi - |\nabla f|^2 + C\phi^2+C.
\end{equation}
We next estimate the evolution of $\log\phi$:
\begin{equation}
    \left(\frac{\partial}{\partial t} - \Delta_{\omega(t)}\right) \log \phi = \frac{1}{\phi} \left(\frac{\partial}{\partial t} - \Delta_{\omega(t)}\right) \phi + \frac{|\nabla \phi|^2}{\phi^2}.
\end{equation}
To control the gradient term, we apply the Cauchy-Schwarz inequality. Recall that $\phi = g^{i\bar{j}} f_i^\alpha \overline{f_j^\beta} h_{\alpha\bar{\beta}}$, which implies:
\begin{equation}
    |\nabla \phi|^2 \le \left( g^{i\bar{j}} f_i^\alpha \overline{f_j^\beta} h_{\alpha\bar{\beta}} \right) \left( g^{k\bar{l}} g^{i\bar{j}} h_{\alpha\bar{\beta}} (\nabla_k f_i^\alpha) \overline{(\nabla_l f_j^\beta)} \right) = \phi |\nabla f|^2.
\end{equation}
This gives the bound $\frac{|\nabla \phi|^2}{\phi^2} \le \frac{|\nabla f|^2}{\phi}$. Substituting this and \eqref{final of phi 2} into the evolution of $\log \phi$, we see that
\begin{align}\label{log phi ineq}
    \left(\frac{\partial}{\partial t} - \Delta_{\omega(t)}\right) \log \phi &\le \frac{1}{\phi} \left( C\phi^2 + \phi - |\nabla f|^2 + C \right) + \frac{|\nabla f|^2}{\phi} \nonumber \\
    &= C \phi + 1 + \frac{C}{\phi},
\end{align}
at every point where $\phi>0$. Recall that $\hat{\omega}_t:=f^*\omega_Y+e^{-t}(\omega_0-f^*\omega_Y)$ and basic calculations yield that for the flow solution $\varphi(t)$ we have
\begin{equation}
    \left(\frac{\partial}{\partial t} - \Delta_{\omega(t)}\right) \varphi=\dot{\varphi}(t)-n+tr_{\omega(t)}(\hat{\omega}_t) \ge \frac{1}{2} \mathrm{tr}_{\omega(t)}(f^*\omega_Y) - C = \frac{1}{2} \phi - C,
\end{equation}
Set $Q := \log \phi - A \varphi$. Choosing the constant $A > 0$ sufficiently large such that
\begin{align*}
    \left(\frac{\partial}{\partial t} - \Delta_{\omega(t)}\right) Q &= \left(\frac{\partial}{\partial t} - \Delta_{\omega(t)}\right) \log \phi - A \left(\frac{\partial}{\partial t} - \Delta_{\omega(t)}\right) \varphi \\
    &\le C \phi+1+\frac{C}{\phi} - A (\frac{1}{2} \phi - C) \\
    &\le - \phi + C+\frac{C}{\phi}.
\end{align*}

Suppose $Q$ attains its maximum at some point $(x_0, t_0) \in X \times [0, T]$. At this maximum point, we may assume that $\phi\geq1$, for otherwise we already obtain an upper bound of $\phi$ thanks to the uniform boundedness of $\varphi$. It follows that at $(x_0,t_0)$,
\begin{equation}
    0 \le \left(\frac{\partial}{\partial t} - \Delta_{\omega(t)}\right) Q \le - \phi(x_0, t_0) + A C.
\end{equation}
Therefore, at the maximum point, we have $\phi(x_0, t_0) \le C$, which yields a uniform upper bound $\phi \le C$ due to the uniform boundedness of $\varphi$ again.
\end{proof}

Having \cref{parabolic Schwarz} in hand, we can give the proof of \cref{introduction:scalar curvature bound} by using the calculations in \cite{Zhang09} directly:

\begin{theorem}\label{Scalar curvature bound}
Let $X$ be a compact K\"ahler minimal model of general type and let $E$ be the null locus of $K_X$. Let $\omega_0$ be a Hermitian metric on $X$, which is moreover K\"ahler in a small neighborhood of $E$. Then the Chern scalar curvature $S_C(t)$ is uniformly bounded for $t\in[0,+\infty)$. 
\end{theorem}
\begin{proof}
The lower bound of $S(t)$ has already been proved in \cref{lem:boundedness of partial_t varphi}.

From the calculations in \cref{parabolic Schwarz}, we have 
$$
 \left(\frac{\partial}{\partial t} - \Delta_{\omega(t)}\right) \phi \le \phi - H + C\phi^2+C\le C-H,
$$
where $H:=|\nabla f|^2\geq C|\nabla\phi|^2$, as described in \cite{Zhang09}. Since in our case the flow degenerates to the K\"ahler-Ricci flow on $U$, all the calculations in \cite{Zhang09} remains valid in $U$, where we obtain
\begin{equation}
     \left(\frac{\partial}{\partial t} - \Delta_{\omega(t)}\right)(\Psi+ \phi)\le C+C|\nabla v|^2-\varepsilon|\nabla v|^4-(2-\varepsilon)\operatorname{Re}\left(\nabla(\Psi+\phi),\frac{\nabla v}{C-v}\right).
\end{equation}
Here we take $v:=\varphi+\dot{\varphi}_t$ and $\Psi:=\frac{|\nabla v|^2}{C-v}$ for some large $C$. We have shown that $v$ and $\phi$ are uniformly bounded. If the maximum point $(x,t)$ of $\Psi+\phi$ is located outside $U$, then all the metrics $\omega(t)$ are uniformly equivalent and $\varphi$ converges smoothly to $\varphi_\infty$, so we automatically have an upper bound for $\Psi(x,t)$, which yields also an upper bound for $|\nabla v|$. If the maximum point $(x,t)$ of $\Psi+\phi$ lies in $U$, then we have at $(x,t)$,
$$
0\leq C+C|\nabla v|^2-\varepsilon|\nabla v|^4,
$$
this in turn gives an upper bound for $\Psi$. Consequently, we derive that
\begin{equation}
    |\nabla v|\leq C.
\end{equation}
By considering $\Phi:=\frac{C-\Delta v}{C-v}$ as in \cite{Zhang09} and using similar arguments as above, we have that there is a uniform constant $C$ such that
\begin{equation}
    -\Delta_{\omega(t)}v\le C.
\end{equation}
Now, an easy calculation using the flow equation gives that
$$
S_C(t)=-\Delta_{\omega(t)}v-\phi\le C,
$$
the proof is therefore concluded.
    
\end{proof}

\section{Perelman's reduced length and reduced volume}\label{section Per}

In this section, we attempt to introduce Perelman's famous reduced length and reduced volume under the Chern-Ricci flow and calculate explicitly the evolution of the reduced length $L$. Moreover, we establish a version of almost monotonicity of the Jacobian of the $L$-exponential map in our case. Most of the calculations can be carried out by imitating Perelman's proof \cite{Per02} or the note \cite{KL08} directly.

\subsection{The $\mathcal{L}$-Geodesic equation and the $\mathcal{L}$-exponential map}

As in \cite[section 4.2]{LTZ26}, we first perform a reparametrization $\tilde{\omega}(s):=e^{t-T}\omega(t)$ with $s:=\frac{1}{2}(e^{t-T}-1)$ to get the unnormalized Chern-Ricci flow
\begin{equation}\label{reparametrization}
    \partial_s\tilde{\omega}(s)=-2Ric^C(\tilde{\omega}(s)),\quad s\geq s_T:=\frac{1}{2}(e^{-T}-1).
\end{equation}
We then use the change of variable $\tau=-s$, \eqref{reparametrization} becomes
\begin{equation}\label{backward CRF}
    \partial_\tau \omega=2Ric^C(\omega).
\end{equation}
If we write $M$ for the manifold and for each $X,Y\in T_{\mathbb{R}}M$,
$$
g(X,Y)=\omega(X,JY), \quad Rc(X,Y):=Ric^C(X,JY),
$$
where $J$ is the complex structure on $M$ and $Rc$ is a symmetric $2$-tensor on $M$. The backward flow \eqref{backward CRF} is equivalent to 
\begin{equation}\label{real backward CRF}
    \partial_\tau g=2Rc_g.
\end{equation}
So in the sequel, we will do calculations for the real flow \eqref{real backward CRF} on the real tangent bundle $T_{\mathbb{R}}M$. In practical use, we will always choose a large $T$ and a small positive number $0<\bar{\tau}\ll -s_T<\frac{1}{2}$ and consider the flow \eqref{real backward CRF} on $[0,\bar{\tau}]$.

For each piecewise smooth curve $\gamma:[\tau_1,\tau_2]\to M$ with $0\leq\tau_1<\tau_2$, the $\mathcal{L}$-length of $\gamma$ is defined to be 
\begin{equation}
    \mathcal{L}(\gamma):=\int_{\tau_1}^{\tau_2}\sqrt{\tau}\left(S(\gamma(\tau))+|\dot{\gamma}(\tau)|^2_{g(\tau)}\right)d\tau,
\end{equation}
where $S$ denotes the Chern scalar curvature along the curve $\gamma$ in $M$, which can be computed by taking the real trace with respect to the tensor $Rc$ (it equals to two times of the complex trace $S_C$ of $Ric^C$). Consider a variation $$\tilde{\gamma}(s,\tau):(-\varepsilon,\varepsilon)\times[\tau_1,\tau_2]\to M$$ of the curve $\gamma$. If we write $\nabla$ as the usual Levi-Civita connection on $T_{\mathbb{R}}M$ and set 
$$
X:=\frac{\partial\tilde{\gamma}}{\partial\tau},\quad Y:=\frac{\partial\tilde{\gamma}}{\partial s},
$$ 
then exactly the same calculations as in \cite[Section 7]{Per02} yields that the first variation of $\mathcal{L}$ is
\begin{equation}\label{first variation of L}
    \delta_Y\mathcal{L}=\int_{\tau_1}^{\tau_2}\sqrt{\tau}(\langle Y,\nabla S\rangle+2\langle\nabla_XY,X\rangle)d\tau,
\end{equation}
where we write $\delta_Y$ to mean $\frac{d}{ds}|_{s=0}$. Using the flow equation \eqref{real backward CRF} and doing integration by parts we derive that $\mathcal{L}$-geodesic equation (the Euler-Lagrange equation of \eqref{first variation of L}) is
\begin{equation}\label{L geodesic equation}
    \nabla_XX-\frac{1}{2}\nabla S+\frac{1}{2\tau}X+2Rc(X,\cdot)=0.
\end{equation}
Making a change of variable $s=\sqrt{\tau}$ and set $\hat{X}:=\frac{d\gamma}{ds}=2sX$, we get
\begin{equation}
    \mathcal{L}(\gamma)=\int_{s_1}^{s_2}\left(2s^2S(\gamma(s))+\frac{1}{2}|\dot{\gamma}(s)|^2_{g(s)}\right)ds,
\end{equation}
with $s_i=\sqrt{\tau_i}$ and \eqref{L geodesic equation} becomes
\begin{equation}\label{normalized L geodesic equation}
    \nabla_{\hat{X}}\hat{X}-2s^2\nabla S+4sRc(\hat{X},\cdot)=0.
\end{equation}
Now standard ODE theory implies that for each $p\in M$ and $v\in T_pM$ there exists an $\varepsilon>0$ and $\gamma(s):[s_1,s_1+\varepsilon)\to M$ such that \eqref{normalized L geodesic equation} holds, with $\gamma(s_1)=p$ and
\begin{equation}\label{initial value}
    \frac{1}{2}\dot{\gamma}(s_1)=\lim_{\tau\to\tau_1}\sqrt{\tau}\frac{d\gamma}{d\tau}=v.
\end{equation}
We claim that the solution $\gamma(s)$ can be defined on the whole interval $[s_1,s_2]$. Indeed, we use \eqref{real backward CRF} and \eqref{normalized L geodesic equation} to compute
\begin{equation}\label{eq:ds X hat}
\begin{aligned}
    &\frac{d}{ds}|\hat{X}|=\frac{1}{2|\hat{X}|}\frac{d}{ds}|\hat{X}|^2=\frac{1}{2|\hat{X}|}(4sRc(\hat{X},\hat{X})+2\langle\nabla_{\hat{X}}\hat{X},\hat{X}\rangle)\\
    =&\frac{1}{2|\hat{X}|}(-4sRc(\hat{X},\hat{X})+4s^2\langle\nabla S,\hat{X}\rangle)\\
    \leq& C|\hat{X}|+2s^2\left\langle\nabla S,\frac{\hat{X}}{|\hat{X}|}\right\rangle.
\end{aligned}
\end{equation}
An application of Shi's type estimate \cite{Shi89} (see also \cite[Appendix D]{KL08}) yields that we have the estimate
$$
|\nabla S|(x,\tau)\leq \frac{C}{\sqrt{\tau_2-\tau}},
$$
for $x\in M$ and $\tau\in[\tau_1,\tau_2)$. Indeed, our flow degenerates to the standard Ricci flow on some $U^\prime\Subset U$, so the local estimates of \cite{Shi89} can be directly applied in $U^\prime$. On the other hand, the flow converges smoothly outside $U^\prime$ by Gill's result \cref{smooth convergence}, hence all kinds of curvatures are uniformly bounded there. We thus arrive at
\begin{equation}
    \frac{d}{ds}|\hat{X}|\leq C|\hat{X}|+\frac{C}{\sqrt{s_2-s}},
\end{equation}
and hence the claim follows.

\begin{definition}\label{def:L exp map}
    For each $p\in M$ and $\tau\in[\tau_1,\tau_2]$, the $\mathcal{L}$-exponential map $\mathcal{L}\exp_\tau:T_pM\to M$ is defined to take each $v\in T_pM$ to $\gamma_{p,v}(\tau)$, where $\gamma_{p,v}$ is the unique curve satisfying \eqref{L geodesic equation}, \eqref{initial value} and starting from $p$. 
\end{definition}
In the sequel we will always take $\tau_1=0$.
\begin{definition}
    Fix $p\in M$. For each $q\in M$ and $\bar{\tau}>0$, define
    $$
L(q,\bar{\tau}):=\inf_\gamma \mathcal{L}(\gamma),
    $$
    where $\gamma$ is taken from all curves such that $\gamma(0)=p$ and $\gamma(\bar{\tau})=q$. We can then define the reduced length $l$ to be
    $$
l(q,\bar{\tau}):=\frac{1}{2\sqrt{\bar{\tau}}}L(q,\bar{\tau}),
    $$
    and the reduced volume to be
    $$
\tilde{V}(\tau)=\int_M\tau^{-n}e^{-l(q,\tau)}dq.
    $$
    Here we remind the reader that $M$ has real dimension $2n$.
\end{definition}

The following proposition says that $\mathcal{L}$-geodesics minimize $\mathcal{L}$ for a short time:

\begin{proposition}\label{short time minimize}
    For each fixed $p\in M$, there is an $r=r(p)>0$ such that for every $q\in M$ with $d_{g(0)}(q,p)\leq10r$ and every $0<\bar{\tau}<r^2$, there is a unique $\mathcal{L}$-geodesic $\gamma:[0,\bar{\tau}]\to B_{g(0)}(p,100r)$ with $\gamma(0)=p,\gamma(\bar{\tau})=q$, and such that $\mathcal{L}(\gamma)=L(q,\bar{\tau})$. Furthermore, $\gamma$ and $L(q,\bar{\tau})$ depends smoothly on $q,\bar{\tau}$.
\end{proposition}
\begin{proof}
    In order to use the implicit function theorem, we make the change of variable $s=\sqrt{\tau}$ with $\bar{s}=\sqrt{\bar{\tau}}$ and set $y(\sigma):=\gamma(\sigma \bar{s})$, $\sigma\in[0,1]$. Then $\dot{y}(\sigma)=\bar{s}\dot{\gamma}(\sigma\bar{s})=\bar{s}\hat{X}(\sigma\bar{s})$. The normalized geodesic equation \eqref{normalized L geodesic equation} then becomes
    \begin{equation}\label{y geodesic}
\nabla_{\dot{y}}\dot{y}=\bar{s}^2\nabla_{\hat{X}}\hat{X}=\bar{s}^2(2\sigma^2\bar{s}^2\nabla S-4\sigma\bar{s}Rc(\bar{s}^{-1}\dot{y},\cdot))=2\sigma^2\bar{s}^4\nabla S-4\sigma\bar{s}^2Rc(\dot{y},\cdot).
    \end{equation}
    Note that at $\bar{s}=0$, \eqref{y geodesic} is just the standard Riemannian geodesic equation $\nabla^{g(0)}_{\dot{y}}\dot{y}=0$. Given $w\in T_pM$ and $\bar{s}>0$ small, we have to find a curve $y$ with $y(0)=p,y(1)=q$ and $\dot{y}(0)=w$. By standard ODE existence theory, there exists a curve $y=y_{w,\bar{s}}$ satisfying the above properties except that $y(1)=q$. Consider the map $G:T_pM\times[0,\varepsilon)\to M$:
    $$
G(w,\bar{s}):=\gamma_{w,\bar{s}}(1).
    $$
    Then we have to solve the equation $G(w,\bar{s})=q$. By our discussion above $G(w,0)=exp_p(w)$ is just the Riemannian exponential map, then it is clear $G(0,0)=p$ and $G_w(0,0)=Id$. The implicit function theorem now yields an $r=r(p)$ satisfying our assumptions.

        To see that $\gamma$ minimize the $\mathcal{L}$-length, we set $\hat{L}(q,\tau)$ to be the $\mathcal{L}$-length of $\gamma$ and prove that $\hat{\mathcal{L}}=\mathcal{L}$ locally. By definition we clearly have $\mathcal{L}\leq\hat{\mathcal{L}}$. For the reverse inequality, let $\alpha:[0,\bar{\tau}]\to M$ be a smooth curve from $p$ to $q$, then 
        \begin{equation}\label{derivative of hat L}
            \frac{d}{d\tau}\hat{L}(\alpha(\tau),\tau)=\hat{L}_\tau(\alpha(\tau),\tau)+\langle\nabla\hat{L}(\alpha(\tau),\tau),\dot{\alpha}(\tau)\rangle.
        \end{equation}
        View $\alpha(\tau):=q^\prime$ as an endpoint, then the same calculations as in \cite[(18.2)]{KL08} using the first variation formula of $\mathcal{L}$ shows that
        \begin{equation}
            \nabla\hat{L}(\alpha(\tau),\tau)=2\sqrt{\tau}X(\tau),
        \end{equation}
        where $X(\tau)=\dot{\gamma}(\tau)$ with $\gamma$ the unique $\mathcal{L}$-geodesic from $p$ to $q^\prime$. Similarly, the calculations in \cite[18.6]{KL08} yields that
        \begin{equation}
            \hat{L}_\tau(\alpha(\tau),\tau)=  \hat{L}_\tau(q^\prime,\tau)=\sqrt{\tau}(S(\alpha(\tau))-|X(\tau)|^2).
        \end{equation}
        
        Putting these into \eqref{derivative of hat L} we obtain
        \begin{align*}
             \frac{d}{d\tau}\hat{L}(\alpha(\tau),\tau)&=\langle2\sqrt{\tau}X(\tau),\dot{\alpha}(\tau)\rangle+\sqrt{\tau}(S(\alpha(\tau))-|X(\tau)|^2)\\
             &\leq\sqrt{\tau}(S+|\dot{\alpha}|^2)=\frac{d}{d\tau}\mathcal{L}(\alpha|_{[0,\tau]}).
        \end{align*}
        The proof is therefore concluded by integrating both sides.
        \end{proof}

Now from \cref{short time minimize} and basic broken line arguments we can conclude that for each $q\in M$ and $\bar{\tau}>0$, there is a piecewise smooth $\mathcal{L}$-geodesic $\gamma:[0,\bar{\tau}]\to M$ with $\gamma(0)=p$ and $\gamma(\bar{\tau})=q$ such that $\gamma$ minimize the $\mathcal{L}$-length.  

\begin{definition}\label{def: def of cut locus}
Fix $p\in M$ and $\tau>0$. Let $\mathcal{L}\exp_\tau:T_pM\to M$ be the $\mathcal{L}$-exponential map defined in \cref{def:L exp map}. The time-$\tau$ $\mathcal{L}$-cut locus $B_\tau\subset M$ is the set of points which are either endpoints of more than one minimizing $\mathcal{L}$-geodesics, or endpoints of a minimizing $\mathcal{L}$-geodesic $\gamma_{p,v}$ with $v\in T_pM$ a critical point of $\mathcal{L}\exp_\tau$.
\end{definition}

Let $G_\tau$ be the complement of $B_\tau$ and let $\Omega_\tau$ be the corresponding initial directions $v\in T_pM$ of the points in $G_\tau$. It is then easy to see that $\mathcal{L}\exp_\tau|_{\Omega_\tau}$ is a bijective local diffeomorphism. Arguing exactly the same as in the Riemannian case (An easy consequence of Sard's theorem and the implicit function theorem) we can deduce that $B_\tau$ is contained in a countable union of hypersurfaces in $M$, whence has measure zero. We summerize our discussions as follows:

\begin{proposition}\label{cut locus has measure zero}
    Fix $p\in M$ and $\tau>0$. Let $\mathcal{L}\exp_\tau:T_pM\to M$ be the $\mathcal{L}$-exponential map and let $B_\tau$ be the time-$\tau$ $\mathcal{L}$ cut-locus. Then $B_\tau$ has measure zero and there is an open set $\Omega_\tau\subset T_pM$ such that $\mathcal{L}\exp_\tau:\Omega_\tau\to M\setminus B_\tau$ is a diffeomorphism.
\end{proposition}

\subsection{Derivatives of the reduced length}

The calculation of the first derivatives of $L(q,\tau)$ and the second variation of $\mathcal{L}$ is almost identical to that under the Ricci flow provided that we are using the Levi-Civita connection, just replacing the Ricci curvature by the Chern-Ricci curvature. So we omit the details when the calculations can be copied along the lines from \cite[section 7]{Per02} and \cite{KL08}.

Suppose that $\bar{\tau}>0$ and $q\notin B_{\tau}$. Let $\gamma:[0,\bar{\tau}]\to M$ be the unique minimizing $\mathcal{L}$-geodesic from $p$ to $q$ and let $X(\tau):=\dot{\gamma}(\tau)$. The same calculations as the real case yield that
\begin{equation}\label{nabla L}
    \nabla L(q,\bar{\tau})=2\sqrt{\bar{\tau}}X(\bar{\tau}),
\end{equation}
and
\begin{equation}\label{L_tau}
    L_{\bar{\tau}}(q,\bar{\tau})=\sqrt{\bar{\tau}}(S(q)-|X(\bar{\tau})|^2).
\end{equation}

For each $X\in T_\mathbb{R}M$, Set 
\begin{equation}\label{def of H and K}
    H(X):=-S_\tau-\frac{1}{\tau}S-2\langle\nabla S,X\rangle+2Rc(X,X),\quad K:=\int_0^{\bar{\tau}}\tau^{\frac{3}{2}}H(X(\tau))d\tau.
\end{equation}
By using the $\mathcal{L}$-geodesic equation, calculating the quantity $\frac{d}{d\tau}(S(\gamma(\tau))+|X(\tau)|^2)$ and doing integration by parts, we can derive
\begin{equation}\label{equality of K}
S(\gamma(\bar{\tau}))+|X(\bar{\tau})|^2=\frac{1}{\bar{\tau}^{\frac{3}{2}}}\left(-K+\frac{1}{2}L(q,\bar{\tau})\right).
\end{equation}
Plugging this into \eqref{nabla L} and \eqref{L_tau} we can rewrite 
\begin{equation}\label{nabla L 1}
    |\nabla L|^2(q,\bar{\tau})=-4\bar{\tau}S(q)+\frac{2}{\sqrt{\bar{\tau}}L(q,\bar{\tau})}-\frac{4}{\sqrt{\bar{\tau}}}K,
\end{equation}
and
\begin{equation}\label{L_tau 1}
       L_{\bar{\tau}}(q,\bar{\tau})=2\sqrt{\bar{\tau}}S(q)-\frac{1}{2\bar{\tau}}L(q,\bar{\tau})+\frac{1}{\bar{\tau}}K.
\end{equation}
Taking derivative again of the first variation formula \eqref{first variation of L} we have
\begin{equation}\label{second variation of L}
\delta^2_Y\mathcal{L}=\int_0^{\bar{\tau}}\sqrt{\tau}\left(Y\cdot Y\cdot S+2\langle\nabla_X\nabla_YY,X\rangle+2\langle R(Y,X)Y,X\rangle+2|\nabla_XY|^2\right)d\tau,
\end{equation}
where $R(Y,X)Y:=\nabla_Y\nabla_XY-\nabla_X\nabla_YY$ is the Riemannian curvature tensor. An integration by parts using the calculations in \eqref{first variation of L} yields that
\begin{equation}\label{delta_nabla YY}
\delta_{\nabla_YY}\mathcal{L}=2\sqrt{\bar{\tau}}\langle\nabla_YY,X\rangle-\int_0^{\bar{\tau}}\left\langle\nabla_YY,\nabla S-2\nabla_XX-4Rc(X,\cdot)-\frac{1}{\tau}X\right\rangle.
\end{equation}
Then, we define a quadratic form
\begin{equation}
    \begin{aligned}
        Q(Y,Y):=&\delta^2_Y\mathcal{L}-\delta_{\nabla_YY}\mathcal{L}\\
        =&\int_0^{\bar{\tau}}\sqrt{\tau}[Hess_S(Y,Y)+2\langle R(Y,X)Y,X\rangle+2|\nabla_XY|^2\\
       &-4(\nabla_YRc)(Y,X)+2(\nabla_XRc)(Y,Y)]d\tau.\\
    \end{aligned}
\end{equation}
To compute the real Hessian of $L$, we have to introduce the notation of $\mathcal{L}$-Jacobi fields under the Chern-Ricci flow:

\begin{definition}\label{def:L Jacobi field}
    Let $\gamma:[0,\bar{\tau}]\to M$ be a $\mathcal{L}$-geodesic and let $X(\tau):=\dot{\gamma}(\tau)$. A vector field $Y(\tau)$ along $\gamma$ is called an $\mathcal{L}$-Jacobi field if $TY\equiv0$, where the operator $T$ is defined by
    \begin{equation}
        \begin{aligned}
            TY:=&-\nabla_X\nabla_XY-\frac{1}{2\tau}\nabla_XY+R(Y,X)X\\
            &+\frac{1}{2}Hess_S(Y,\cdot)-2(\nabla_YRc)(X,\cdot)-2Rc(\nabla_YX,\cdot).
        \end{aligned}
    \end{equation}
\end{definition}
Given $v,w\in T_pM$, by standard ODE theory, there is a unique $\mathcal{L}$-Jacobi field $Y$ along $\gamma$ with $Y(0)=v$ and $\lim_{\tau\to0^+}\sqrt{\tau}(\nabla_{\dot{\gamma}}Y)(\tau)=w$. In fact, the definition of $\mathcal{L}$-Jacobi field comes from the following proposition:
\begin{proposition}\label{prop:L jacobi field derive}
    Consider a family of $\mathcal{L}$-geodesic variations $\tilde{\gamma}(s,\tau):(-\varepsilon,\varepsilon)\times[0,\bar{\tau}]\to M$ of the $\mathcal{L}$-geodesic $\gamma$. Then, $Y(\tau):=\frac{\partial\tilde{\gamma}(s,\tau)}{\partial s}|_{s=0}$ is an $\mathcal{L}$-Jacobi field along $\gamma$.
\end{proposition}
\begin{proof}
    Each $\gamma_s:=\tilde{\gamma}(s,\cdot)$ is an $\mathcal{L}$-geodesic, so we have that
    $$
\nabla_XX-\frac{1}{2}\nabla S+\frac{1}{2\tau}X+2Rc(X,\cdot)=0.
    $$
    Act $\nabla_Y$ on both sides, we compute each terms:
    $$
\nabla_Y\nabla_XX=\nabla_X\nabla_YX+R(Y,X)X=\nabla_X\nabla_XY+R(Y,X)X,
    $$
    $$
\langle\nabla_Y\nabla S,Y\rangle=Hess_S(Y,Y),
    $$
    and
    $$
\nabla_Y(Rc(X,\cdot))=(\nabla_YRc)(X,\cdot)-Rc(\nabla_YX,\cdot).
    $$
    Putting everything together we derive that $TY\equiv0$.
\end{proof}
We remark that converse to \cref{prop:L jacobi field derive}, given any $\mathcal{L}$-Jacobi field $Y$ along $\gamma$, we can also use the differential of the $\mathcal{L}$-exponential map to construct an $\mathcal{L}$-geodesic variation of $\gamma$ such that the variation vector field of this variation is exactly $Y$. 

Now, a standard calculation gives
\begin{lemma}\label{expression of Q}
    $$
Q(Y,Y)=2\int_0^{\bar{\tau}}\sqrt{\tau}\langle Y,TY\rangle d\tau+2\sqrt{\bar{\tau}}\langle\nabla_XY(\bar{\tau}),Y(\bar{\tau})\rangle.
    $$
\end{lemma}

To compute $Hess_{L(\cdot,\bar{\tau})}(\cdot,\cdot)$, fix $q\in M\setminus B_{\bar{\tau}}$ and $w\in T_qM$ and let $\gamma:[0,\bar{\tau}]\to M$ be the minimizing $\mathcal{L}$-geodesic from $p$ to $q$. Choose a curve $c(s):(-\varepsilon,\varepsilon)\to M$ with $c(0)=q$ and $\dot{c}(0)=w$. Consider the $\mathcal{L}$-geodesic variation $\tilde{\gamma}(s,\tau)$ such that each $\gamma_s$ has endpoint $c(s)$, that is, $\tilde{\gamma}(s,\bar{\tau})=c(s)$. By \cref{prop:L jacobi field derive}, $Y(\tau)$ is an $\mathcal{L}$-Jacobi field along $\gamma$ with $Y(0)=0$ and $Y(\bar{\tau})=w$. It follows from \eqref{nabla L} that
\begin{equation}\label{Hess of L}
\begin{aligned}
    Hess_{L(q,\bar{\tau})}(w,w)&=\langle\nabla_Y\nabla L,Y\rangle(\bar{\tau})=\langle\nabla_Y2\sqrt{\bar{\tau}}X(\bar{\tau}),Y(\bar{\tau})\rangle\\
    &=2\sqrt{\bar{\tau}}\langle\nabla_YX(\bar{\tau}),Y(\bar{\tau})\rangle=Q(Y,Y).
    \end{aligned}
\end{equation}

\begin{lemma}\label{Hess bound by Q}
    A minimizer $Y$ of $Q(Y,Y)$ with $Y(0)=0$ and $Y(\bar{\tau})=w$ must be an $\mathcal{L}$-Jacobi field. In particular, 
    $$
Hess_{L(q,\bar{\tau})}(w,w)\leq Q(Y,Y)
    $$
    for any vector field $Y$ along $\gamma$ with $Y(0)=0$ and $Y(\bar{\tau})=w$.
\end{lemma}
\begin{proof}
  Since $Q$ is a bilinear quadratic form, if $Y_0$ is a minimizer of $Q$, we can write
  \begin{align*}
0&=\frac{d}{ds}|_{s=0}Q(Y_0+sZ,Y_0+sZ)=2Q(Y_0,Z)\\
&=4\int_0^{\bar{\tau}}\sqrt{\tau}\langle Z,TY_0\rangle d\tau+4\sqrt{\bar{\tau}}\langle\nabla_XY(\bar{\tau}),Z(\bar{\tau})\rangle\\
&=4\int_0^{\bar{\tau}}\sqrt{\tau}\langle Z,TY_0\rangle d\tau,
  \end{align*}
for any vector field $Z$ along $\gamma$ with $Z(0)=0$ and $Z(\bar{\tau})=0$. This implies that $TY_0\equiv0$ and hence $Y_0$ is an $\mathcal{L}$-Jacobi field. 

It follows from \eqref{Hess of L} that 
$$
Hess_L(w,w)=Q(Y_0,Y_0)\leq Q(Y,Y),
$$
 for any vector field $Y$ along $\gamma$ with $Y(0)=0$ and $Y(\bar{\tau})=w$.
\end{proof}

\subsection{The Laplacian bound of $L$}

Having \cref{Hess bound by Q} in hand, we have to choose an appropriate $Y$ on the right-hand side to control $Hess_L$. Suppose $Y(\bar{\tau})$ is a unit vector field at $\gamma(\bar{\tau})$. As in \cite{Per02}, we solve
\begin{equation}\label{construction of tilde Y}
    \nabla_X\tilde{Y}=-Rc(\tilde{Y},\cdot)+\frac{1}{2\tau}\tilde{Y},
\end{equation}
on the interval $(0,\bar{\tau}]$ with the boundary condition $\tilde{Y}(\bar{\tau})=Y(\bar{\tau})$. We compute
$$
\frac{d}{d\tau}|\tilde{Y}|^2=2\langle\nabla_X\tilde{Y},\tilde{Y}\rangle+2Rc(\tilde{Y},\tilde{Y})=\frac{1}{\tau}|\tilde{Y}|^2.
$$
This in turn implies that
\begin{equation}\label{tilde Y size}
    |\tilde{Y}(\tau)|^2=\frac{\tau}{\bar{\tau}},
\end{equation}
since $|Y(\bar{\tau})|=1$. Clearly, \eqref{tilde Y size} yields that we can extend $\tilde{Y}$ continuously to $[0,\bar{\tau}]$ and setting $\tilde{Y}(0)=0$. 

\begin{lemma}\label{lem:estimate of Hess}
    We have
    \begin{equation}\label{hess bound 1}
        Hess_L(Y(\bar{\tau}),Y(\bar{\tau}))\leq Q(\tilde{Y},\tilde{Y})=\frac{1}{\sqrt{\bar{\tau}}}-2\sqrt{\bar{\tau}}Rc(Y(\bar{\tau}),Y(\bar{\tau}))-\int_0^{\bar{\tau}}\sqrt{\tau}H(X,\tilde{Y})d\tau,
    \end{equation}
    where
    \begin{equation}\label{expression of H}
        \begin{aligned}
            H(X,\tilde{Y}):=&-Hess_S(\tilde{Y},\tilde{Y})-2\langle R(\tilde{Y},X)\tilde{Y},X\rangle-4\nabla_XRc(\tilde{Y},\tilde{Y})\\
            &+4\nabla_{\tilde{Y}}Rc(\tilde{Y},X)-2Rc_\tau(\tilde{Y},\tilde{Y})+2|Rc(\tilde{Y},\cdot)|^2-\frac{1}{\tau}Rc(\tilde{Y},\tilde{Y}).
        \end{aligned}
    \end{equation}
\end{lemma}
\begin{proof}
   Using \cref{Hess bound by Q}, the calculations follow from the Ricci-flow case verbatim, see \cite{Per02} and \cite{KL08}.
\end{proof}

We are now ready to derive the real Laplacian bound of $L(\cdot,\bar{\tau})$. Choose an orthonormal basis $\{Y_i(\bar{\tau})\}_{i=1}^{2n}$ of $T_{\gamma(\bar{\tau})}M$ and let $\tilde{Y}_i(\tau)$ be defined as in \eqref{construction of tilde Y}. Set $e_i(\tau):=\left(\frac{\bar{\tau}}{\tau}\right)^{\frac{1}{2}}\tilde{Y}_i(\tau)$, we have the following

\begin{lemma}\label{onb of e_i}
    $\{e_i(\tau)\}_{i=1}^{2n}$ forms an orthonormal basis of $T_{\gamma(\tau)}M$ for each $\tau\in[0,\bar{\tau}]$.
\end{lemma}
\begin{proof}
    We have $\tilde{Y}_i(\tau)=\left(\frac{\tau}{\bar{\tau}}\right)^{\frac{1}{2}}e_i(\tau)$, putting this into \eqref{construction of tilde Y}, we derive
    \begin{equation}
    \begin{aligned}
    \nabla_X \tilde{Y}_i &= \frac{d}{d\tau}\left(\sqrt{\frac{\tau}{\overline{\tau}}}\right) e_i(\tau) + \sqrt{\frac{\tau}{\overline{\tau}}} \nabla_X e_i(\tau) = \frac{1}{2\tau}\sqrt{\frac{\tau}{\overline{\tau}}} e_i(\tau) + \sqrt{\frac{\tau}{\overline{\tau}}} \nabla_X e_i(\tau)\\
   & =-Rc\left(\sqrt{\frac{\tau}{\overline{\tau}}} e_i(\tau), \cdot\right) + \frac{1}{2\tau} \sqrt{\frac{\tau}{\overline{\tau}}} e_i(\tau),
    \end{aligned}
    \end{equation}
    this gives the evolution of $e_i(\tau)$:
    \begin{equation}
    \nabla_X e_i = -Rc(e_i, \cdot).
    \end{equation}
    Using the flow equation \eqref{real backward CRF}, we can write
    \begin{equation}
    \frac{d}{d\tau} \langle e_i, e_j \rangle_{g(\tau)} = \langle \nabla_X e_i, e_j \rangle + \langle e_i, \nabla_X e_j \rangle + 2Rc(e_i, e_j)=0.
    \end{equation}
    Consequently, $\langle e_i(\tau), e_j(\tau) \rangle=\delta_{ij}$ for each $\tau\in[0,\bar{\tau}]$.
\end{proof}

We will record two results due to Liu-Yang \cite{LY17}. We first briefly recall the notations involving the torsion tensor introduced in \cite{LY17}. Let $\omega = \sqrt{-1} g_{i\bar{j}} dz^i \wedge d\bar{z}^j$ be a Hermitian metric. The torsion tensor $\mathcal{T}$ of the Chern connection is locally given by
\begin{equation}
    \mathcal{T}_{ij}^k = g^{k\bar{l}} \left( \frac{\partial g_{j\bar{l}}}{\partial z^i} - \frac{\partial g_{i\bar{l}}}{\partial z^j} \right).
\end{equation}
Following \cite{LY17}, we define two different $(1,1)$-forms constructed from the contractions of $\mathcal{T}$ and its complex conjugate $\bar{\mathcal{T}}$:
\begin{equation}
    \mathcal{T} \boxdot \bar{\mathcal{T}} := g^{p\bar{q}} g_{k\bar{l}} \mathcal{T}_{ip}^k \overline{\mathcal{T}_{jq}^l} dz^i \wedge d\bar{z}^j,
\end{equation}
and 
\begin{equation}
    \mathcal{T} \circ \bar{\mathcal{T}} := g^{p\bar{q}} g^{s\bar{t}} g_{k\bar{j}} g_{i\bar{l}} \mathcal{T}_{sp}^k \overline{\mathcal{T}_{tq}^l} dz^i \wedge d\bar{z}^j.
\end{equation}
In general, these two $(1,1)$-forms are not identical. 
\begin{lemma}\cite[Theorem 1.12 (4)]{LY17} \label{LY17 1}
    On a compact Hermitian manifold $(M,\omega)$, the relation between the Chern-Ricci curvature  $Ric^C$ and the complexified Riemannian Ricci curvature $Ric^{\mathbb{C}}$ is 
    \begin{align*}
        Ric^{\mathbb{C}}=&Ric^C-\sqrt{-1}\Lambda(\partial\bar{\partial}\omega)-\frac{1}{2}(\partial\partial^*\omega+\bar{\partial}\bar{\partial}^*\omega)+\frac{\sqrt{-1}}{4}(2\mathcal{T}\boxdot \bar{\mathcal{T}}+\mathcal{T}\circ\bar{\mathcal{T}})\\
        &+\frac{1}{2}\left(\mathcal{T}([\partial^*\omega]^\#+\overline{\mathcal{T}([\partial^*\omega]^\#})\right)=Ric^C+\mathcal{E}(\mathcal{T}),
    \end{align*}
    here, by $\mathcal{E}(\mathcal{T})$ we always mean various product and contraction operations of the metric tensor and torsion terms.
\end{lemma}

\begin{lemma}\cite[Corollary 1.13]{LY17} \label{LY17 2}
    On a compact Hermitian manifold $(M,\omega)$, the Riemannian scalar curvature $R$ and the Chern scalar curvature are related by
    $$
R=2S_{C}+\left(\langle\partial\partial^*\omega+\bar{\partial}\bar{\partial}^*\omega,\omega\rangle-2|\partial^*\omega|^2\right)-\frac{1}{2}|\mathcal{T}|^2.
    $$
    In particular, we have that
    $$
R-S=\left(\langle\partial\partial^*\omega+\bar{\partial}\bar{\partial}^*\omega,\omega\rangle-2|\partial^*\omega|^2\right)-\frac{1}{2}|\mathcal{T}|^2=\mathcal{E}(\mathcal{T}).
    $$
\end{lemma}

\begin{remark}\label{rmk:torsion bound}
    In fact, we even do not need the explicit relation in \cref{LY17 1} and \cref{LY17 2}. Since our Chern-Ricci flow is K\"ahler in a neighborhood $U$ of $E$, it follows that $Rc=Ric$ and $S=R$ in $U$. By Gill's result \cref{smooth convergence} we have that $\omega(t)$ converges smoothly outside $U$, hence all the curvatures are uniformly bounded. As a result, we always have $|Rc-Ric|\leq C$ $|S-R|\leq C$ and $|\nabla S-\nabla R|\leq C$ globally on $M$. 
\end{remark}

\begin{proposition}\label{prop:estimate of Delta L}
  We have the following Laplacian estimate of $L$:
    $$
\Delta_{g(\bar{\tau})}L(\cdot,\bar{\tau})\leq \frac{2n}{\sqrt{\bar{\tau}}}-2\sqrt{\bar{\tau}}S-\frac{1}{\bar{\tau}}K+C+C(T)L(\cdot,\bar{\tau}),
    $$
    where $C$ is a uniform constant and $C(T)$ tends to $0$ as the time $T$ in \eqref{reparametrization} tends to $+\infty$.
\end{proposition}
\begin{proof}
    Combining \cref{lem:estimate of Hess} and \cref{onb of e_i} and summing over $i$ we can write
    \begin{equation}\label{Lap of L 1}
        \begin{aligned}
            \Delta_{g(\bar{\tau})}L\leq\frac{2n}{\sqrt{\bar{\tau}}}-2\sqrt{\bar{\tau}}S-\frac{1}{\bar{\tau}}\int_0^{\bar{\tau}}\sum_{i=1}^{2n}{\tau}^{\frac{3}{2}}H(X(\tau),e_i(\tau))d\tau.
        \end{aligned}
    \end{equation}
    We next compute $\sum_{i=1}^{2n}H(X,e_i)$. Invoking \eqref{expression of H}, we can write
    \begin{equation}
        \begin{aligned}
            \sum_{i=1}^{2n}H(X,e_i)=&-\Delta_{g(\tau)}S+2Ric(X,X)-4X\cdot S+4\sum_{i=1}^{2n}\nabla_{e_i}Rc(e_i,X)\\
            &-2\sum_{i=1}^{2n}Rc_{\tau}(e_i,e_i)+2|Rc|^2-\frac{1}{\tau}S.
        \end{aligned}
    \end{equation}
We have 
\begin{equation} S_\tau=\partial_\tau(g^{ij}Rc_{ij})=-2Rc^{ij}Rc_{ij}+g^{ij}\partial_\tau Rc_{ij}=-2|Rc|^2+tr_g(\partial_\tau Rc).
\end{equation}
On the other hand, we know that the complex Chern scalar curvature (defined by $tr_\omega(Ric^C)$) satisfies $S_C=\frac{1}{2}S$, $|Ric^C|^2=\frac{1}{2}|Rc|^2$ and
$$
\partial_\tau S_C = -2\Delta_{\omega(\tau)} S_C - 2|Ric^C|^2,
$$
where $\Delta_{\omega(\tau)}$ is the Chern-Laplacian. This gives that
\begin{equation}\label{partial_tau S}
\partial_\tau S = -2\Delta_{\omega(\tau)} S - 2|Rc|^2_g.
\end{equation}
Consequently,
\begin{equation}
    \sum_{i=1}^{2n}Rc_{\tau}(e_i,e_i)=tr_g(\partial_\tau Rc)=S_\tau+2|Rc|^2=-2\Delta_{\omega(\tau)}S=-\Delta_{g(\tau)}S+\langle\nabla S,\theta^\#\rangle.
\end{equation}
Using \cref{LY17 1} and the contracted second Bianchi identity for the Riemannian Ricci curvature $Ric$, we can write
\begin{equation}
    \begin{aligned}
+4\sum_{i=1}^{2n}\nabla_{e_i}Rc(e_i,X)&=4\sum_{i=1}^{2n}\nabla_{e_i}(Ric+\mathcal{E}(\mathcal{T}))(e_i,X)\\
&=2\langle\nabla R,X\rangle+4\sum_{i=1}^{2n}\nabla_{e_i}\mathcal{E}(\mathcal{T})(e_i,X),
    \end{aligned}
\end{equation}
where $R$ is now the Riemannian scalar curvature and $\mathcal{E}(\mathcal{T})$ is the torsion term calculated explicitly in \cref{LY17 1}.  Summing all the terms, we have
\begin{equation}
    \begin{aligned}
\sum_{i=1}^{2n}H(X,e_i)=&\Delta_{g(\tau)}S+2Ric(X,X)-4\langle\nabla S,X\rangle+2\langle\nabla R,X\rangle\\
         &+4\sum_{i=1}^{2n}\nabla_{e_i}\mathcal{E}(\mathcal{T})(e_i,X)-2\langle\nabla S,\theta^\#\rangle+2|Rc|^2-\frac{1}{\tau}S\\
         =&-S_\tau+2Ric(X,X)-4\langle\nabla S,X\rangle+2\langle\nabla R,X\rangle\\         
&+4\sum_{i=1}^{2n}\nabla_{e_i}\mathcal{E}(\mathcal{T})(e_i,X)-\langle\nabla S,\theta^\#\rangle-\frac{1}{\tau}S,
    \end{aligned}
\end{equation}
where in the last equality we have used \eqref{partial_tau S}. Recall from \eqref{def of H and K} that
$$
  H(X):=-S_\tau-\frac{1}{\tau}S-2\langle\nabla S,X\rangle+2Rc(X,X),
$$
Therefore, \cref{LY17 1} and \cref{LY17 2} (see also \cref{rmk:torsion bound}) yields that
\begin{equation}\label{H-H}
    \begin{aligned}
        &\left|H(X)-\sum_{i=1}^{2n}H(X,e_i)\right|&\\
        =&\left|2Ric(X,X)-2Rc(X,X)-2\langle\nabla S,X\rangle+2\langle\nabla R,X\rangle+4\sum_{i=1}^{2n}\nabla_{e_i}\mathcal{E}(\mathcal{T})(e_i,X)-\langle\nabla S,\theta^\#\rangle\right|\\
        \leq&\mathcal{E}(\mathcal{T})|X|^2+\mathcal{E}(\mathcal{T})|X|+|\nabla S|\cdot|\theta^\#|.
    \end{aligned}
\end{equation}
Putting \eqref{H-H} into \eqref{Lap of L 1}, we get
\begin{equation}
    \begin{aligned}
           \Delta_{g(\bar{\tau})}L\leq&\frac{2n}{\sqrt{\bar{\tau}}}-2\sqrt{\bar{\tau}}S-\frac{1}{\bar{\tau}}\int_0^{\bar{\tau}}\tau^{\frac{3}{2}}H(X(\tau))d\tau\\
           &+\frac{1}{\bar{\tau}}\int_0^{\bar{\tau}}\tau^{\frac{3}{2}}\left(\mathcal{E}(\mathcal{T})|X|^2+\mathcal{E}(\mathcal{T})|X|+|\nabla S|\cdot|\theta^\#|\right)d\tau\\
           \leq&\frac{2n}{\sqrt{\bar{\tau}}}-2\sqrt{\bar{\tau}}S-\frac{1}{\bar{\tau}}K+C+C\int_0^{\bar{\tau}}\tau^{\frac{1}{2}}(|X|^2+|X|)d\tau\\
           \leq&\frac{2n}{\sqrt{\bar{\tau}}}-2\sqrt{\bar{\tau}}S-\frac{1}{\bar{\tau}}K+C+C(T)\int_0^{\bar{\tau}}\tau^{\frac{1}{2}}(S+|X|^2+|X|)d\tau\\
            \leq&\frac{2n}{\sqrt{\bar{\tau}}}-2\sqrt{\bar{\tau}}S-\frac{1}{\bar{\tau}}K+C+C(T)L,
    \end{aligned}
\end{equation}
where we have used \cref{Scalar curvature bound} and the fact that all the torsion terms remain uniformly bounded in our case.
\end{proof}
As in \cite{Per02}, setting $\bar{L}(q,\tau):=2\sqrt{\tau}L(q,\tau)$, we then have the following evolution estimate of $\bar{L}$:

\begin{lemma}\label{lem:evolution of L bar}
    There is a uniform constant $C$ such that
    $$
\left(\frac{\partial}{\partial\bar{\tau}}+\Delta_{g(\bar{\tau})}\right)\bar{L}(\cdot,\bar{\tau})\leq C(\bar{L}(\cdot,\bar{\tau})+1).
    $$
\end{lemma}
\begin{proof}
    Combining \eqref{L_tau 1} and \cref{prop:estimate of Delta L}, we have
    \begin{align*}
        \bar{L}_{\bar{\tau}}+\Delta_{g(\bar{\tau})}\bar{L}&\leq 4\bar{\tau}S(q)+\frac{2}{\sqrt{\bar{\tau}}}K+4n-4\bar{\tau}S(q)-\frac{2}{\sqrt{\bar{\tau}}}K+2\sqrt{\bar{\tau}}C+C(T)\sqrt{\bar{\tau}}L\\
        &\leq 4n+C+C(T)\bar{L}\leq C(\bar{L}+1),
    \end{align*}
    the proof is finished.
\end{proof}

\subsection{Almost monotonicity of the reduced volume}

In this subsection, we calculate the derivatives of the reduced volume $V(\tau)$, which we recall is defined by
\begin{align*}
    V(\tau):=\int_X\tau^{-n}e^{-l(q,\tau)}dq=\int_{T_pM}\tau^{-n}e^{-l(\mathcal{L}\exp_\tau(v),\tau)}\mathcal{J}(v,\tau)\mathds{1}_{\Omega_\tau}dv,
\end{align*}
where $\Omega_\tau$ is the directions $v\in T_pM$ such that $\mathcal{L}\exp_\tau(v)\notin B_\tau$, the time-$\tau$ $\mathcal{L}$-cut locus and $\mathcal{J}(v,\tau)$ is the Jacobian of the differential of the $\mathcal{L}$-exponential map. By \cref{cut locus has measure zero}, the change of variable formula above is valid. We are now going to calculate the derivatives of
\begin{align*}
    \log\left(\tau^{-n}e^{-l(\mathcal{L}\exp_\tau(v),\tau)}\mathcal{J}(v,\tau)\right)=-n\log(\tau)-l(\mathcal{L}\exp_\tau(v),\tau)+\log\mathcal{J}(v,\tau).
\end{align*}
\begin{lemma}\label{lem:almost monotonicity}
    We have the following almost monotonicity estimate of the integrand of the reduced volume:
    $$
\mathcal{J}(v,\tau_1)\geq\mathcal{J}(v,\tau_2)\left(\frac{\tau_1}{\tau_2}\right)^{-n}e^{l(\gamma(\tau_1),\tau_1)-l(\gamma(\tau_2),\tau_2)}\cdot e^{ -C \int_{\tau_1}^{\tau_2} l(\gamma(\tau), \tau) d\tau -C(\tau_2-\tau_1)},
    $$
    where $v\in T_pM$, $0\leq\tau_1<\tau_2\leq\bar{\tau}$ and $C$ is a uniform constant.
\end{lemma}
\begin{proof}
    Let $\gamma$ be an $\mathcal{L}$-geodesic starting from $p$ with initial vector $v\in T_pM$, then by definition we have $\mathcal{L}\exp_\tau(v)=\gamma(\tau)$. It follows from \eqref{equality of K} and the definition of $L$ that
    \begin{equation}\label{derivative of l}
    \begin{aligned}
        \left.\frac{d}{d\tau}\right|_{\tau=\bar{\tau}}l(\gamma(\tau),\tau)= &\left.\frac{d}{d\tau}\right|_{\tau=\bar{\tau}}\left(\frac{1}{2\sqrt{\tau}}L(\gamma(\tau),\tau)\right)=\frac{1}{2}\left(S(\gamma(\bar{\tau}))+|X(\bar{\tau})|^2\right)-\frac{1}{4\bar{\tau}^{\frac{3}{2}}}L\\
        =&\frac{1}{\bar{\tau}^{\frac{3}{2}}}\left(-\frac{1}{2}K+\frac{1}{4}L-\frac{1}{4}L\right)=-\frac{1}{2\bar{\tau}^{\frac{3}{2}}}K.
        \end{aligned}
    \end{equation}
  For $v\in T_pM$, choose a real orthonormal basis $\{w_i\}_{i=1}^{2n}$ of $T_vT_pM=T_pM$ and consider the differential map
  $$
d(\mathcal{L}\exp_\tau)_v:T_vT_pM\to T_{\gamma(\tau)}M.
  $$
  
    Set $Y_i(\tau):=d(\mathcal{L}\exp_\tau)_v(w_i)$. It is easy to see that $\left.Y_i(\tau)=\frac{\partial}{\partial s}\right|_{s=0}\tilde{\gamma}(s,\tau)$ for 
    $$
\tilde{\gamma}(s,\tau):=\mathcal{L}\exp_\tau(v+sw_i),
    $$
    which is clearly an $\mathcal{L}$-geodesic variation of $\gamma$ when the points do not meet the $\mathcal{L}$-cut locus $B_\tau$. It follows from \cref{prop:L jacobi field derive} that  $\{Y_i\}_{i=1}^{2n}$ is a basis of the $\mathcal{L}$-Jacobi fields along $\gamma$ with $Y_i(0)=0$. Without loss of generality, we shall also assume that $\{Y_i(\bar{\tau})\}_{i=1}^{2n}$ forms an orthonormal basis of $T_{\gamma{(\bar{\tau}})}M$. Indeed, we can first choose an orthonormal basis $\{e_i\}_{i=1}^{2n}$ of $T_{\gamma{(\bar{\tau}})}M$ and set $w_i:=(d(\mathcal{L}\exp_{\bar{\tau}})_v)^{-1}(e_i)$.

    It is now well known from linear algebra that 
    \begin{equation}\label{J^2}
    \mathcal{J}(v,\tau)^2=\det\left[\left(d(\mathcal{L}\exp_\tau)_v\right)^*d(\mathcal{L}\exp_\tau)_v\right]=\lambda \det(S(\tau)),
    \end{equation} 
    where $S_{ij}(\tau):=\langle Y_i(\tau),Y_j(\tau)\rangle$ is the Gram matrix. Consequently, we can write
    \begin{align*}
        \frac{d}{d\tau}\log \mathcal{J}(v,\tau)^2=\frac{d}{d\tau}\log\det(S(\tau))=tr\left(S^{-1}\frac{dS}{d\tau}\right).
    \end{align*}
    Since we have that $S(\bar{\tau})=I_{2n}$,
    \begin{equation}\label{log j 1}
         \left.\frac{d}{d\tau}\right|_{\tau=\bar{\tau}}\log \mathcal{J}(v,\tau)=\left.\frac{1}{2}\sum_{i=1}^{2n}\frac{d|Y_i|^2}{d\tau}\right|_{\tau=\bar{\tau}}.
    \end{equation}
    Since $Y_i$ are $\mathcal{L}$-Jacobi fields, invoking \eqref{Hess of L} we can write
    \begin{align*}
       \left. \frac{d|Y_i|^2}{d\tau}\right|_{\tau=\bar{\tau}}=2Rc(Y_i(\bar{\tau}),Y_i(\bar{\tau}))+\frac{1}{\bar{\tau}}Hess_L(Y_i(\bar{\tau}),Y_i(\bar{\tau})).
    \end{align*}
    Let $\tilde{Y}_i(\tau)$ be defined as in \eqref{construction of tilde Y} with $\tilde{Y}_i(\bar{\tau})=Y_i(\bar{\tau})$ and let $e_i(\tau):=\left(\frac{\bar{\tau}}{\tau}\right)^{\frac{1}{2}}\tilde{Y}_i(\tau)$, which form an orthonormal basis for each $\tau$ by \cref{onb of e_i}. Then \cref{lem:estimate of Hess} yields that
    \begin{align*}
        \frac{1}{\sqrt{\bar{\tau}}}Hess_L(Y_i(\bar{\tau}),Y_i(\bar{\tau}))\leq \frac{1}{\bar{\tau}}-2Rc(Y_i(\bar{\tau}),Y_i(\bar{\tau}))-\frac{1}{\sqrt{\bar{\tau}}}\int_0^{\bar{\tau}}\sqrt{\tau}H(X,\tilde{Y}_i)d\tau
    \end{align*}
    and hence
    $$
 \left.\frac{d|Y_i|^2}{d\tau}\right|_{\tau=\bar{\tau}}\leq\frac{1}{\bar{\tau}}-\frac{1}{\sqrt{\bar{\tau}}}\int_0^{\bar{\tau}}\sqrt{\tau}H(X,\tilde{Y}_i)d\tau.
    $$
    Putting this into \eqref{log j 1}, we derive
    \begin{equation}\label{log j 2}
       \left.\frac{d}{d\tau}\right|_{\tau=\bar{\tau}}\log \mathcal{J}(v,\tau)\leq\frac{n}{\bar{\tau}}-\frac{1}{2}\bar{\tau}^{-\frac{3}{2}}\int_0^{\bar{\tau}}\tau^{\frac{3}{2}}\sum_{i=1}^{2n}H(X(\tau),e_i(\tau))d\tau.
    \end{equation}
    We have already computed in \eqref{H-H} that
    \begin{align*}
        \left|H(X)-\sum_{i=1}^{2n}H(X,e_i)\right|\leq\mathcal{E}(\mathcal{T})|X|^2+\mathcal{E}(\mathcal{T})|X|+|\nabla S|\cdot|\theta^\#|.
    \end{align*}
    So, the derivatives of $\log\mathcal{J}$ can finally be controlled by
    \begin{equation}\label{log j 3}
    \begin{aligned}
          \left.\frac{d}{d\tau}\right|_{\tau=\bar{\tau}}\log \mathcal{J}(v,\tau)&\leq\frac{n}{\bar{\tau}}-\frac{1}{2}\bar{\tau}^{-\frac{3}{2}}K+C\bar{\tau}^{-\frac{1}{2}}\int_0^{\bar{\tau}}\tau^{\frac{1}{2}}(|X|^2+|X|+1)d\tau\\
         &\leq\frac{n}{\bar{\tau}}-\frac{1}{2}\bar{\tau}^{-\frac{3}{2}}K+C(l(\gamma(\bar{\tau}),\bar{\tau})+1).
         \end{aligned}
    \end{equation}
    Combining \eqref{derivative of l} and \eqref{log j 3}, we finally arrive at 
    \begin{equation}\label{final derivative}
          \left.\frac{d}{d\tau}\right|_{\tau=\bar{\tau}}\log\left(\tau^{-n}e^{-l(\mathcal{L}\exp_\tau(v),\tau)}\mathcal{J}(v,\tau)\right)\leq C(l(\mathcal{L}\exp_{\bar{\tau}}(v),\bar{\tau})+1).
    \end{equation}
  For convenience, we denote the integrand of the reduced volume by$$W(v, \tau) := \tau^{-n}e^{-l(\mathcal{L}\exp_\tau(v),\tau)}\mathcal{J}(v,\tau).$$Then the differential inequality \eqref{final derivative} can be rewritten as
  \begin{equation}
  \frac{d}{d\tau}\log W(v, \tau) \leq C l(\mathcal{L}\exp_\tau(v),\tau).
  \end{equation}
  By integrating this differential inequality along the $\mathcal{L}$-geodesic $\gamma(\tau)$ from $0\leq\tau_1$ to $\tau_2\leq\bar{\tau}$, we obtain
  \begin{equation}\label{integral of W}W(v, \tau_2) \leq W(v, \tau_1) \exp\left( C \int_{\tau_1}^{\tau_2} l(\gamma(\tau), \tau) d\tau +C(\tau_2-\tau_1)\right).
  \end{equation}
This yields that
\begin{align*}
    \mathcal{J}(v,\tau_1)\geq\mathcal{J}(v,\tau_2)\left(\frac{\tau_1}{\tau_2}\right)^{-n}e^{l(\gamma(\tau_1),\tau_1)-l(\gamma(\tau_2),\tau_2)}\cdot e^{ -C \int_{\tau_1}^{\tau_2} l(\gamma(\tau), \tau) d\tau -C(\tau_2-\tau_1)},
\end{align*}
and hence the proof is concluded.
\end{proof}
\begin{remark}
    In applications, we will have $l(\gamma(\tau),\tau)\leq\frac{C}{\tau}$ and $\tau_1\geq C^{-1}\bar{\tau}$ (see \cref{lem:L-geodesic-exit-time} below), so the above torsion terms will not cause essential problems.
\end{remark}

\section{Reduction to a key estimate}

We have now established various a priori estimates and techniques to run the arguments in \cite{LTZ26}. Recall that we have a holomorphic birational map $f$ from $X$ to $Y=X_{can}$, the canonical model of $X$. We let $(\omega_{can},d_{can})$ be the K\"ahler-Einstein current and the associated distance on $Y$. We also denote $(\omega_Y,d_Y)$ to be the restriction of the Fubini-Study metric from a projective space to $Y$ and its corresponding distance function. The following important bi-H\"older equivalence of $d_Y$ and $d_{can}$ was established very recently in \cite{LTZ26}:
\begin{theorem}\cite[Theorem 3.6]{LTZ26} \label{Holder of d_can}
    There are constants $C>0$ and $\alpha\in(0,1]$ such that for all $x,y\in Y$, we have
    $$
C^{-1}d_Y(x,y)\leq d_{can}(x,y)\leq d_Y(x,y)^\alpha.
    $$
\end{theorem}

In the sequel, we will fix a small constant $\varepsilon>0$ and set 
$$
V_\varepsilon:=\{y\in Y\;|\;d_{can}(E,y)<\varepsilon\},\quad\tilde{V}_\varepsilon:=f^{-1}(V_\varepsilon).
$$
The notation $(Y\setminus V_\varepsilon,d_{can})$ will always mean the restricted metric $d_{can}|_{Y\setminus V_\varepsilon}$, and similarly for $(X\setminus\tilde{V}_\varepsilon,d_t)$. Following the notations used in a series works of Cheeger-Colding, we will use the notation $\Psi(\alpha,\beta,\gamma|A,B)$ to denote a number depending on $\alpha,\beta,\gamma,A$ and $B$ such that $\Psi(\alpha,\beta,\gamma|A,B)\to0$ as $\alpha,\beta,\gamma\to0$ while $A$ and $B$ are fixed.

\begin{lemma}\label{LTZ lemma 4.1}
    We have
    \begin{equation}\label{dt distance}
        d_{GH}((X,d_t),(X\setminus\tilde{V}_\varepsilon,d_t))=\Psi(\varepsilon,t^{-1}),
    \end{equation}
    and
    \begin{equation}\label{d_can distance}
           d_{GH}((Y,d_{can}),(Y\setminus{V}_\varepsilon,d_{can}))=\Psi(\varepsilon).
    \end{equation}
\end{lemma}
\begin{proof}
    We only give a proof of \eqref{dt distance} as the proof of \eqref{d_can distance} is similar using the volume non-collapsing estimate in \cite[Theorem 3.5]{GPSS23}. 

    For each $0<\varepsilon<r_0$, where $r_0$ is the uniform constant in \cref{volume non-collapsing estiamte}, it suffices to show that there is a neighborhood $W$ of $E$ with smooth boundary such that for any $x\in W$, we have $d_t(x,\partial W)<\varepsilon$ for all large $t$. Recall that from \cref{estimates of varphi} we have $\omega_t^n\leq C\Omega$ for some uniform constant $C$ and a fixed volume form $\Omega$ on $X$, so we can choose $W$ so close to $E$ that 
$$
\operatorname{Vol}_{g_t}(W)<c\left(\frac{\varepsilon}{2}\right)^\alpha,\quad \forall t>0,
$$
where $\alpha,c$ are the uniform constants in \cref{volume non-collapsing estiamte}. 
On the other hand, fix $x\in W$ and suppose that $d_t(x,\partial W)\geq\varepsilon$ for some $t>0$. It follows that the geodesic ball $B_{g_t}(x,\frac{\varepsilon}{2})\subset W$, the volume non-collapsing estimate \cref{volume non-collapsing estiamte} then yields that
\begin{align*}
    c\left(\frac{\varepsilon}{2}\right)^\alpha\leq\operatorname{Vol}_{g_t}\left(B_{g_t}\left(x,\frac{\varepsilon}{2}\right)\right)\leq\operatorname{Vol}_{g_t}(W)<c\left(\frac{\varepsilon}{2}\right)^\alpha,
\end{align*}
a contradiction.
\end{proof}

As in \cite{LTZ26}, \cref{main thm:uniqueness} can be reduced to establish an upper bound of $d_{can}$ in terms of $d_t$:

\begin{lemma}
    If for any given $\varepsilon>0$ and $p,q\in X\setminus\tilde{V}_\varepsilon$, we have
    \begin{equation}\label{eq:4.10}
        d_{can}(f(p),f(q))\leq d_t(p,q)+\Psi(t^{-1}|\varepsilon).
    \end{equation}
    Then, $(X,d_t)$ converges to $(Y,d_{can})$ in the Gromov-Hausdorff topology.
\end{lemma}
\begin{proof}
    By the smooth convergence \cref{smooth convergence} and \cref{LTZ lemma 4.1}, it is not hard to see that for each $p,q\in X\setminus\tilde{V}_\varepsilon$, we have
    $$
d_t(p,q)\leq d_{can}(f(p),f(q))+\Psi(t^{-1}|\varepsilon),
    $$
    this combined with our condition \eqref{eq:4.10} yields that
    $$
|d_t(p,q)-d_{can}(f(p),f(q))|\leq\Psi(t^{-1}|\varepsilon),
    $$
    and we get the desired Gromov-Hausdorff approximation.
\end{proof}
Now, everything is reduced to show \eqref{eq:4.10} in our case. Recall the reparametrization \eqref{backward CRF}, and here we write as
\begin{equation}
    \partial_\tau\tilde{g}=2Ric^C(\tilde{g}(\tau)),\quad\tau\in[0,\bar{\tau}],
\end{equation}
to distinguish notations. Since it depends on the time $T>>1$ ($\tilde{g}(0)=g(T)$), we will use $L_T(q,\bar{\tau})$ to denote the $\mathcal{L}$-length of the minimizing $\mathcal{L}$-geodesic from $p$ to $q$ in the sequel. Also, since we always consider parameters
\begin{equation}
    T>>1>>\bar{\tau}>0,
\end{equation}
in this convention we have the uniform Chern Scalar curvature bound
\begin{equation}\label{eq:chern scalar curvature bound}
    \sup_X\left|S_C(\tilde{g}(\tau))\right|\leq\frac{C}{1-2\bar{\tau}}\leq C,
\end{equation}
by \cref{Scalar curvature bound} and the fact that $\bar{\tau}\ll 1$. As in \cite{LTZ26}, we need the following two estimates of the upper and lower bound of $L_T$:
\begin{lemma}\label{lem:upper bound of L}
    Fix $p\in X\setminus\tilde{V}_\varepsilon$ and let $L_T(q,\bar{\tau})$ be the $\mathcal{L}$-distance from $p$ to $q$. Then for each $q\in X\setminus\tilde{V}_\varepsilon$, we have
    \begin{equation}\label{eq:uper bound of L}
        L_T(q,\bar{\tau})\leq\frac{1}{2\sqrt{\bar{\tau}}}d_{can}(f(p),f(q))^2+\Psi(T^{-1}|\varepsilon,\bar{\tau})+\Psi(\bar{\tau}|\varepsilon).
    \end{equation}
\end{lemma}
\begin{proof}
    Taking into account the Chern scalar curvature bound \eqref{eq:chern scalar curvature bound} and the smooth convergence \cref{smooth convergence}, the proof of \cite[Page 18, equation (4.29)]{LTZ26} can be carried out word by word. Note that the term $\Psi(T^{-1}|\varepsilon,\bar{\tau})$ can essentially be written as $\Psi(T^{-1}|\varepsilon)$ here, as it cannot blowup as $\bar{\tau}\to0$.
\end{proof}

The hard part is to establish the lower bound of $L_T$:

\begin{proposition}\label{prop:key estimate}
    Fix $p\in X\setminus\tilde{V}_\varepsilon$ and let $L_T(q,\bar{\tau})$ be the $\mathcal{L}$-distance from $p$ to $q$. Then for each $q\in X\setminus\tilde{V}_\varepsilon$, we have
    \begin{equation}\label{key inequality}
        L_T(q,\bar{\tau})\geq\frac{1}{2\sqrt{\bar{\tau}}}d_{can}(f(p),f(q))^2+\Psi(T^{-1}|\varepsilon,\bar{\tau})+\Psi(\bar{\tau}|\varepsilon).
    \end{equation}
\end{proposition}
Accepting this proposition, we show that it implies \cref{main thm:uniqueness}:
\begin{theorem}
    \cref{prop:key estimate} implies \cref{main thm:uniqueness}.
\end{theorem}
\begin{proof}
    Set $\bar{L}_T(q,\bar{\tau})=2\sqrt{\bar{\tau}}L_T(q,\bar{\tau})$, it follows from \cref{lem:evolution of L bar} that 
     \begin{equation}\label{eq:evolution of bar L}
\left(\frac{\partial}{\partial\bar{\tau}}+\Delta_{\tilde{g}(\bar{\tau})}\right)\bar{L}_T(\cdot,\bar{\tau})\leq C(\bar{L}_T(\cdot,\bar{\tau})+1).
    \end{equation}
    Meanwhile, \cref{prop:key estimate} gives that
    \begin{equation}\label{eq:key estimate for bar L}
        \bar{L}_T(q^\prime,\bar{\tau})\geq d_{can}(f(p),f(q))^2+\Psi(T^{-1}|\varepsilon,\bar{\tau})+\Psi(\bar{\tau}|\varepsilon)+\Psi(\delta|\varepsilon),
    \end{equation}
for any $q^\prime\in f^{-1}\left(B^{d_{can}}(f(q),\delta)\right)$ with $\delta\ll \varepsilon$. Let $\chi$ be a smooth cutoff function on $Y$ supported in $B^{d_{can}}(f(q),2\delta)$ such that $\chi|_{B^{d_{can}}(f(q),\delta)}\equiv1$. It is easy to see that there exists a constant $C(\varepsilon)$ such that $|\nabla\chi|_{g_{can}}\leq C(\varepsilon)\delta^{-1}$ and $-C(\varepsilon)\omega_Y\leq\delta^2\sqrt{-1}\partial\bar{\partial}\chi\leq C(\varepsilon)\omega_Y$, which in turn combined with the parabolic Schwarz lemma \cref{parabolic Schwarz} implies the Chern-Lapalcian bound of $f^*\chi$:
$$
\sup_X\left|\Delta^C_{\tilde{\omega}(\tau)}(f^*\chi)\right|=\sup_X\left|tr_{\tilde{\omega}(\tau)}\left(\sqrt{-1}\partial\bar{\partial}f^*\chi\right)\right|\leq C(\varepsilon)\delta^2.
$$
By \cref{real and complex Laplacian}, we also have the Riemannian Laplacian bound
\begin{equation}\label{eq:lap of chi}
    \sup_X\left|\Delta_{\tilde{g}(\tau)}(f^*\chi)\right|\leq2\sup_X\left|\Delta^C_{\tilde{\omega}(\tau)}(f^*\chi)\right|+\sup_X|\nabla\chi|_{g_{can}}|\theta^\#|_{g_{can}}\leq C(\varepsilon)\delta^{-2}.
\end{equation}
Now, multiplying both sides of \eqref{eq:evolution of bar L} by $f^*\chi$ and integrating we obtain
\begin{equation}
    \begin{aligned}
&\int_0^{\bar{\tau}}\int_X(f^*\chi)\partial_\tau\bar{L}_T(\cdot,\tau)\tilde{\omega}(\tau)^nd\tau+
\int_0^{\bar{\tau}}\int_X(f^*\chi)\Delta_{\tilde{g}(\tau)}\bar{L}_T(\cdot,\tau)\tilde{\omega}(\tau)^nd\tau\\
\leq&\int_0^{\bar{\tau}}\int_XC(\bar{L}_T(\cdot,\tau)+1)\tilde{\omega}(\tau)^nd\tau.
    \end{aligned}
\end{equation}
An integration by parts using $\partial_\tau(f^*\chi)$ and $\partial_\tau\tilde{\omega}(\tau)^n=2S_C(\tilde{g}(\tau))\tilde{\omega}(\tau)^n$ yields that
\begin{equation}\label{eq:int by parts of L}
    \begin{aligned}
       & \int_X(f^*\chi)\bar{L}_T(\cdot,\bar{\tau})\tilde{\omega}(\tau)^n- \int_X(f^*\chi)\bar{L}_T(\cdot,0)\tilde{\omega}(0)^n+\int_0^{\bar{\tau}}\int_X\Delta_{\tilde{g}(\tau)}(f^*\chi)\bar{L}_T(\cdot,\tau)\tilde{\omega}(\tau)^nd\tau\\
        \leq&\int_0^{\bar{\tau}}\int_XC(f^*\chi)(\bar{L}_T(\cdot,\tau)+1)\tilde{\omega}(\tau)^nd\tau+2\int_0^{\bar{\tau}}\int_X(f^*\chi)\bar{L}_T(\cdot,\tau)S_C(\tilde{g}(\tau))\tilde{\omega}(\tau)^n\\
        \leq& C\int_0^{\bar{\tau}}\int_X(f^*\chi)\left[d_{can}(f(p),f(q))^2+\Psi(T^{-1}|\varepsilon,\bar{\tau})+\Psi(\bar{\tau}|\varepsilon)+\Psi(\delta|\varepsilon)\right]\tilde{\omega}(\tau)^nd\tau\\
        \leq&\Psi(\bar{\tau}|\delta,\varepsilon)\int_X(f^*\chi)\tilde{\omega}(\tau)^n,
    \end{aligned}
\end{equation}
where in the second inequality we have used the upper bound of $\bar{L}_T$ \cref{lem:upper bound of L} and the Chern scalar curvature bound \eqref{eq:chern scalar curvature bound}. Applying the lower bound estimate \eqref{eq:key estimate for bar L} to the term $\int_X(f^*\chi)\bar{L}_T(\cdot,\bar{\tau})\tilde{\omega}(\tau)^n$ and the Riemannian Laplacian estimate \eqref{eq:lap of chi} to the term $\int_0^{\bar{\tau}}\int_X\Delta_{\tilde{g}(\tau)}(f^*\chi)\bar{L}_T(\cdot,\tau)\tilde{\omega}(\tau)^nd\tau$, we can further write
\begin{equation}\label{eq:control of dcan 1}
  \left[d_{can}(f(p),f(q))^2+\Psi(T^{-1}|\varepsilon,\bar{\tau})+\Psi(\bar{\tau}|\varepsilon,\delta)\right]\int_X(f^*\chi)\tilde{\omega}(\tau)^n\leq  \int_X(f^*\chi)\bar{L}_T(\cdot,0)\tilde{\omega}(0)^n.
\end{equation}
We have $\tilde{\omega}(0)=\omega(T)$, and basic calculations using the Cauchy-Schwarz inequality give
\begin{equation}\label{eq:bar L tau to 0}
    \lim_{\tau\to0^+}\bar{L}_T(q,\tau)=d_{\tilde{g}(0)}(p,q)^2=d_{g(T)}(p,q)^2.
\end{equation}
Therefore, we can continue to estimate the right-hand side of \eqref{eq:control of dcan 1}:
\begin{equation}\label{eq:key intermediate calculation}
    \int_X(f^*\chi)\bar{L}_T(\cdot,0)\tilde{\omega}(0)^n= \int_X(f^*\chi)d_{T}(p,\cdot)\tilde{\omega}(0)^n\leq\left(d_T(p,q)+\Psi(T^{-1},\delta|\varepsilon)\right)^2\int_X(f^*\chi)\tilde{\omega}(0)^n.
\end{equation}
On the other hand, from $\partial_\tau\tilde{\omega}(\tau)^n=2S_C(\tilde{g}(\tau))\tilde{\omega}(\tau)^n$, we know that 
\begin{equation}\label{omega tilde}
\tilde{\omega}(\tau)^n\geq\left(1-\bar{\tau}2S_C\right)\tilde{\omega}(0)^n\geq\left(1-\Psi(\bar{\tau})\right)\tilde{\omega}(0)^n,\quad\forall\tau\in[0,\bar{\tau}].
\end{equation}
Putting \eqref{eq:key intermediate calculation} and \eqref{omega tilde} into \eqref{eq:control of dcan 1}, we can write
\begin{equation}\label{eq:4.10 another version}
     \left[d_{can}(f(p),f(q))^2+\Psi(T^{-1}|\varepsilon,\bar{\tau})+\Psi(\bar{\tau}|\varepsilon,\delta)\right](1-\Psi(\bar{\tau}))\leq\left(d_T(p,q)+\Psi(T^{-1},\delta|\varepsilon)\right)^2.
\end{equation}
Based on \eqref{eq:4.10 another version}, \eqref{eq:4.10} is achieved by first choosing $\delta$ small enough, then $\bar{\tau}$ sufficiently small, and finally $T^{-1}$ sufficiently small. The proof is thus completed.
\end{proof}

\section{An almost avoidance principle}

In this section, we have to extend Lee-Tosatti-Zhang's almost avoidance principle to our Hermitian case. As in \cite{LS22}, the metric $\varepsilon$-regular set of $Y$ is defined to be
\begin{equation}\label{eq:epsilon regular set}
    \mathcal{R}_\varepsilon:=\left\{y\in Y,\lim_{r\to0^+}\frac{\operatorname{Vol}(B^{d_{can}}(y,r))}{\omega_{2n}r^{2n}}>1-\varepsilon\right\},
\end{equation}
where $\omega_{2n}$ is the volume of the unit ball in $\mathbb{C}^n$. By the continuity of volume functions we know that $\mathcal{R}_\varepsilon$ is open and that $\mathcal{R}_{\varepsilon^\prime}\subset\mathcal{R}_\varepsilon$ if $\varepsilon^\prime<\varepsilon$. It was proved in \cite[Theorem 5.4, Lemma 5.5]{LTZ26} that one can fix $\varepsilon_1$ sufficiently small such that
\begin{equation}\label{eq:Lp integrability of quotient}
    \mathcal{R}_{\varepsilon_1}\subset Y^{reg},\quad\frac{\omega_{can}^n}{\omega_Y^n}\in L^{p_0}_{loc}(\mathcal{R}_{\varepsilon_1},\omega_Y^n).
\end{equation}
Then we decompose $D\subset Y$ into a disjoint union of metric singular set and algebraic singular set:
\begin{equation}\label{eq:decomposition of D}
    D=(D\cap(Y\setminus\mathcal{R}_{\varepsilon_1}))\cup(D\cap\mathcal{R}_{\varepsilon_1}):=D_1\cup(D\cap\mathcal{R}_{\varepsilon_1}).
\end{equation}
Now, we set
\begin{equation}
    U_\eta:=\{y\in Y,d_{can}(y,D_1)<\eta\},
\end{equation}
for $0<\eta\leq\eta_0$ such that $U_{\eta_0}\cap\left(B^{d_{can}}(f(p),2\delta)\cup B^{d_{can}}(f(q),2\delta)\right)$, where $p,q\in\tilde{V}_\varepsilon$ and $\delta\ll \varepsilon$. This $\varepsilon$ will be fixed throughout the discussion.

It was also illustrated in \cite{LTZ26} that by an extension of Cheeger-Naber's density estimate of metric regular sets \cite{CN13} to $\operatorname{RCD}$ spaces (see \cite{Szé25, ABS19}), we can derive a Minkowski dimension estimate of $D_1$: $\operatorname{dim}_{\mathcal{M}}D_1\leq 2n-2$. That is to say, for any small $\rho$ with $0<\rho<1$, we have
\begin{equation}\label{eq:minkow dim}
    \operatorname{Vol}(U_\eta,\omega_{can}^n)\leq C_\rho\eta^{2-\rho}.
\end{equation}
\begin{remark}\label{rmk:remove RCD}
  Since we are concerning the non-collapsing case here, the general theory of $\operatorname{RCD}$ spaces used for \eqref{eq:minkow dim} could be avoided. We give a sketch of proof here for the reader's convenience.
    
Let $\chi^\prime:=f^*\omega_Y$ be a semi-positive and big form on $X$ and consider a family of complex Monge-Amp\`ere equations
\begin{equation}\label{song MA}
    (\chi^\prime+e^{-t}\omega_X+\sqrt{-1}\partial\bar{\partial}\varphi_t)^n=e^{\varphi_t} f^*\Omega,
\end{equation}
on $X$, where $\omega_X$ is a K\"ahler metric on $X$ and $\Omega$ is a volume form on $Y$ with $\sqrt{-1}\partial\bar{\partial}\log \Omega=\omega_Y$. \eqref{song MA} is also equivalent to 
$$
Ric(\omega^\prime(t))=-\omega^\prime(t)+e^{-t}\omega_X.
$$
It follows from standard a priori estimates that $\varphi_t$ is uniformly bounded, hence it is easy to see that the family of K\"ahler metrics $\omega^\prime(t):=\chi^\prime+e^{-t}\omega_X+\sqrt{-1}\partial\bar{\partial}\varphi_t$ fit into the framework of \cite[Theorem 1.1]{GPSS24a}, which implies the uniform diameter estimates and volume non-collapsing estimates of $\omega^\prime(t)$. Note that it was also proved in \cite{Song14} that $(X,\omega^\prime(t))$ converges in the Gromov-Hausdorff topology to $Y$. Consequently, $Y$ is a non-collapsed Ricci-limit space satisfying all the assumptions in \cite[Theorem 1.3]{CN13}.

It thus remains to show that there is $\rho_0=\rho_0(\varepsilon)>0$ such that $Y\setminus\mathcal{R}_\varepsilon\subset S_{\rho,r}^{2n-2}$ for all $0<\rho<\rho_0$, where $S_{\rho,r}^{2n-2}$ is the $(2n-2)$-th effective singular stratum defined by
$$
S_{\rho,r}^{2n-2}:=\left\{y\in Y|d_{GH}(B^{d_{can}}(y,s),B((0,z^*),s))\geq\rho s,\,\forall\mathbb{R}^{2n-1}\times C(Z),r\leq s\leq1\right\},
$$
as in \cite[Definition 1.2]{CN13}, where $z^*$ is the vertex of the metric cone $C(Z)$. 

If $y\notin S_{\rho,r}^{2n-2}$, then there is $s\geq r$ and a metric cone $C(Z)$ with $\operatorname{diam}Z\leq\pi$ such that $B^{d_{can}}(y,s)$ is $\rho s$-close to $B((0,z^*),s)\subset \mathbb{R}^{2n-1}\times C(Z)$. Rescaling metrics by $s^{-1}$, we have $d_{GH}(B^{s^{-1}d_{can}}(y,1),B((0,z^*),1))<\rho$. Since $Y$ is a Gromov-Hausdorff limit of K\"ahler manifolds with Ricci curvature bounded below, its tangent cones must be K\"ahler cones (cf. Cheeger-Colding-Tian \cite[Theorem 9.1]{CCT02}). This forces $\mathbb{R}^{2n-1}\times C(Z)$ to actually be $\mathbb{R}^{2n}$ by standard K\"ahler cone spilting.

Now, applying Colding's volume convergence theorem \cite{Co97, CC97}, for any $\delta > 0$, we can choose $\rho_0 > 0$ sufficiently small such that if the rescaled ball is $\rho_0$-close to the unit ball in $\mathbb{R}^{2n}$, its volume must satisfy
$$
\frac{\operatorname{Vol}(B^{d_{can}}(y,s))}{\omega_{2n}s^{2n}} > 1 - \delta.
$$
We fix $\delta = \varepsilon/2$. On the other hand, since $y\in Y\setminus\mathcal{R}_\varepsilon$, by definition we have
$$
\lim_{\tau\to 0} \frac{\operatorname{Vol}(B^{d_{can}}(y,\tau))}{\omega_{2n}\tau^{2n}} \leq 1 - \varepsilon.
$$
By the Bishop-Gromov volume comparison on $Y$, the volume ratio $\tau \mapsto \frac{\operatorname{Vol}(B^{d_{can}}(y,\tau))}{\omega_{2n}\tau^{2n}}$ is monotonically non-increasing (this can be proved similarly as in \cite[Proposition 2.3]{CW17}). Thus, for our $s \geq r$, we must have
$$
\frac{\operatorname{Vol}(B^{d_{can}}(y,s))}{\omega_{2n}s^{2n}} \leq \lim_{\tau\to 0} \frac{\operatorname{Vol}(B^{d_{can}}(y,\tau))}{\omega_{2n}\tau^{2n}} \leq 1 - \varepsilon.
$$
This leads to the contradiction $1 - \varepsilon \geq \frac{\operatorname{Vol}(B^{d_{can}}(y,s))}{\omega_{2n}s^{2n}} > 1 - \varepsilon/2$ and hence the proof is finished.
\end{remark}
The following observation due to \cite{LT23, LTZ26} is of key importance, whose proof follows verbatim by using \eqref{eq:chern scalar curvature bound} and \cref{lem:upper bound of L}. 

\begin{lemma}\label{lem:L-geodesic-exit-time}
Suppose $\gamma$ is a minimizing $\mathcal{L}$-geodesic connecting $(p,0)_T$ and $(q',\bar{\tau})_T$, with the endpoint $q'$ satisfying $d_{can}(f(q'), f(q)) < \delta$. Let $\bar{\tau}'$ be the first time when the projected curve $f(\gamma)$ leaves the metric ball $B^{d_{can}}(f(p), \delta)$. Provided that $\bar{\tau}$ is chosen small enough and $T$ is sufficiently large, there exists a uniform constant $C = C(\delta) > 0$ such that
\begin{equation}\label{eq:L-geodesic-exit-time}
    \bar{\tau}' \geq C^{-1}\bar{\tau}.
\end{equation}
\end{lemma}
Following \cite{LTZ26}, we introduce the notion of $\eta$-events:
\begin{definition}
    Let $\gamma$ be a piecewise differentiable curve in $X$, we say $\gamma$ has an $\eta$-event with respect to $U_\eta$ if $f(\gamma)$ enters $U_\eta$ and reaches $U_{\frac{\eta}{2}}$ before returning to the boundary of $U_\eta$.
\end{definition}
We next move on to establish the almost avoidance principle in our case, utilizing \cref{lem:almost monotonicity} and the preparations in \cref{section Per}, the proof of \cite[Proposition 5.3]{LTZ26} can be carried over to our case:
\begin{proposition}\label{prop:almost-avoidance-D1}
Fix $p,q\in X\setminus\tilde{V}_\varepsilon$ again. Let $D_1$ and $U_\eta$ be constructed as above. Fix constants $\sigma>\frac{\rho}{2}$ and consider parameters satisfying $T^{-1} \ll \eta \ll \bar{\tau} \ll \delta\ll\varepsilon$, there exists a measurable subset $\Omega = \Omega(T, \eta, \bar{\tau}, \delta)$ of the ball $B^{d_T}(q, \delta)$ with volume
\begin{equation}\label{eq:vol-omega-D1}
    \operatorname{Vol}(\Omega, \omega(T)^n) \geq \frac{1}{2} \operatorname{Vol}(B^{d_T}(q, \delta), \omega(T)^n),
\end{equation}
such that for any $q' \in \Omega$, there is a minimizing $\mathcal{L}$-geodesic from $(p, 0)_T$ to $(q', \bar{\tau})_T$ which has at most $\eta^{-\sigma}$ $\eta$-events with respect to the family $\{U_\eta\}$.
\end{proposition}
\begin{proof}
Thanks to \cref{cut locus has measure zero}, $\mathcal{L}\exp_{\bar{\tau}}:\Omega_{\bar{\tau}}\to X\setminus B_\tau$ is a diffeomorphism. By slight abusing of notations, we also denote $\Omega_\tau$ its restriction to $\mathcal{L}\exp_{\bar{\tau}}^{-1}(B^{d_T}(q,\delta))$, which is also a diffeomorphism onto its image $B^{d_T}(q,\delta)\setminus B_\tau$. We have $f(B^{d_T}(q,\delta))\subset B^{d_{can}}(f(q),2\delta)$ provided that $T$ is large enough. Let $\Omega_{\bar{\tau}}^\prime\subset\Omega_{\bar{\tau}}$ be the subset of initial vectors whose associated $\mathcal{L}$-geodesic has more than $\eta^{-\sigma}$ $\eta$-events with respect to $U_\eta$ and set $\Omega:=B^{d_T}(q,\delta)\setminus\mathcal{L}\exp_{\bar
\tau}(\Omega_{\bar{\tau}}^\prime)$. It remains to estimate the volume of $\mathcal{L}\exp_{\bar
\tau}(\Omega_{\bar{\tau}}^\prime)$.

For any initial vector $v \in \Omega_{\bar{\tau}}^\prime$, the corresponding $\mathcal{L}$-geodesic $\gamma(\tau) = \mathcal{L}\exp_{p,\tau}(v)$ intersects the pulled-back tubular neighborhood $\tilde{U}_\eta := f^{-1}(U_\eta)$ over a union of open intervals, which we denote by $I_v \subset [0, \bar{\tau}]$. By \cref{lem:L-geodesic-exit-time}, we know that $f(\gamma)$ cannot exit $B^{d_{can}}(f(p), \delta)$ too quickly, which implies that $I_v \subset [C^{-1}\bar{\tau}, \bar{\tau}]$ for some uniform constant $C > 0$.

Consider the measurable subset in the product space $\mathcal{T} = \{(v, \tau) \mid v \in \Omega_{\bar{\tau}}^\prime, \tau \in I_v\} \subset T_p X \times [0, \bar{\tau}]$. By applying Fubini's theorem to the exponential map $\mathcal{L}\exp$, we can estimate the time-integrated volume of $\tilde{U}_\eta$. First, utilizing the almost monotonicity formula from \cref{lem:almost monotonicity} with $\tau_1 = \tau\geq C^{-1}\bar{\tau}$ and $\tau_2 = \bar{\tau}$, and combining it with the uniform upper bound of the $\mathcal{L}$-distance \cref{lem:upper bound of L}, we derive
$$
\int_{\tau}^{\bar
\tau} l(\gamma(\tau), \tau) d\tau\leq\int_{\tau}^{\bar
\tau}\frac{C}{\tau}d\tau=C\log\frac{\bar{\tau}}{\tau}\leq C\log C,
$$
Consequently, we easily obtain the following lower bound for the Jacobian $\mathcal{J}(v, \tau)$:
\begin{equation}\label{eq:jacobian-bound}
    \mathcal{J}(v, \tau) \geq C^{-1} e^{-\frac{C}{\bar{\tau}}} \mathcal{J}(v, \bar{\tau}),\quad\forall\tau\in [C^{-1}\bar{\tau},\bar{\tau}].
\end{equation}
By integrating over the time intervals $I_v$ and applying Fubini's theorem to the exponential map $\mathcal{L}\exp_{\bar{\tau}}$, we can now bound the target volume by the time-integrated volume of the tubular neighborhood $\tilde{U}_\eta$:
\begin{align}\label{eq:fubini-volume}
    \min_{v \in \Omega_{\bar{\tau}}^\prime} |I_v| \operatorname{Vol}(\mathcal{L}\exp_{\bar{\tau}}(\Omega_{\bar{\tau}}^\prime), \tilde{\omega}(\bar{\tau})^n) 
    &= \min_{v \in \Omega_{\bar{\tau}}^\prime} |I_v| \int_{\Omega_{\bar{\tau}}^\prime} \mathcal{J}(v, \bar{\tau}) dv \notag \\
    &\leq \int_{\Omega_{\bar{\tau}}^\prime} \int_{I_v} \mathcal{J}(v, \bar{\tau}) d\tau dv \notag \\
    &\leq C e^{\frac{C}{\bar{\tau}}} \int_{\Omega_{\bar{\tau}}^\prime} \int_{I_v} \mathcal{J}(v, \tau) d\tau dv \notag \\
    &\leq C e^{\frac{C}{\bar{\tau}}} \int_{C^{-1}\bar{\tau}}^{\bar{\tau}} \operatorname{Vol}(\tilde{U}_\eta, \tilde{\omega}(\tau)^n) d\tau.
\end{align}
Next, we establish a lower bound for the time measure $|I_v|$. By the definition of $\Omega_{\bar{\tau}}^\prime$, each curve undergoes at least $\eta^{-\sigma}$ $\eta$-events. Since the $d_{can}$-distance between $\partial U_\eta$ and $\partial U_{\eta/2}$ is at least $c\eta$, the total length contributed by these events is bounded from below by $c\eta^{1-\sigma}$. Applying the Cauchy-Schwarz inequality and the upper bound of the $\mathcal{L}$-length, we estimate:
\begin{align}
    \eta^{1-\sigma} &\leq C \int_{I_v} |\partial_\tau \gamma|_{\tilde{g}(\tau)} d\tau \notag \\
    &\leq C \bar{\tau}^{-\frac{1}{4}} \left( \int_{I_v} \sqrt{\tau} |\partial_\tau \gamma|_{\tilde{g}(\tau)}^2 d\tau \right)^{\frac{1}{2}} |I_v|^{\frac{1}{2}} \notag \\
    &\leq C \bar{\tau}^{-\frac{1}{2}} |I_v|^{\frac{1}{2}},
\end{align}
which directly implies that $|I_v| \geq C^{-1} \bar{\tau} \eta^{2(1-\sigma)}$.

Substituting this lower bound back into \eqref{eq:fubini-volume}, and observing that the metrics $\tilde{\omega}(\tau)$ and $\omega(T)$ are uniformly equivalent up to a constant factor in this small time region, we obtain:
\begin{equation}
    \operatorname{Vol}(\mathcal{L}\exp_{\bar{\tau}}(\Omega_{\bar{\tau}}^\prime), \omega(T)^n) \leq C e^{\frac{C}{\bar{\tau}}} \eta^{-2(1-\sigma)} \operatorname{Vol}(\tilde{U}_\eta, \omega(T)^n).
\end{equation}

Now, the Minkowski content bound \eqref{eq:minkow dim} for $U_\eta$ guarantees that 
$$
\operatorname{Vol}(\tilde{U}_\eta, \omega(T)^n)
\leq C\operatorname{Vol}(X, \omega(T)^n)\operatorname{Vol}(U_\eta, \omega_{can}^n)\leq C \eta^{2-\rho}.
$$
Therefore, the volume of the bad set can be bounded by:
\begin{equation}
    \operatorname{Vol}(\mathcal{L}\exp_{\bar{\tau}}(\Omega_{\bar{\tau}}^\prime), \omega(T)^n) \leq C e^{\frac{C}{\bar{\tau}}} \eta^{2\sigma - \rho} .
\end{equation}

Since we initially fixed $\sigma > \frac{\rho}{2}$, the exponent $2\sigma - \rho$ is strictly positive. Furthermore, the volume non-collapsing estimate \cref{volume non-collapsing estiamte} gives $\operatorname{Vol}(B^{d_T}(q, \delta), \omega(T)^n) \geq c(\delta)$. 

Because our parameters are chosen such that $T^{-1} \ll \eta \ll \bar{\tau} \ll \delta$, we can choose $\eta$ sufficiently small to suppress the $e^{C/\bar{\tau}}$ term. This ensures that the volume of the bad set satisfies:
\begin{equation}
    \operatorname{Vol}(\mathcal{L}\exp_{\bar{\tau}}(\Omega_{\bar{\tau}}^\prime), \omega(T)^n) \leq \frac{1}{2} \operatorname{Vol}(B^{d_T}(q, \delta), \omega(T)^n).
\end{equation}

The proof is therefore completed.
\end{proof}

Having \cref{prop:almost-avoidance-D1} in hand, we now fix this $\eta$ and choosing $\eta_0^\prime\ll \eta$ to define
\begin{equation}
    U^\prime_{\eta^\prime}:=\{y\in Y|d_{can}(y,\overline{D\setminus T_\eta(D_1)}<\eta^\prime),\quad  \forall0<\eta^\prime<\eta_0^\prime.
\end{equation}
It was also shown in \cite{LTZ26} that using \eqref{eq:Lp integrability of quotient} we have the Minkowski dimension estimate of the second part of $D$:
\begin{equation}\label{eq:Mink dim bound for D2}
    \operatorname{Vol}(U_\eta,\omega_{can}^n)\leq C_\rho\eta^{2-\rho},\quad   \forall\rho>0.
\end{equation}
In practice, we will first choose an appropriate $\sigma$ and then choose $\rho$ small enough. Repeating the proof of \cref{prop:almost-avoidance-D1} gives the following result analogous to \cite[Proposition 5.6]{LTZ26}:
\begin{proposition}\label{prop:almost aviodable:final}
    Let $\Omega$ and $\eta$ be fixed as in \cref{prop:almost-avoidance-D1}. For parameters satisfying $T^{-1} \ll \eta^\prime \ll\eta\ll \bar{\tau} \ll \delta\ll\varepsilon$ with $\eta^\prime < \eta_0^\prime$ and constants $\sigma>\frac{\rho}{2}$, there exists a measurable subset $\Omega^\prime \subset \Omega$ with volume
\begin{equation}\label{eq:vol-omega-D2}
    \operatorname{Vol}(\Omega^\prime, \omega(T)^n) \geq \frac{1}{4} \operatorname{Vol}(B^{d_T}(q, \delta), \omega(T)^n),
\end{equation}
such that for any $q^\prime \in \Omega^\prime$, the unique minimizing $\mathcal{L}$-geodesic from $(p, 0)$ to $(q^\prime, \bar{\tau})$ has at most $(\eta^\prime)^{-\sigma}$ $\eta^\prime$-events with respect to the family $\{U^\prime_{\eta^\prime}\}$, in addition to having at most $\eta^{-\sigma}$ $\eta$-events with respect to $\{U_\eta\}$.
\end{proposition}

\section{End of proof of the key estimate}
Recall that we have fixed $\varepsilon>0$ and $p,q\in X\setminus\tilde{V}_\varepsilon$. Before proving \cref{prop:key estimate}, we must establish a local spatial gradient bound for the reduced length $L_T$. In general, a global gradient estimate for $L_T$ is difficult to achieve without a global Hamilton's matrix Harnack inequality under the Chern-Ricci flow (as done in \cite{Per02, KL08}). However, since we are only concerned with points strictly bounded away from the singular set, a local Lipschitz bound can be achieved via a backward exit-time argument similar to \cref{lem:L-geodesic-exit-time} combined with the ODE of the $\mathcal{L}$-geodesics.

\begin{lemma}\label{lem:local_gradient_L}
    Fix $p,q \in X \setminus \tilde{V}_\varepsilon$ and $\delta \ll \varepsilon$. For any $q^\prime \in B_{d_T}(q, \delta)$ which does not lies in the $\mathcal{L}$-cut locus $B_{\bar{\tau}}$ with respect to $p$. For sufficiently small $T^{-1},\bar{\tau} > 0$, the spatial gradient of the reduced length $L_T$ satisfies
    \begin{equation}
        |\nabla L_T(q^\prime, \bar{\tau})|_{\tilde{g}(\bar{\tau})} \le \frac{C(\varepsilon)}{\sqrt{\bar{\tau}}},
    \end{equation}
    where $C(\varepsilon)$ is a uniform constant depending on $\varepsilon$ but independent of $\bar{\tau}$ and $T$.
\end{lemma}

\begin{proof}
    By the first variation formula of the $\mathcal{L}$-length, the spatial gradient is precisely given by (see \eqref{nabla L})
    \begin{equation}
        |\nabla L_T(q^\prime, \bar{\tau})| = 2\sqrt{\bar{\tau}}|\dot{\gamma}(\bar{\tau})|_{\tilde{g}(\bar{\tau})} = |\hat{X}(\sqrt{\bar{\tau}})|,
    \end{equation}
    where we use the reparametrization $s = \sqrt{\tau}$ and set $\hat{X}(s) := 2s\dot{\gamma}(\tau)$.
    
    Since $q \in X \setminus \tilde{V}_\varepsilon$ and $\delta \ll \varepsilon$, the endpoint $q^\prime$ satisfies $d_{can}(E, f(q^\prime)) \ge \frac{3}{4}\varepsilon$. We first claim that there exists a small constant $c=c(\varepsilon) \in (0,1)$ such that $\gamma(\tau)$ remains strictly within the regular region for the final time interval $\tau \in [(1-c)\bar{\tau}, \bar{\tau}]$. 
    
By the Cauchy-Schwarz inequality, the spatial distance traveled by $\gamma$ with respect to the evolving metric in this final interval is bounded by
    \begin{equation}\label{eq:spatial_distance_backward}
    \begin{aligned}
        d &= \int_{(1-c)\bar{\tau}}^{\bar{\tau}} |\dot{\gamma}(\tau)|_{\omega(T)} d\tau \leq  (1+\Psi(T^{-1}|\varepsilon))\int_{(1-c)\bar{\tau}}^{\bar{\tau}} |\dot{\gamma}(\tau)|_{\tilde{g}(\tau)} d\tau  \\
        &= (1+\Psi(T^{-1}|\varepsilon))\int_{(1-c)\bar{\tau}}^{\bar{\tau}} \tau^{-\frac{1}{4}} \left( \tau^{\frac{1}{4}} |\dot{\gamma}(\tau)|_{\tilde{g}(\tau)} \right) d\tau \\
        &\le  2\left( \int_{(1-c)\bar{\tau}}^{\bar{\tau}} \tau^{-\frac{1}{2}} d\tau \right)^{\frac{1}{2}} \left( \int_{(1-c)\bar{\tau}}^{\bar{\tau}} \sqrt{\tau} |\dot{\gamma}(\tau)|^2_{\tilde{g}(\tau)} d\tau \right)^{\frac{1}{2}}\\
        &\le2\sqrt{2c}{\bar{\tau}}^{\frac{1}{4}}\frac{C(\varepsilon)}{\bar{\tau}^{\frac{1}{4}}}= \sqrt{2cC_1(\varepsilon)},
    \end{aligned}
    \end{equation}
  
    where we used the elementary inequality $1-\sqrt{1-c} \le c$ for $c \in (0,1)$ and we have chosen $T>T(\varepsilon)$ such that $1+\Psi(T^{-1}|\varepsilon)<2$. We can then choose $c$ sufficiently small (depending only on $\varepsilon$ and $C_1$) such that $d \le \frac{\varepsilon}{4}$. Therefore, for $s \in [s_1, \sqrt{\bar{\tau}}]$ where $s_1 := \sqrt{1-c}\sqrt{\bar{\tau}}$, the curve $\gamma$ is strictly contained in $X \setminus \tilde{V}_{\varepsilon/2}$. 
    
    Within this regular region, the geometry is uniformly bounded by Gill's smooth convergence result \cref{smooth convergence}, implying $|Rc| \le C(\varepsilon)$ and $|\nabla S| \le C(\varepsilon)$. The $\mathcal{L}$-geodesic equation $\nabla_{\hat{X}}\hat{X} = 2s^2\nabla S - 4s Rc(\hat{X}, \cdot)$ thus yields the differential inequality (see \eqref{eq:ds X hat})
    \begin{equation}\label{eq:ODE_velocity}
        \frac{d}{ds} |\hat{X}| \le C(\varepsilon)s |\hat{X}| + C(\varepsilon)s^2,\quad s\in[s_1,\bar{\tau}].
    \end{equation}
    
    Applying the Mean Value Theorem for integrals over the interval $[s_1, \sqrt{\bar{\tau}}]$, we have
    \begin{equation}
        \int_{s_1}^{\sqrt{\bar{\tau}}} |\hat{X}|^2 ds \le 2 \int_0^{\bar{\tau}} \sqrt{\tau}|\dot{\gamma}|^2 d\tau \le \frac{C_2}{\sqrt{\bar{\tau}}}.
    \end{equation}
    Since the length of the interval is $\sqrt{\bar{\tau}} - s_1 = (1-\sqrt{1-c})\sqrt{\bar{\tau}} \ge \frac{c}{2}\sqrt{\bar{\tau}}$, there must exist a time $s_0 \in [s_1, \sqrt{\bar{\tau}}]$ such that
    \begin{equation}
        |\hat{X}(s_0)|^2 \le \frac{1}{\frac{c}{2}\sqrt{\bar{\tau}}} \left( \frac{C_2}{\sqrt{\bar{\tau}}} \right) = \frac{2C_2}{c\bar{\tau}} \quad \implies \quad |\hat{X}(s_0)| \le \frac{C_3(\varepsilon)}{\sqrt{\bar{\tau}}}.
    \end{equation}
    
    Finally, we integrate the differential inequality \eqref{eq:ODE_velocity} from $s_0$ to $\sqrt{\bar{\tau}}$. Let $y(s) := |\hat{X}(s)|$. The inequality takes the form $y^\prime(s) \le As y(s) + Bs^2$ for $s \in [s_0, \sqrt{\bar{\tau}}]$, where $A, B > 0$ are constants depending only on $\varepsilon$. Applying the standard integral form of Gronwall's inequality yields:
    \begin{equation}
        y(\sqrt{\bar{\tau}}) \le y(s_0)\exp\left(\int_{s_0}^{\sqrt{\bar{\tau}}} Ar \, dr\right) + \int_{s_0}^{\sqrt{\bar{\tau}}} Bu^2 \exp\left(\int_u^{\sqrt{\bar{\tau}}} Ar \, dr\right) du.
    \end{equation}
  It is then clear that we can write
    \begin{equation}
    \begin{aligned}
        |\hat{X}(\sqrt{\bar{\tau}})| &\le K(\varepsilon) |\hat{X}(s_0)| + K(\varepsilon) B \int_0^{\sqrt{\bar{\tau}}} u^2 du \\
        &\le K(\varepsilon) \frac{C_2(\varepsilon)}{\sqrt{\bar{\tau}}} + K(\varepsilon) \frac{B}{3} \bar{\tau}^{\frac{3}{2}}.
    \end{aligned}
    \end{equation}
We can therefore conclude that the endpoint velocity satisfies the desired bound:
    \begin{equation}
        |\nabla L_T(q^\prime, \bar{\tau})| = |\hat{X}(\sqrt{\bar{\tau}})| \le \frac{C_4(\varepsilon)}{\sqrt{\bar{\tau}}}.
    \end{equation}
    This completes the proof.
\end{proof}

\begin{proof}[Proof of \cref{prop:key estimate}]
    Applying \cref{prop:almost aviodable:final}, for each $q^\prime\in\Omega^\prime$, we can find an $\mathcal{L}$-geodesic $\gamma$ connecting $(p,0)$ and $(q^\prime,\bar{\tau})$. We then subdivide the interval $[0,\bar{\tau}]$ into two parts: we say that $\tau\in I$ if $\gamma(\tau)\in\tilde{U}_{\frac{\eta}{2}}=f^{-1}(U_{\frac{\eta}{2}})$ and $I^\prime$ is defined by means of $U^\prime_{\eta^\prime}$ similarly. By \cref{lem:L-geodesic-exit-time} we have that $I\cup I^\prime\subset[C^{-1}\bar{\tau},\bar{\tau}]$. Set $J:=[0,\bar{\tau}]\setminus I\cup I^\prime$. 

Now, to estimate $d_{can}(f(p),f(q))$, we write
\begin{equation}\label{eq:estimate of d_can 1}
\begin{aligned}
    d_{can}(f(p),f(q))&\leq d_{can}(f(p),f(q^\prime))+\Psi(T^{-1},\delta|\varepsilon)\\
&\leq\int_J|\partial_\tau(f\circ\gamma)|_{g_{can}}d\tau+\int_{I\cup I^\prime}|\partial_\tau(f\circ\gamma)|_{g_{can}}d\tau+\Psi(T^{-1},\delta|\varepsilon)\\
&\leq\left(1+\Psi(T^{-1}|\bar{\tau},\eta,\eta^\prime,\varepsilon)\right)\int_J|\partial_\tau\gamma|_{\tilde{g}(\tau)}d\tau+\int_{I\cup I^\prime}|\partial_\tau(f\circ\gamma)|_{g_{can}}d\tau+\Psi(T^{-1},\delta|\varepsilon)\\
&\leq\left(1+\Psi(T^{-1}|\bar{\tau},\eta,\eta^\prime,\varepsilon)\right)\int_0^{\bar{\tau}}|\partial_\tau\gamma|_{\tilde{g}(\tau)}d\tau+\int_{I\cup I^\prime}|\partial_\tau(f\circ\gamma)|_{g_{can}}d\tau+\Psi(T^{-1},\delta|\varepsilon).
\end{aligned}
\end{equation}
Here, in the third inequality we have used the fact that $\gamma|_{J}$ lies outside $\tilde{U}_{\frac{\eta}{2}}$, where we have the smooth convergence $g(T)\to f^*g_{can}$ and hence $\tilde{g}(\tau)$ is very close to $g_{can}$. For the right-hand side of \eqref{eq:estimate of d_can 1}, we further using the Cauchy-Schwarz inequality to write
\begin{equation}\label{eq:deal with J}
\begin{aligned}
    \int_0^{\bar{\tau}}|\partial_\tau\gamma|_{\tilde{g}(\tau)}d\tau&\leq\sqrt{2}\bar{\tau}^{\frac{1}{4}}\left(\int_0^{\bar{\tau}}\sqrt{\tau}|\partial_\tau\gamma|_{\tilde{g}(\tau)}d\tau\right)^{\frac{1}{2}}\\
    &=\sqrt{2}\bar{\tau}^{\frac{1}{4}}\left(\int_0^{\bar{\tau}}\sqrt{\tau}\left(S(\tau)+|\partial_\tau\gamma|_{\tilde{g}(\tau)}\right)d\tau-\int_0^{\bar{\tau}}\sqrt{\tau}S(\tau)d\tau\right)^{\frac{1}{2}}\\
    &\leq\sqrt{2}\bar{\tau}^{\frac{1}{4}}\left(L_{T}(q^\prime,\bar{\tau})+C\bar{\tau}^{\frac{3}{2}}\right)^{\frac{1}{2}}\\
    &=\sqrt{2}\bar{\tau}^{\frac{1}{4}}\left(L_{T}(q,\bar{\tau})+\Psi(\delta|\varepsilon)\bar{\tau}^{-\frac{1}{2}}+C\bar{\tau}^{\frac{3}{2}}\right)^{\frac{1}{2}}\\
    &\leq\sqrt{2}\bar{\tau}^{\frac{1}{4}}(L_{T}(q,\bar{\tau}))^{\frac{1}{2}}+\Psi(\delta|\varepsilon)+C\bar{\tau},
    \end{aligned}
\end{equation}
where in the fourth line above we have used crucially \cref{lem:local_gradient_L}.
It remains to estimate $\int_{I\cup I^\prime}|\partial_\tau(f\circ\gamma)|_{g_{can}}d\tau$. By definition of $q^\prime\in\Omega$, the curve $f\circ\gamma$ makes at most $\eta^{-\sigma}$ excursions into $U_{\frac{\eta}{2}}$, which we denote by intervals $[\tau^-_i,\tau^+_i],1\leq i\leq\eta^{-\sigma}$. To relate $g_{can}$ and $\tilde{g}({\tau})$ in these intervals, we utilize the bi-H\"older equivalence \cref{Holder of d_can}. Since $d_{can}\leq Cd_Y^\alpha$, we may find a curve $\gamma_i:[\tau^-_i,\tau^+_i]\to Y$ connecting $f\circ\gamma(\tau_i^-)$ and $f\circ\gamma(\tau_i^+)$ such that
\begin{equation}
\begin{aligned}
|\gamma_i|_{g_{can}}&=\int_{\tau_i^-}^{\tau_i^+}|\partial_\tau\gamma_i|_{g_{can}}d\tau\approx d_{can}(f\circ\gamma(\tau_i^-),f\circ\gamma(\tau_i^+))\leq C\left(\int_{\tau_i^-}^{\tau_i^+}|\partial_\tau (f\circ\gamma)|_{g_Y}d\tau\right)^\alpha\\
&\leq C\left(\int_{\tau_i^-}^{\tau_i^+}|\partial_\tau \gamma|_{\tilde{g}(\tau)}d\tau\right)^\alpha,
\end{aligned}
\end{equation}
where in the last inequality we have used the parabolic Schwarz estimate \cref{parabolic Schwarz}. Therefore, if we replace $f\circ\gamma$ by $\gamma_i$ on each interval $[\tau^-_i,\tau^+_i],1\leq i\leq\eta^{-\sigma}$, we can further use the discrete H\"older inequality to write (by abuse of notations)
\begin{equation}\label{eq:deal with I}
    \int_I |\partial_\tau(f\circ\gamma)|_{g_{can}}d\tau\leq C\sum_{i=1}^{\eta^{-\sigma}}\left(\int_{\tau_i^-}^{\tau_i^+}|\partial_\tau \gamma|_{\tilde{g}(\tau)}d\tau\right)^\alpha \leq C \eta^{-(1-\alpha)\sigma}\left(\int_I|\partial_\tau \gamma|_{\tilde{g}(\tau)}d\tau\right)^\alpha.
\end{equation}
By the Cauchy-Schwarz inequality again, we have
\begin{equation}\label{eq:I integral}
    \begin{aligned}
        \int_I|\partial_\tau \gamma|_{\tilde{g}(\tau)}d\tau&\leq C\bar{\tau}^{-\frac{1}{4}}\left(\int_I\sqrt{\tau}|\partial_\tau\gamma|^2_{\tilde{g}(\tau)}d\tau\right)^{\frac{1}{2}}\cdot|I|^{\frac{1}{2}}\\
        &\leq C\bar{\tau}^{-\frac{1}{4}}\left(L_T(q^\prime,\bar{\tau})+C\bar{\tau}^{\frac{3}{2}}\right)^{\frac{1}{2}}\cdot|I|^{\frac{1}{2}}\\
        &\leq C\bar{\tau}^{-\frac{1}{4}}\left(\bar{\tau}^{-\frac{1}{2}}+\bar{\tau}^{\frac{3}{2}}\right)^{\frac{1}{2}}\cdot|I|^{\frac{1}{2}}.
    \end{aligned}
\end{equation}
For the estimate of $|I|$, we use similar ideas as in \cref{prop:almost-avoidance-D1}. Let $\Omega_p\subset T_pM$ be such that $\mathcal{L}\exp_{\bar{\tau}}:\Omega_p\to\Omega^\prime\subset B^{d_T}(q,\delta)$ is a diffeomorphism. The change of variable formula, the uniform equivalence of volume forms (\cref{cor:uniform equivalence of volume forms}) and the Minkowski content estimate \eqref{eq:minkow dim} yield that
\begin{equation}
    \begin{aligned}
     |I| \operatorname{Vol}(\Omega^\prime, \omega(T)^n)     &\leq   C(\varepsilon,\delta,\bar{\tau})|I| \operatorname{Vol}(\Omega^\prime, \tilde{\omega}(\bar{\tau})^n)     
\leq C |I| \int_{\Omega_p} \mathcal{J}(v, \bar{\tau}) dv \notag \\    
&\leq C\int_{\Omega_p} \int_{I} \mathcal{J}(v, \bar{\tau}) d\tau dv \notag \\    
&\leq C e^{\frac{C}{\bar{\tau}}} \int_{\Omega_p} \int_{I} \mathcal{J}(v, \tau) d\tau dv \notag \\    
&\leq C e^{\frac{C}{\bar{\tau}}} \int_{C^{-1}\bar{\tau}}^{\bar{\tau}} \operatorname{Vol}(\tilde{U}_\eta, \tilde{\omega}(\tau)^n) d\tau\\
&\leq C e^{\frac{C}{\bar{\tau}}} \bar{\tau}\eta^{2-\sigma}.
    \end{aligned}
\end{equation}
The volume non-collapsing estimate \cref{volume non-collapsing estiamte} combined with \cref{prop:almost aviodable:final} provide an estimate of the left-hand side:
\begin{equation}
     |I| \operatorname{Vol}(\Omega^\prime, \omega(T)^n)\geq|I|\cdot c\delta^{2n+\beta}.
\end{equation}
As a consequence, we derive the estimate of $|I|$:
\begin{equation}
    |I|\leq C e^{\frac{C}{\bar{\tau}}} \bar{\tau}\eta^{2-\sigma}\delta^{-2n-\beta}.
\end{equation}
Plugging into \eqref{eq:I integral}, we obtain
\begin{equation}\label{eq:deal with I 3}
      \int_I|\partial_\tau \gamma|_{\tilde{g}(\tau)}d\tau\leq  Ce^{\frac{C}{2\bar{\tau}}}\bar{\tau}^{\frac{1}{4}}\left(\bar{\tau}^{-\frac{1}{2}}+\bar{\tau}^{\frac{3}{2}}\right)^{\frac{1}{2}}\eta^{1-\frac{\sigma}{2}}\delta^{-n-\frac{\beta}{2}}.
\end{equation}
Plugging \eqref{eq:deal with I 3} into \eqref{eq:deal with I}, we finally get
\begin{equation}\label{eq:deal with I final}
     \int_I |\partial_\tau(f\circ\gamma)|_{g_{can}}d\tau\leq\Psi(\eta|\delta,\bar{\tau},\varepsilon) \eta^{-(1-\alpha)\sigma+\alpha(1-\frac{\sigma}{2})}=\Psi(\eta|\delta,\bar{\tau},\varepsilon) \eta^{\alpha(\frac{3}{4}-\frac{\sigma}{2})},
\end{equation}
if we choose $\sigma:=\frac{\alpha}{4(1-\alpha)}$.
Similarly for $I^\prime$,
\begin{equation}\label{eq:deal with I prime}
    \int_{I^\prime} |\partial_\tau(f\circ\gamma)|_{g_{can}}d\tau \leq \Psi(\eta^\prime|\delta,\bar{\tau},\eta,\varepsilon) (\eta^\prime)^{\alpha(\frac{3}{4}-\frac{\sigma}{2})}.
\end{equation}

Substituting \eqref{eq:deal with J}, \eqref{eq:deal with I final}, and \eqref{eq:deal with I prime} back into \eqref{eq:estimate of d_can 1}, we get
\begin{equation}\label{eq:estimate of d_can 2}
\begin{aligned}
    d_{can}(f(p),f(q)) \leq & \left(1+\Psi(T^{-1}|\bar{\tau},\eta,\eta^\prime,\varepsilon)\right) \sqrt{2}\bar{\tau}^{\frac{1}{4}} \left( L_{T}(q,\bar{\tau}) +\Psi(\delta|\bar{\tau},\varepsilon)+ C\bar{\tau}^{\frac{3}{2}} \right)^{\frac{1}{2}} \\
    & + \Psi(\eta|\delta,\bar{\tau},\varepsilon) \eta^{\alpha(\frac{3}{4}-\frac{\sigma}{2})} + \Psi(\eta^\prime|\delta,\eta,\bar{\tau},\varepsilon) (\eta^\prime)^{\alpha(\frac{3}{4}-\frac{\sigma}{2})} + \Psi(T^{-1},\delta|\varepsilon).
\end{aligned}
\end{equation}

Recall that from the uniform upper bound of the $\mathcal{L}$-distance \cref{lem:upper bound of L}, we have
\begin{equation}
    L_T(q,\bar{\tau}) \leq \frac{1}{2\sqrt{\bar{\tau}}} d_T(p,q)^2 + \Psi(T^{-1}|\bar{\tau},\varepsilon) + \Psi(\bar{\tau}|\varepsilon).
\end{equation}
Applying the fundamental subadditivity $\sqrt{A+B} \leq \sqrt{A} + \sqrt{B}$ for $A,B \geq 0$, the main term involving the $\mathcal{L}$-distance can be cleanly evaluated as:
\begin{equation}
\begin{aligned}
    \sqrt{2}\bar{\tau}^{\frac{1}{4}} \left( L_{T}(q,\bar{\tau}) \right)^{\frac{1}{2}} &\leq \sqrt{2}\bar{\tau}^{\frac{1}{4}} \left( \frac{d_T(p,q)^2}{2\sqrt{\bar{\tau}}} \right)^{\frac{1}{2}} + \Psi(T^{-1}, \bar{\tau}|\varepsilon) \\
    &= d_T(p,q) + \Psi(T^{-1}, \bar{\tau}|\varepsilon).
\end{aligned}
\end{equation}

Plugging this simplification back into \eqref{eq:estimate of d_can 2}, we arrive at
\begin{equation}
\begin{aligned}
    d_{can}(f(p),f(q)) \leq d_T(p,q)+ \Psi(\eta|\delta,\bar{\tau},\varepsilon)  + \Psi(\eta^\prime|\delta,\eta,\bar{\tau},\varepsilon)  + \Psi(\delta|\varepsilon) + \Psi(\bar{\tau}|\varepsilon) + \Psi(T^{-1}|\bar{\tau}, \eta, \eta^\prime,\varepsilon).
\end{aligned}
\end{equation}

Finally, by taking the parameters successively in the order $T^{-1} \ll \eta^\prime \ll \eta \ll \bar{\tau} \ll \delta \ll \varepsilon\ll 1$ (i.e., for fixed $\varepsilon$, first choosing $\delta$ small, then $\bar{\tau}$ small enough, then $\eta, \eta^\prime$, and finally $T$ sufficiently large) successfully yields the desired key estimate:
\begin{equation}
    d_{can}(f(p),f(q)) \leq d_T(p,q) + \Psi(T^{-1}|\varepsilon),
\end{equation}
which concludes the proof.

\end{proof}

\textbf{Data availability} Data sharing not applicable to this article as no datasets were generated or analyzed during the current study.

\textbf{Conflict of interest} The authors declare that they have no conflict of interest.

\end{document}